\documentclass[10pt,oneside,a4paper,reqno]{amsart}

\usepackage[T1]{fontenc}
\usepackage[utf8]{inputenc}
\usepackage[english]{babel}
\usepackage{lmodern}
\usepackage{microtype}

\usepackage{amssymb} 
\usepackage{mathrsfs}
\usepackage{bm,esint}
\usepackage{graphicx,tabularx,booktabs}
\usepackage[shortlabels]{enumitem}
\usepackage[normalem]{ulem}
\usepackage[
  a4paper,
  margin=2.8cm,
]{geometry}

\numberwithin{equation}{section}
\usepackage[dvipsnames]{xcolor}
\definecolor{citecolor}{HTML}{0066cc}
\definecolor{eqcolor}{HTML}{0066cc}

\usepackage{tikz}
\usetikzlibrary{shapes.geometric, arrows.meta, positioning, shadows}

\usepackage[colorlinks=true]{hyperref}
\hypersetup{
  citecolor=citecolor,
  linkcolor=eqcolor,
  urlcolor=eqcolor,
  pdfborder={0 0 0}
}

\theoremstyle{plain}
\newtheorem{theorem}{Theorem}[section]
\newtheorem{lemma}[theorem]{Lemma}
\newtheorem{proposition}[theorem]{Proposition}

\theoremstyle{definition}
\newtheorem{definition}[theorem]{Definition}
\newtheorem{axiom}[theorem]{Axiom}

\theoremstyle{remark}
\newtheorem{remark}[theorem]{Remark}

\newcommand{\restr}{\bm{|}}
\newcommand{\st}{\, : \,}
\newcommand{\jqq}{q}
\newcommand{\pHold}{p}
\newcommand{\pSob}{p_S}

\newcommand{\jeta}{\eta}
\newcommand{\jalfa}{\alpha}
\newcommand{\jax}{{Axiom}}
\newcommand{\jaxs}{{Axioms}}
\newcommand{\jaOm}{\bm{\Omega}} 
\newcommand{\jc}{\bm{c}}

\DeclareMathOperator{\divg}{\mathbf{div}}
\DeclareMathOperator{\curl}{\mathbf{curl}}
\DeclareMathOperator{\divr}{div}

\newcommand{\jth}{\theta}
\newcommand{\Osc}{\Psi_p}
\newcommand{\Energy}{\mathcal{E}}
\newcommand{\class}{\bm{\mathcal{S}}}

\newcommand{\jgam}{\gamma}

\newcommand{\rr}{r}
\newcommand{\jav}[1]{\fint_{#1}}
\newcommand{\mydef}[1]{#1}
\newcommand{\myprop}[1]{#1}

\newcommand{\cpt}{\text{{\upshape{)}}}}
\newcommand{\opt}{\text{{\upshape{(}}}}
\newcommand{\lapl}{\Delta}
\newcommand{\grad}{\nabla}
\newcommand{\Om}{\Omega}

\newcommand{\RR}{\mathbb{R}}
\newcommand{\NN}{\mathbb{N}}

\newcommand{\Stwo}{\mathbb{S}}

\newcommand{\eqs}{=}

\author[G. D. Fratta]{Giovanni Di Fratta}
\date{}

\AtEndDocument{\bigskip
	\hrule
	\medskip
         \small{
		\noindent
		{\sc G. Di Fratta},
		Dipartimento di Matematica e Applicazioni ``R. Caccioppoli'',
		Universit\`a degli Studi di Napoli ``Federico II'',
		Via Cintia, Complesso Monte S. Angelo, 80126 Naples, Italy.
		
		\noindent
		E-mail:
		\href{mailto:giovanni.difratta@unina.it}{\tt giovanni.difratta@unina.it}.
	}
\medskip
\hrule
}

\makeatletter
\def\@tocline#1#2#3#4#5#6#7{\relax
  \ifnum #1>\c@tocdepth 
  \else
    \par \addpenalty\@secpenalty\addvspace{#2}
    \begingroup \hyphenpenalty\@M
    \@ifempty{#4}{
      \@tempdima\csname r@tocindent\number#1\endcsname\relax
    }{
      \@tempdima#4\relax
    }
    \parindent\z@ \leftskip#3\relax \advance\leftskip\@tempdima\relax
    \rightskip\@pnumwidth plus4em \parfillskip-\@pnumwidth
    #5\leavevmode\hskip-\@tempdima
      \ifcase #1
       \or\or \hskip 1pc \or \hskip 2pc \else \hskip 3pc \fi
      #6\nobreak\relax
    \dotfill\hbox to\@pnumwidth{\@tocpagenum{#7}}\par
    \nobreak
    \endgroup
  \fi}

\begin{document}

\title[A Classical Elliptic Regularity Approach to Almost Harmonic Maps]{A Classical Elliptic Regularity Approach to\\
Almost Harmonic Maps and Related Systems}

\begin{abstract}
We develop an abstract regularity framework for a class of two-dimensional nonlinear elliptic systems, including almost harmonic maps. The approach combines a Campanato-type iteration scheme with a Caccioppoli-type estimate and identifies general assumptions under which local Hölder continuity follows. More precisely, we prove that any class of admissible pairs $(\bm{u},\bm{f})$ that is stable under rescaling and satisfies an oscillation-decay property consists of locally Hölder continuous maps. The resulting Hölder exponent is explicit and attains the classical Morrey--Campanato threshold determined by the Lebesgue integrability of the source term $\bm{f}$. The framework is purely analytic and avoids the use of $\mathcal{H}^1$--$\mathrm{BMO}$ duality, Wente's inequality, moving frames, and conformal uniformization.

We illustrate the flexibility of the framework through several classes of elliptic systems. As a first example, we recover local Hölder continuity for almost harmonic maps
\[
-\Delta \bm{u}=|\nabla\bm{u}|^2\bm{u}+\bm{f}
\]
into $\Stwo^n$ with $L^q$-integrable tension fields by means of a direct argument independent of the classical harmonic map regularity theory. We next consider systems of the form
\[
-\Delta \bm{u}=\jaOm\cdot\nabla\bm{u}+\bm{f},
\]
showing that the analytic condition $\divg\jaOm\in L^q$ for some $q>1$ is sufficient to ensure regularity
(for classical harmonic maps, $\divg\jaOm=\bm{0}$). 

We further apply the framework to anisotropic systems of the form
\[
-\divg(A\nabla\bm{u})=\jaOm\cdot\nabla\bm{u}+\bm{f},
\]
where $A$ is a Hölder-continuous, non-unimodular coefficient matrix, a setting in which conformal uniformization
is unavailable. Finally, we develop a dimension-independent bootstrap procedure that upgrades continuity to full
$C^\infty$ regularity. As a physical application, we establish the interior regularity of magnetic skyrmions by
proving that finite-energy critical points of the full Brown micromagnetic energy in dimensions $m\in\{2,3\}$,
including the Dzyaloshinskii--Moriya interaction, magnetocrystalline anisotropy, and the nonlocal
demagnetizing field, possess the optimal regularity $C^{k+1,\beta}$ permitted by the regularity of
the underlying magnetocrystalline anisotropy density.
\end{abstract}


\maketitle


\section{Introduction and Statement of Results}

\subsection{Motivation and Problem Setting}
Harmonic maps and their generalizations occupy a
central position in geometric analysis, the calculus of variations,
and a wide array of models arising in mathematical physics.
Emerging as critical points of the Dirichlet energy, they provide a unifying
framework that connects nonlinear elliptic systems, minimal surface theory,
and geometric flows. The rich analytical structure of
these systems has catalyzed the development of foundational techniques across the discipline,
driving major advances in regularity theory, singularity analysis,
and the mathematics of pattern formation.

In this paper, we investigate the interior regularity for a broad class of strongly coupled, nonlinear elliptic systems exhibiting critical quadratic growth in the gradient. The prototypical example driving this investigation is the regularity theory of \emph{almost harmonic maps} from an open subset of $\RR^m$ into the sphere $\Stwo^n$. More precisely, these correspond to critical points of perturbed Dirichlet energies whose Euler--Lagrange equation takes the form
\begin{equation}
  \label{eq:almost_harmonic_intro} - \Delta \bm{u}= | \nabla \bm{u}|^2 \bm{u}+\bm{f},
\end{equation}
where the lower-order term $\bm{f}$ represents an inhomogeneous perturbation of the classical harmonic map equation. Such systems arise naturally in physically relevant models, including the Oseen--Frank theory of liquid crystals and variational models in micromagnetics, where external fields, anisotropies, or defects in the crystal lattice introduce nontrivial forcing terms.

The main analytical difficulty in \eqref{eq:almost_harmonic_intro}, and in the
larger class of systems considered here, comes from the critical growth of the
nonlinearity: under the natural dilational rescaling, the quadratic gradient
term scales in the same way as the Laplacian, placing the equation at the
borderline of classical Calderón--Zygmund theory. In dimension two, weak
solutions in the energy space $H^1$ are known to be continuous. Existing
proofs of this fact rely on several analytical tools developed over the years,
including compensated compactness, the Hardy space $\mathcal{H}^1$ and its
duality with $\mathrm{BMO}$, Wente's inequality, and gauge-theoretic
decompositions.

These methods exploit specific structural properties of the equation, such as
antisymmetry or conformal invariance. The aim of this work is to identify a
regularity mechanism that can be formulated in terms of classical linear
elliptic estimates. By expressing this mechanism through a Campanato-type
integrability principle, we obtain a framework that can also be applied to
perturbations of the classical setting, including systems without
antisymmetry and equations with anisotropic coefficients.

This perspective naturally motivates the following objectives of the present work:
\smallskip

\begin{enumerate}
  \item[\text{{\bfseries{($\ensuremath{\boldsymbol{Q}}_1$)}}}]
  To show that the interior regularity of two-dimensional almost harmonic maps
  with $L^q$ tension fields can be recovered through classical linear elliptic
  estimates, without relying on $\mathcal{H}^1$--$\mathrm{BMO}$ duality, and to
  obtain quantitative estimates in terms of $q$.
  \smallskip
   
  \item[\text{{\bfseries{($\ensuremath{\boldsymbol{Q}}_2$)}}}]
  To show that the geometric antisymmetry assumption on the connection form
  $\jaOm=\bm{u}\wedge\nabla\bm{u}$ can be replaced by the analytic condition
  $\divg\jaOm\in L^q$.
  \smallskip
  
  \item[\text{{\bfseries{($\ensuremath{\boldsymbol{Q}}_3$)}}}]
  To extend the regularity framework to systems with variable-coefficient
  principal parts $\divg(A\nabla\bm{u})$ in the non-unimodular regime
  ($\det A\neq\operatorname{const}$), where classical conformal uniformization
  is unavailable.
  \smallskip
  
  \item[\text{{\bfseries{($\ensuremath{\boldsymbol{Q}}_4$)}}}]
  To provide a bootstrap procedure for higher Schauder regularity based on
  classical subcritical estimates once initial continuity has been established.
\end{enumerate}

\subsection{Background and State of the Art}
We begin by recalling the classical results on harmonic maps and then highlight
the distinct analytical challenges introduced by almost harmonic maps and their
broader structural generalizations, thereby clarifying the precise scope and
context of the present work.

For harmonic maps from two-dimensional domains $(m=2)$, the theory is well
established, ultimately grounded in the conformal invariance of the Dirichlet energy at the critical dimension.
  The regularity of energy-minimizing harmonic maps from surfaces rests on Morrey's foundational analysis of multiple integrals in the calculus of variations~{\cite{Morrey1966}}. This theory was subsequently advanced by Schoen~{\cite{Schoen1984}} who proved the interior smoothness of \emph{stationary} harmonic
maps, that is, of weak solutions subject to the additional variational assumption of
stability under domain deformations.

The definitive regularity result for generic weak solutions is due to Hélein. He first proved that every weakly harmonic map (i.e., every
$W^{1, 2}$ solution of the Euler--Lagrange equation) from a surface into the
sphere $\Stwo^n$ is smooth~{\cite{Helein1990}}, and subsequently generalized
this optimal regularity result to maps taking values in closed Riemannian
manifolds~{\cite{Helein1991,Helein1991a,H_lein_2002}}. For the specific case
of $\Stwo^n$, Hélein's argument relies on rewriting the critical
nonlinearity using a divergence-free antisymmetric structural potential and
applying compensated-compactness techniques. Obtaining continuity
in this framework relies on Wente's inequality and the duality
between the Hardy space $\mathcal{H}^1$ and BMO.

In dimensions $m \geqslant 3$ the picture changes markedly, precisely because
the two-dimensional conformal invariance is lost. As documented in Hardt's
survey~{\cite{Hardt1997}}, regularity may fail even for minimizers, with
singularities forced by topology or by energy considerations. The canonical
instance is the radial projection $\ensuremath{\boldsymbol{u}} (x) = x / |x|$
from the unit ball of $\RR^3$ into $\Stwo^2$: it minimizes the energy among
maps sharing its boundary data, yet its singular set is the single point at
the origin, of Hausdorff dimension zero. The earliest positive results imposed
geometric restrictions precluding such defects; thus Hildebrandt, Kaul, and
Widman~{\cite{Hildebrandt_1977}} obtained everywhere regularity for minimizers
whose image lies within a sufficiently small geodesic ball. An important
advance, removing all restrictions on the image, was achieved by Schoen and
Uhlenbeck~{\cite{Schoen1982}} (see also the historical account
in~{\cite{Sormani_2018}}): through the monotonicity formula and a systematic
analysis of tangent maps obtained by blow-up, they proved that the singular
set of any energy minimizer is closed and has Hausdorff dimension at most $m -
3$. In particular, minimizers from a surface are everywhere continuous, and in
three dimensions the singularities reduce to isolated points.

For the broader class of \emph{stationary} weak solutions
Evans~{\cite{Evans_1991}} established partial regularity for maps into
spheres, demonstrating that the singular set has $(m - 2)$-dimensional
Hausdorff measure zero. Evans exploited the specific coordinate representation
of the spherical harmonic map equation to reveal the antisymmetric structure
of the quadratic nonlinearity, deploying $\mathcal{H}^1$--BMO duality via the
theorem of Coifman, Lions, Meyer, and Semmes~{\cite{Coifman1993}}.
Bethuel~{\cite{Bethuel_1993}} subsequently extended this bound to stationary
maps into arbitrary compact target manifolds by utilizing moving frames to
uncover a hidden antisymmetric structure.

When the classical harmonic map equation is perturbed by a lower-order
inhomogeneous source term, the corresponding analytic theory generally
follows two distinct mathematical directions.

The first concerns the asymptotic analysis of defect formation. In this
framework, one studies sequences of almost (or approximate) harmonic maps
whose tension fields vanish or remain uniformly bounded in a suitable Lebesgue
space. In this direction, Topping {\cite{Topping2004}} utilized sequences with
parametrically small $L^2$ tension fields to establish sharp quantization and
repulsion phenomena for emerging singularities within a bubbling regime, a
perspective later extended to $L^p$-bounded tension fields ($p > 1$) by Wang,
Wei, and Zhang {\cite{Wang2017}}.

The second paradigm---and the one relevant to the present work---concerns the
interior regularity of an individual map subject to a fixed, prescribed
tension field. In this setting, the governing system presents non-trivial analytical challenges that require a dedicated geometric and analytic
treatment, as discussed for instance in Moser's monograph {\cite{Moser_2005}}.

For generic weak solutions subject to a fixed (not necessarily small) source
term, positive regularity results in dimension two were initially achieved
only in highly specialized contexts, such as Carbou's analysis of nonlocal
micromagnetic energies {\cite{Carbou_1997}}. Beyond specific physical models,
systematic treatments (e.g., Moser {\cite{Moser2005,Moser_2005}}) frequently
focus on higher-dimensional settings ($m \geqslant 3$), where stationarity is
 necessary to derive monotonicity formulas underlying partial regularity.

For conformally invariant systems without stationarity, a significant
generalization of Hélein's method was achieved by
Rivière~{\cite{Riviere2007}}, who systematically studied generic equations of
the form $- \Delta \bm{u}= \jaOm \cdot \nabla
\bm{u}$, where $\jaOm$ is an arbitrary antisymmetric
matrix in $L^2 \left( B_1, \mathfrak{s}\mathfrak{o}(n + 1) \otimes \RR^m
\right)$, $B_1 \subseteq \RR^m$. A key observation was that the Euler--Lagrange equation of \emph{any} two-dimensional conformally invariant variational
problem can be brought into this antisymmetric form. Combining this with
Uhlenbeck's gauge decomposition, he proved that all such systems admit
continuous weak solutions; the framework was carried up to the boundary, and
to systems carrying an inhomogeneous term, by Müller and
Schikorra~{\cite{Mueller2009}}. We stress that, even with the optimal gauge in
hand, the passage to continuity continues to rely on Wente's inequality and
$\mathcal{H}^1$--BMO duality.

The present work develops a different route to the regularity of almost
harmonic maps, combining ideas of Chang~{\cite{Changbi_1999,Chang_1999}} and
Carbou~{\cite{Carbou_1997}} with the structural insight of
Hélein~{\cite{Helein1991a}}. Working entirely within the classical linear
theory of divergence-form elliptic equations, we proceed without relying on compensated-compactness machinery or $\mathcal{H}^1$--BMO duality. In
their place we establish a Campanato-type iterative decay estimate; a new
Caccioppoli-type inequality then makes explicit how the fractional
integrability of the source $\bm{f}$ governs the local
Hölder exponent of the solution. The same scheme yields a transparent,
self-contained bootstrap from initial continuity to full $C^{\infty}$
regularity. Finally, the estimates apply directly to Brown's static
micromagnetic equations and establish the full regularity of topological spin
textures~{\cite{Di_Fratta_2020a,Di_Fratta_2023}}.

In sum, the stationarity-free method developed here shows that the regularity
of two-dimensional almost harmonic maps can be recovered in its entirety from
classical linear elliptic theory, with quantitative control on the
perturbation that is well suited to physical applications.

\subsection{Contributions of the present work}To clarify the scope
of our approach, we first highlight a model corollary that emerges from our
framework:

{\myprop{\begin{theorem}
  \label{thm:regtheoremharmonic}Any weakly almost harmonic map
  $\bm{u} \in H^1 \left( B_1, \Stwo^n \right)$ from the
  unit two-dimensional disk $B_1 \subseteq \RR^2$ to a sphere $\Stwo^n \subset
  \RR^{n + 1}$, solving the equation
  \begin{equation}
    \label{eq:AHMeq0} - \Delta \bm{u} = | \nabla
    \bm{u} |^2 \bm{u} +
    \bm{f},
  \end{equation}
  with a source term $\bm{f} \in L^{\jqq} (B_1, \RR^{n +
  1})$ for some $\jqq > 1$, is locally Hölder continuous.  Specifically, there exists an exponent $\jeta \in (0, 1)$ such that
  $\bm{u} \in C^{0,
  \jeta}_{\ensuremath{\operatorname{loc}}}  (B_1, \RR^{n + 1})$.
\end{theorem}}}

The proof of Theorem~\ref{thm:regtheoremharmonic} is based on absorbing the
critical geometric nonlinearity into an antisymmetric connection matrix. While
local continuity for $L^{\jqq}$ sources was previously established by Müller
and Schikorra~\cite{Mueller2009} through gauge transformations, our approach
provides a different proof based on a Campanato iteration scheme.
In particular, it avoids both the construction of a gauge and the
$\mathcal{H}^1$--$\mathrm{BMO}$ duality, yielding a direct quantitative
estimate in which the Hölder regularity is determined by the fractional
integrability of the source term.

We further show that the same approach applies to a broader class of systems
where the antisymmetry assumption on the connection matrix is replaced by the
analytic condition $\divg\jaOm\in L^\jqq$ for some $\jqq>1$. A related idea
appears in Bethuel~\cite{Bethuel1992}, who, in the setting of the prescribed
mean curvature equation, imposed integrability assumptions ensuring that the
divergence of the associated connection matrix belongs to a Lorentz space.
Our framework shows that, for globally bounded $\bm{u}$, the weaker condition
$\divg\jaOm\in L^\jqq$ is sufficient for regularity.

Regularity for such systems can also be obtained by combining Hodge
decomposition with Hardy space methods, splitting the right-hand side into a
Hardy component and an $L^\jqq$ remainder. The purpose here is instead to
provide a self-contained alternative based on Campanato iteration.
The advantage of this formulation becomes apparent in
settings where classical tools, such as uniformization and standard
Hodge-theoretic arguments, are not available, including the non-unimodular
variable-coefficient case.

{\myprop{\begin{theorem}[Regularity of generalized systems]
  \label{thm:generalized_regularity}Let $B_1 \subset \RR^2$. Let
  $\bm{u} \in H^1 \cap L^{\infty} (B_1, \RR^{n + 1})$ be
  a weak solution to the generalized system
  \begin{equation}
    \label{eq:generalized_riviere} - \Delta \bm{u}= \jaOm
    \cdot \nabla \bm{u}+\bm{f} \quad
    \text{in } B_1 .
  \end{equation}
  Assume that the structural matrix and the source terms satisfy the purely
  analytic hypotheses for some $\jqq > 1$:
  \begin{itemize}
    \item $\jaOm \in L^2 (B_1, \RR^{(n + 1) \times (n + 1)} \otimes \RR^2)$,
    
    \item $\bm{f} \in L^{\jqq} (B_1, \RR^{n + 1})$,
    
    \item $\divg \jaOm \in L^{\jqq} (B_1, \RR^{(n + 1) \times (n + 1)})$.
  \end{itemize}
  Then, $\bm{u}$ is locally Hölder continuous in $B_1$.
  Specifically, there exists an exponent $\jeta \in (0, 1)$ such that
  $\bm{u} \in C^{0,
  \jeta}_{\ensuremath{\operatorname{loc}}}  (B_1, \RR^{n + 1})$.
\end{theorem}}}

To place the condition $\divg \jaOm \in L^q$ in context, it is useful to compare our setting with the classical gauge-theoretic framework. In Hélein's original proof for $\mathbb{S}^n$-valued harmonic maps~\cite{Helein1990}, the antisymmetry of the connection is used to recast the system in a gauge in which the resulting connection matrix is divergence-free, namely $\divg \jaOm = \bm{0}$. Only after this divergence-free structure has been obtained can one invoke the analytical $\mathcal{H}^1$--$\mathrm{BMO}$ regularity theory. A similar strategy appears in the general gauge-theoretic constructions of Rivière~\cite{Riviere2007} and Müller--Schikorra~\cite{Mueller2009}, where antisymmetry acts as the algebraic mechanism that produces a zero-divergence gauge, after which purely analytic compensation estimates apply.

By contrast, our framework shows that this geometric step is not necessary for regularity. The enforcement of $\divg \jaOm = \bm{0}$ through a gauge transformation can be entirely avoided: we prove that an a priori $L^q$ bound on $\divg \jaOm$ already suffices to initiate the Campanato iteration and derive regularity, independently of any underlying geometric gauge structure. This underscores that, in the classical theory, antisymmetry functions primarily as a means of producing a divergence-free gauge, rather than as an intrinsic analytic requirement of the regularity mechanism itself.

Our third result extends the theory to non-unimodular anisotropic media. Through a localized coefficient-freezing argument, we generalize the isotropic principal part $-\Delta \bm{u}$ to a fully variable-coefficient divergence-form operator $-\divg(A \nabla \bm{u})$. This extension successfully establishes regularity in a regime that cannot be treated by uniformization.

{\myprop{\begin{theorem}[Regularity of anisotropic systems]
  \label{thm:aniso_regularity}Let $B_1 \subset \RR^2$. Let
  $\bm{u} \in H^1 \cap L^{\infty} (B_1, \RR^{n + 1})$ be
  a weak solution of
  \begin{equation}
    - \divg (A \nabla \bm{u}) = \jaOm \cdot \nabla
    \bm{u}+\bm{f} \qquad \text{in }
    B_1 .
  \end{equation}
  Assume the structural data satisfy, for some $\jqq > 1$ and some $\gamma_0
  \in (0, 1]$:
  \begin{itemize}
    \item[\text{{\upshape{(A1)}}}] $A : B_1 \to \RR^{2 \times 2}$ is
    symmetric, uniformly elliptic, and Hölder continuous:
    \begin{equation} \lambda | \xi |^2 \leqslant A (x) \xi \cdot \xi \leqslant \Lambda | \xi
       |^2 \qquad \forall x \in B_1, \xi \in \RR^2, \quad 0 < \lambda
       \leqslant \Lambda < + \infty, \end{equation}
    and $A \in C^{0, \gamma_0} (\overline{B_1})$;
    
    \item[\text{{\upshape{(A2)}}}] $\jaOm \in L^2 (B_1, \RR^{(n + 1) \times (n
    + 1)} \otimes \RR^2)$;
    
    \item[\text{{\upshape{(A3)}}}] $\bm{f} \in L^{\jqq}
    (B_1, \RR^{n + 1})$;
    
    \item[\text{{\upshape{(A4)}}}] $\divg \jaOm \in L^{\jqq} (B_1, \RR^{(n +
    1) \times (n + 1)})$.
  \end{itemize}
  Then $\bm{u}$ is locally Hölder continuous in $B_1$.
  Specifically, there exists an exponent $\eta \in (0, 1)$ such that
  $\bm{u} \in C^{0,
  \eta}_{\ensuremath{\operatorname{loc}}} (B_1, \RR^{n + 1})$.
\end{theorem}}}

When $\det A$ is constant, a conformal change of variables reduces the system
to the isotropic setting. Theorem~\ref{thm:aniso_regularity} answers
($\ensuremath{\boldsymbol{Q}}_{\ensuremath{\boldsymbol{3}}}$) by resolving the
{{\em non-unimodular}} regime ($\det A \neq
\ensuremath{\operatorname{const}}$), which is  inaccessible to
uniformization. 

While initial Hölder continuity addresses the critical geometric barrier,
physical applications demand full classical smoothness. To complete the
regularity paradigm, our final main theorem provides a general bootstrap
argument. Valid in arbitrary dimensions, it demonstrates that initial
continuity implies higher-order smoothness, dictated solely by
the integrability of the source.

{\myprop{\begin{theorem}[Bootstrap Regularity for Almost Harmonic Maps]
  \label{lemma:Bootstrap}Let $U \subset \RR^m$ be an open domain, and let
  $\bm{u} \in W^{1, 2}_{\mathrm{loc}} (U, \Stwo^n) \cap
  C^0 (U, \Stwo^n)$ be a weakly almost harmonic map solving the
  Euler--Lagrange equation:
  \begin{equation}
    - \Delta \bm{u}= | \nabla
    \bm{u}|^2
    \bm{u}+\bm{f} \quad \text{in }
    \mathcal{D}' (U),
  \end{equation}
  where the source term initially satisfies $\bm{f} \in
  L^{\jqq}_{\mathrm{loc}} (U, \RR^{n + 1})$.
  
  {\noindent}If the source integrability $\jqq$ satisfies $2 \leqslant \jqq
  \leqslant m / 2$ {\opt}which can only occur in dimensions $m \geqslant
  4${\cpt}, we  additionally assume the map is already locally Hölder
  continuous: $\bm{u} \in C^{0, \gamma}_{\mathrm{loc}}
  (U, \Stwo^n)$ for some $\gamma \in (0, 1]$.
  
  {\noindent}Then, the regularity of $\bm{u}$ 
  depends on the regularity of $\bm{f}$:
  \begin{enumerate}
    \item If $\bm{f} \in L^{\jqq}_{\mathrm{loc}} (U,
    \RR^{n + 1})$ and $2 \leqslant \jqq \leqslant m$ then the regularity is
    bounded by the critical Sobolev embedding threshold:
    \begin{equation}
      \bm{u} \in W^{1, \jqq^{\ast}}_{\mathrm{loc}} (U,
      \Stwo^n)  \quad \text{with } \jqq^{\ast} = \frac{m \jqq}{m - \jqq},
    \end{equation}
    with the usual understanding that $\jqq^{\ast}$ (necessarily greater than
    $1$) can be any finite real exponent if $m = \jqq$. In particular, if $m =
    \jqq$ then $\bm{u} \in C^{0, \alpha}_{\mathrm{loc}}
    (U, \Stwo^n)$ for all $0 < \alpha < 1$, while if $\jqq = 2$ and $m = 3$
    then $\bm{u} \in C^{0, 1 / 2}_{\mathrm{loc}} (U,
    \Stwo^n)$.
    
    \item If $\bm{f} \in L^{\jqq}_{\mathrm{loc}} (U,
    \RR^{n + 1})$ with $\jqq > m$, then
    \begin{equation}
      \bm{u} \in C^{1, \eta}_{\mathrm{loc}} (U, \Stwo^n),
      \quad \text{where } \eta := 1 - m / \jqq .
    \end{equation}
    Note that $\eta$ does not depend on any intermediate fractional exponent.
    In particular, if $\bm{f} \in
    L^{\infty}_{\mathrm{loc}} (U, \RR^{n + 1})$ then
    $\bm{u} \in C^{1, \alpha}_{\mathrm{loc}} (U,
    \Stwo^n)$ for all $0 < \alpha < 1$.
    
    \item For any integer $k \geqslant 1$ and exponent $\beta \in (0, 1)$, if
    $\bm{f} \in C^{k - 1, \beta}_{\mathrm{loc}} (U,
    \RR^{n + 1})$ then $\bm{u} \in C^{k + 1,
    \beta}_{\mathrm{loc}} (U, \Stwo^n)$. In particular, if
    $\bm{f} \in C^{\infty} (U, \RR^{n + 1})$, then
    $\bm{u} \in C^{\infty} (U, \Stwo^n)$. Again, note
    that $\eta$ does not play any role in the higher regularity class we end
    up in.
  \end{enumerate}
\end{theorem}}}

Theorem~\ref{lemma:Bootstrap} answers
($\ensuremath{\boldsymbol{Q}}_{\ensuremath{\boldsymbol{4}}}$). The bootstrap procedure
 1.$\to$2.$\to$3. is dimension-free, requires no stationarity assumption,
and relies  on Calderón--Zygmund and Schauder theory. In $m = 2$, the
continuity hypothesis is supplied by Theorem~\ref{thm:regtheoremharmonic},
making the bootstrap unconditional.

As a direct application of the preceding structural theorems, we turn to micromagnetics and determine the maximal interior regularity of topological spin textures such as magnetic skyrmions. These configurations arise as weak critical points of a non-convex variational energy $\mathcal{G}$ that governs the magnetization field
$\bm{u}: U \to \Stwo^2$:
\begin{equation}
  \label{eq:micro_energy_intro} \mathcal{G} (\bm{u}) :=
  \int_U \left[ \frac{1}{2} | \nabla \bm{u}|^2 + \kappa
  \bm{u} \cdot \curl
  \bm{u}+ \varphi (\bm{u}) \right]
  \mathrm{d} x - \frac{1}{2}  \int_U \ensuremath{\boldsymbol{h}}
  [\bm{u}] \cdot \bm{u}
  \hspace{0.17em} \mathrm{d} x.
\end{equation}
This functional combines the Dirichlet exchange energy (favoring uniform spin
alignment), the antisymmetric Dzyaloshinskii--Moriya interaction (inducing
chiral twisting), the magnetocrystalline anisotropy density $\varphi$
(favoring alignment along preferred crystal axes), and the nonlocal
demagnetizing stray field $\bm{h}[\bm{u}]$. The constrained Euler--Lagrange equations
attached to~\eqref{eq:micro_energy_intro} are of almost harmonic map type.
Deferring the rigorous physical derivation to Section~\ref{sec:appsmicromag},
we state the regularity result for these critical points autonomously.

{\myprop{\begin{theorem}[Interior Regularity of Micromagnetic Maps]
  \label{thm:micro_regularity}Let $U \subseteq \RR^m$ {\opt}$m = 2, 3${\cpt}
  be a bounded open set and let $\bm{u} \in H^1 \cap C^0
  \left( U, \Stwo^2 \right)$ be a continuous critical point of the
  micromagnetic energy $\mathcal{G}$. Assume that
  $\ensuremath{\boldsymbol{h}}: L^2 \left( U, \RR^3 \right) \rightarrow L^2
  \left( U, \RR^3 \right)$ satisfies the regularity hypotheses {{\em
  \eqref{eq:demagregp}}} and {{\em \eqref{eq:demagreg}}}.
  
  If the anisotropy density $\varphi$ is of class $C^{k, \beta}$ for $k
  \geqslant 1$ and exponent $\beta \in (0, 1)$, then
  $\bm{u} \in C^{k + 1, \beta}_{\mathrm{loc}} \left( U,
  \Stwo^2 \right)$. In particular, if $\varphi \in C^{\infty}$, then
  $\bm{u} \in C^{\infty} \left( U, \Stwo^2 \right)$.
  
  For $m = 2$, the initial continuity assumption is automatically satisfied
  for finite-energy solutions.
\end{theorem}}}

\subsection{Strategy of Proof}To prove the theorems stated above without
relying on standard conformal uniformization or $\mathcal{H}^1$--BMO duality,
we develop a purely analytic architecture. Synthesizing insights for classical
harmonic maps originating in Chang et al.~{\cite{Changbi_1999,Chang_1999}},
the structural observations of Hélein~{\cite{Helein1991a}}, and the direct
methods of Carbou~{\cite{Carbou_1997}}, we establish regularity through a
flexible Campanato-type iteration scheme.

Framing the problem  within classical linear elliptic theory yields a distinct analytical advantage when addressing the anisotropic generalizations
of Theorem \ref{thm:aniso_regularity}. When the coefficient matrix $A$ is
uniformly elliptic and merely Hölder continuous, and its determinant $\det A$
is non-constant, the system cannot be conformally reduced to an isotropic
equation on a two-dimensional Riemannian manifold. Such rough, non-unimodular systems present challenges for established gauge-theoretic methods.

The architecture of the paper is built upon three structural mechanisms:

{\noindent}{{\em Algebraic Reformulation.}}{ } For
sphere-valued maps, the geometric constraint $|\bm{u}| =
1$ and orthogonality $\bm{f} \cdot
\bm{u}= 0$ yield $| \nabla \bm{u}|^2
\bm{u}= \jaOm \cdot \nabla \bm{u}$,
which allows a Leibniz expansion against an arbitrary constant $\jc \in \RR^{n
+ 1}$:
\begin{equation} - \Delta \bm{u} =
   \divg \left( \jaOm (\bm{u}- \jc) \right) +
   (\bm{u} \wedge \bm{f}) 
   (\bm{u}- \jc) +\bm{f}. \end{equation}
{\noindent}{{\em The Coupled Caccioppoli Estimate.}}{ }
Comparing $\bm{u}$ to its harmonic extension
$\ensuremath{\boldsymbol{h}}_{\rr}$ on $B_{\rr}$, we establish a coupled
Caccioppoli-type bound:
\begin{equation} 
\| \nabla
   (\bm{u}-\ensuremath{\boldsymbol{h}}_{\rr})\|_{L^q
   (B_{\rr})}  \leqslant  C \hspace{0.17em} \|
   \nabla \bm{u}\|_{L^2 \left( B_{\rr} \right)}
   \hspace{0.17em} \left( \rr^m \hspace{0.17em} \Osc
   (\bm{u}, B_{\rr}) \right)^{1 / \pHold} + C \rr
   \hspace{0.17em} \|\bm{f}\|_{L^{\jqq} (B_{\rr})} . \end{equation}
A key algebraic relationship occurs here: the exponents are coupled via
$1 / \jqq = 1 / 2 + 1 / \pHold$. In dimension $m = 2$, this forces the
oscillation exponent $\pHold$ to be  the Sobolev conjugate $\jqq^{\ast}
= 2 \jqq / (2 - \jqq)$. This relationship is what permits the
Campanato iteration to close in 2D without Wente's inequality.

{\noindent}{{\em An abstract regularity principle.}}{ } We
isolate the mechanism behind Hölder regularity as a property of an abstract
class $\class$ satisfying closure under rescaling and a Campanato
decay axiom:
\begin{equation}
\Osc (\bm{u}, B_{\jth}) \leqslant \jgam
   \hspace{0.17em} \Osc (\bm{u}, B_1) + \kappa
   \hspace{0.17em} \|\bm{f}\|_{L^{\jqq} (B_1)}^{\pHold} .
\end{equation}
The paper's progression can be mapped as a structural hierarchy:

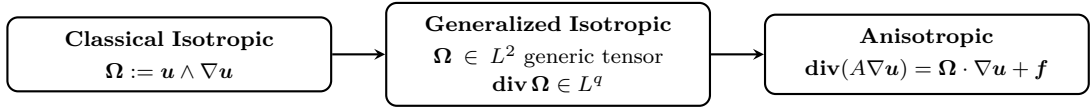
\begin{center}
\begin{tikzpicture}[
    block/.style = {
        rectangle, draw=black, thick, rounded corners=4pt,
        text width=4cm, align=center,
        inner xsep=4pt, inner ysep=5pt,
        font=\footnotesize
    },
    arrow/.style = {
        draw=black, thick, -stealth
    }
]

\node [block] (classical) {
    \textbf{Classical Isotropic}\\[0.5ex]
    \mbox{$\jaOm := \bm{u} \wedge \nabla\bm{u}$}
};

\node [block, right=0.7cm of classical] (generalized) {
    \textbf{Generalized Isotropic}\\[0.5ex]
    $\jaOm\in L^2$ generic tensor \\ 
    \mbox{$\divg \jaOm \in L^\jqq$}
};

\node [block, right=0.7cm of generalized] (anisotropic) {
    \textbf{Anisotropic}\\[0.5ex]
    \mbox{$\divg(A\nabla \bm{u}) = \jaOm \cdot \nabla \bm{u} + \bm{f}$}
};

\draw [arrow] (classical) -- (generalized);
\draw [arrow] (generalized) -- (anisotropic);

\end{tikzpicture}
\end{center}

{\noindent}By formally decomposing the theory into these sequential stages, we provide a transparent framework for identifying what is mathematically lost, and what is gained, when passing from purely geometric analysis to models of heterogeneous physical media.

%
%

\medskip
\noindent\emph{Outline.} In Section~\ref{sec:sec2}, we establish the functional setting and record the foundational linear PDE estimates that drive the subsequent nonlinear iteration schemes. Section~\ref{sec:sec3} introduces the baseline abstract regularity principle (Theorem~\ref{thm:abstract_reg}), establishing H\"older continuity via an algebraic iteration lemma. These axioms are verified for weakly almost harmonic maps into $\Stwo^n$, where the dimensional restriction $m = 2$ enters exclusively through the Sobolev conjugacy coincidence. 

This framework is extended in Section~\ref{sec:sec4} to Rivi\`ere-type systems lacking geometric antisymmetry, demonstrating that the algebraic antisymmetry of the connection form can be replaced by the condition $\divg \jaOm \in L^{\jqq}$. Subsequently, Section~\ref{sec:aniso} addresses the non-unimodular anisotropic regime $-\divg (A \nabla \bm{u})$. We formulate an anisotropic regularity principle (Theorem~\ref{thm:aniso_abstract_reg}) and verify its axioms via localized coefficient-freezing, which imposes the strict dimensional ceiling $m < 4$.

Transitioning away from the $L^2$-critical regime, Section~\ref{sec:sec6} establishes a dimension-free bootstrap
mechanism (Theorem~\ref{lemma:Bootstrap}) to upgrade continuous weak solutions to $C^{\infty}$.
Section~\ref{sec:appsmicromag} applies this theory to physical models in micromagnetics.
Finally, standard interior gradient estimates for weakly harmonic functions are collected in Appendix~\ref{sec:app1}.
%
%

\section{Analytic Setting and Preliminary Tools}\label{sec:sec2}

Our primary focus is on the interior regularity of vector-valued fields defined on an open subset of flat Euclidean space—specifically the unit ball $B_1 \subset \RR^m$—taking values in the standard unit sphere $\Stwo^n \subset \RR^{n + 1}$. Throughout, $B_{\rr}(x_0) \subset \RR^m$ denotes the open ball of radius $\rr > 0$ centered at $x_0$; when the center is the origin or is clear from context, we simply write $B_{\rr}$. For a bounded measurable set $U \subset \RR^m$, we denote the integral average of a measurable map $\bm{v}: U \to \RR^{n+1}$ by $\langle \bm{v} \rangle_U := \jav{U} \bm{v}$.

\subsection{Geometric Setup and Governing Equations}

For a Sobolev map $\bm{u} \in H^1 (B_1, \Stwo^n)$, the \emph{Dirichlet energy} on a sub-ball $B_{\rr}(x_0) \subseteq B_1$ is 
\begin{equation}
\Energy (\bm{u}, B_{\rr}(x_0)) := \frac{1}{2} \int_{B_{\rr}(x_0)} |\nabla \bm{u}|^2  \mathrm{d} x.
\end{equation}

\begin{definition}[Harmonic and Almost Harmonic Maps]
  A map $\bm{u} \in H^1(B_1, \Stwo^n)$ is a \emph{weakly harmonic map} if it is a critical point of the Dirichlet energy subject to the pointwise constraint $|\bm{u}| = 1$, satisfying the Euler--Lagrange equation $-\Delta \bm{u} = |\nabla \bm{u}|^2 \bm{u}$ in $\mathcal{D}'(B_1)$. 
  
  We extend this structure to systems driven by lower-order source terms. A map $\bm{u} \in H^1(B_1, \Stwo^n)$ is a \emph{weakly almost harmonic map} if it solves the perturbed equation:
  \begin{equation}\label{eq:almost_harmonic}
    - \Delta \bm{u} = |\nabla \bm{u}|^2 \bm{u} + \bm{f} \quad \text{in } \mathcal{D}'(B_1),
  \end{equation}
  for a source field $\bm{f} \in L^{\jqq}(B_1, \RR^{n + 1})$, $\jqq > 1$, satisfying the orthogonality condition $\bm{f} \cdot \bm{u} = 0$ almost everywhere in $B_1$.
\end{definition}

\begin{remark}[Algebraic Compatibility]
  The orthogonality condition $\bm{f} \cdot \bm{u} = 0$ is not an artificial assumption, but a rigid algebraic necessity. For any sphere-valued map $|\bm{u}| = 1$, testing the Laplacian against the map itself yields $-\bm{u} \cdot \Delta \bm{u} = |\nabla \bm{u}|^2$. Taking the inner product of \eqref{eq:almost_harmonic} with $\bm{u}$ thus immediately forces $\bm{f} \cdot \bm{u} = 0$ almost everywhere.
\end{remark}

\noindent\emph{Terminology Note.} We reserve the term \emph{harmonic maps} strictly for $\Stwo^n$-valued functions solving the nonlinear constraint equation. Classical $\RR^{n+1}$-valued solutions to the homogeneous linear equation $-\Delta \bm{u} = \bm{0}$ will be referred to as \emph{harmonic functions}.

\subsection{The Oscillation Functional and Scaling Properties}

The engine of our regularity theory relies on quantifying the local deviation of maps from constant states via a variational oscillation functional.

\begin{definition}[$L^{\pHold}$-Oscillation]
  Let $U \subseteq \RR^m$ be an open subset, $\bm{v} \in L^{\pHold}(U)$ for $\pHold \geqslant 1$, and $B \subseteq U$ a ball. The variational $L^{\pHold}$-oscillation of $\bm{v}$ on $B$ is:
  \begin{equation} 
    \Osc (\bm{v}, B) := \inf_{\jc \in \RR^{n + 1}} \jav{B} |\bm{v}(y) - \jc|^{\pHold} \mathrm{d} y.
  \end{equation}
\end{definition}

It is immediate that $\Osc(\bm{v}, B)$ is equivalent to the standard mean oscillation up to a dimensional constant: $\Osc(\bm{v}, B) \leqslant \jav{B} |\bm{v} - \langle \bm{v} \rangle_B|^{\pHold} \leqslant 2^{\pHold} \Osc(\bm{v}, B)$. The functional $\Osc$ obeys the following exact scaling and monotonicity rules, which drive our subsequent Campanato iterations.

\begin{lemma}[Scaling and Monotonicity]\label{lemma:monotonBPsi}
  Let $\bm{v}: U \to \RR^{n + 1}$ be measurable, $\pHold \geqslant 1$, and $B \subseteq U$. Then:
  \begin{enumerate}[i.]
    \item The map $\bm{v} \mapsto \Osc^{1/\pHold}(\bm{v}, B)$ is a seminorm on $L^{\pHold}(B)$.
    \item For $B_{\rr}(x_0) \subseteq U$, the rescaled map $\bm{v}_{x_0, \lambda}(x) := \bm{v}(x_0 + \lambda x)$ defined on $B_{\rr / \lambda}$ satisfies:
    \begin{equation}\label{eq:scalingPsi}
      \Osc \left( \bm{v}, B_{\rr}(x_0) \right) = \Osc \left( \bm{v}_{x_0, \lambda}, B_{\rr / \lambda} \right).
    \end{equation}
    \item For any radii $0 < \rr_1 \leqslant \rr_2$, the volume normalization imposes the geometric decay bound:
    \begin{equation}\label{eq:monotonBPsi}
      \Osc \left( \bm{v}, B_{\rr_1} \right) \leqslant \left( \frac{\rr_2}{\rr_1} \right)^m \Osc \left( \bm{v}, B_{\rr_2} \right).
    \end{equation}
  \end{enumerate}
\end{lemma}

When evaluating the oscillation of a uniformly bounded map (such as a map into a compact manifold), the minimization problem characterizing $\Osc$ can be strictly localized, a fact that will prove technically useful.

\begin{lemma}[Compactification of the Oscillation Infimum]\label{lemma:osc_infimum_bound}
  Let $E \subset \RR^m$ be a domain of finite measure, and let $\bm{u} \in L^{\infty}(E, \RR^{n+1})$. For any exponent $\pHold \geqslant 1$, the infimum defining the $L^{\pHold}$-oscillation can be restricted to the closed ball of radius $M := \|\bm{u}\|_{L^{\infty}(E)}$ without altering its value:
  \begin{equation}
    \inf_{\jc \in \RR^{n+1}} \jav{E} |\bm{u}(y) - \jc|^{\pHold} \mathrm{d} y = \inf_{\substack{\jc \in \RR^{n+1} \\ |\jc| \leqslant M}} \jav{E} |\bm{u}(y) - \jc|^{\pHold} \mathrm{d} y.
  \end{equation}
\end{lemma}

\begin{proof}
  By scaling, it suffices to prove the case $M = 1$. The inequality ``$\leqslant$'' is trivial. For the reverse, let $\jc \in \RR^{n+1}$ with $|\jc| > 1$, and define its radial projection $\jc_{\ast} := \jc / |\jc| \in \Stwo^n$. Because $|\bm{u}| \leqslant 1$ a.e., Cauchy--Schwarz ensures $\bm{u}(y) \cdot \jc \leqslant |\jc|$. We can then estimate the difference in squared distances:
  \begin{align*}
    |\bm{u}(y) - \jc|^2 - |\bm{u}(y) - \jc_{\ast}|^2 
    &= |\jc|^2 - 1 - 2\bm{u}(y) \cdot \jc \left( 1 - \frac{1}{|\jc|} \right) \\
    &\geqslant |\jc|^2 - 1 - 2|\jc| \left( 1 - \frac{1}{|\jc|} \right) = (|\jc| - 1)^2 > 0.
  \end{align*}
  Consequently, $|\bm{u}(y) - \jc_{\ast}| < |\bm{u}(y) - \jc|$ strictly pointwise. Integrating this strict inequality against the exponent $\pHold$ demonstrates that any constant vector outside the unit ball is suboptimal compared to its projection, concluding the proof.
\end{proof}

\noindent\emph{Standard PDE Estimates.} Throughout the remainder of the paper, we will freely utilize classical interior linear estimates (specifically the global and local Calder\'on--Zygmund bounds for divergence-form operators, and standard interior Schauder estimates). We refer the reader to Giaquinta \cite[Chapters 5 and 7]{Giaquinta1983} for the precise statements and structural dependencies of these constants. Similarly, Campanato's classical characterization of $C^{0,\alpha}$ spaces will be invoked directly \cite[Theorem 1.3]{Giaquinta1983}.

\section{The first abstract regularity principle and the proof of
Theorem~\ref{thm:regtheoremharmonic}}\label{sec:sec3}

The main goal of this section is to establish the result already
stated in the introduction, namely:

{\noindent}{\textbf{Theorem~\ref{thm:regtheoremharmonic}.
}}{\itshape{Any weakly almost harmonic map $\bm{u} \in H^1 \left( B_1, \Stwo^n \right)$ 
from the unit two-dimensional disk $B_1 \subseteq \RR^2$ to a sphere 
$\Stwo^n \subset \RR^{n + 1}$, solving the equation
\begin{equation}
  - \lapl \bm{u} = | \grad \bm{u} |^2 \bm{u} + \bm{f},
\end{equation}
with a source term $\bm{f} \in L^{\jqq} (B_1, \RR^{n + 1})$ for some $\jqq > 1$, is locally Hölder continuous.  Specifically, there exists an exponent $\jeta \in (0, 1)$ such that
  $\bm{u} \in C^{0,
  \jeta}_{\ensuremath{\operatorname{loc}}}  (B_1, \RR^{n + 1})$.}}

This result was first established by Hélein in the source-free case and later
extended by Müller and Schikorra~{\cite{Mueller2009}}. Our proof follows the
classical elliptic regularity strategy in the spirit of Campanato.

The proof of Theorem~\ref{thm:regtheoremharmonic} will be presented in
Subsection~\ref{sec:genregproof}. To this end, we first introduce an abstract
regularity principle (stated in Theorem~\ref{thm:abstract_reg}) by defining an
appropriate abstract class of configurations. We then formulate and establish
a second (generalized) abstract regularity theorem, namely
Theorem~\ref{thm:gen_abstract_reg}, in general base dimension $m \geqslant 2$,
although the framework will ultimately be applied to the regularity theory of
generalized systems in the two-dimensional setting to conclude the proof of
Theorem~\ref{thm:regtheoremharmonic}.

\subsection{The first abstract regularity principle}We present an abstract
regularity theory for maps in the Sobolev space $H^1 (B_1, \RR^N)$, where
$B_1$ is the unit disk in $\RR^2$. The framework cleanly separates the
structural features of the underlying equation (such as scaling invariance)
from the analytic ingredients, notably oscillation decay estimates, thereby
yielding a unified approach to partial regularity in geometric analysis.

Let $\class \subseteq H^1 (B_1, \RR^N) \times L^{\jqq} (B_1, \RR^N)$ be a
class of {{\em admissible pairs}} $(\bm{u},
\bm{f})$ defined on the unit disk $B_1 \subset \RR^2$,
which obey the following two axioms:

{\mydef{\begin{axiom}[Locality and Closure]
  \label{ax:closure}The class $\class$ is closed under rescaling.
  Specifically, if $(\bm{u}, \bm{f})
  \in \class$, then for any ball $B_{\rr} (x_0) \subset B_1$, the rescaled
  triplet $(\bm{u}_{x_0, \rr},
  \tilde{\bm{f}}_{x_0, \rr})$ defined by
  \begin{equation}
    \bm{u}_{x_0, \rr} (y) := \bm{u}
    (x_0 + \rr  y), \quad \quad \tilde{\bm{f}}_{x_0, \rr}
    (y) := \rr^2 \bm{f} (x_0 + \rr  y),
  \end{equation}
  is still in $\class$ over the unit ball $B_1 \subseteq \RR^m$.
\end{axiom}}}

{\mydef{\begin{axiom} [The Decay Property]
  \label{ax:decay}  There exist
  structural constants $\jth, \jgam \in (0, 1)$, $\pHold \geqslant 1$, and
  $\kappa > 0$, such that for any pair $(\bm{u},
  \bm{f}) \in \class$, the oscillation strictly contracts
  at the interior scale:
  \begin{equation}
    \Osc (\bm{u}, B_{\jth}) \leqslant \jgam \cdot \Osc
    (\bm{u}, B_1) + \kappa \| \bm{f}
    \|_{L^{\jqq} (B_1)}^{\pHold} .
  \end{equation}
\end{axiom}}}

  Throughout the sequel, we repeatedly exploit several standard scaling
  identities specific to the two-dimensional setting. By a direct change of
  variables, the Dirichlet energy is  invariant under this rescaling:
  \begin{equation}
    \Energy (\bm{u}_{x_0, \rr}, B_1) = \Energy
    (\bm{u}, B_{\rr} (x_0)) .
    \label{eq:2D_energy_scaling}
  \end{equation}
  By contrast, the source term measured in $L^{\jqq}$ acquires the scaling
  factor $\delta := 2 - \frac{2}{\jqq}$, which is strictly positive precisely
  when $\jqq > 1$. More precisely,
  \begin{equation}
    \| \tilde{\bm{f}}_{x_0, \rr} \|_{L^{\jqq} (B_1)} =
    \rr^{\delta} \|\bm{f}\|_{L^{\jqq} (B_{\rr} (x_0))}
    \leqslant \rr^{\delta} \|\bm{f}\|_{L^{\jqq} (B_1)} .
    \label{eq:2D_source_scaling}
  \end{equation}

{\noindent}\text{{\bfseries{Motivating Example.}}} Later in this work, we will
apply the abstract framework by defining $\class$ to be the class of almost
harmonic maps on the unit disk $B_1 \subseteq \RR^2$, whose Dirichlet energy
lies below a fixed critical threshold $\varepsilon_{\ast}$:
\begin{equation}
  \class := \left\{ (\bm{u}, \bm{f})
  \in H^1 (B_1, \Stwo^n) \times L^{\jqq} (B_1, \RR^{n + 1}) \st - \Delta
  \bm{u}= | \nabla \bm{u}|^2
  \bm{u}+\bm{f} \quad \text{and}
  \quad \Energy (\bm{u}, B_1) < \varepsilon_{\ast}
  \right\} .
\end{equation}
Note that, we will formally set $N := n + 1$.

Our first objective is to show that every map belonging to this
two-dimensional admissible class is locally Hölder continuous. The key point
is that, in dimension $m = 2$, the initial mean oscillation on $B_1$ is
naturally controlled by the scale-invariant $H^1$ Dirichlet energy through the
Poincaré--Sobolev inequality. Consequently, no additional $L^{\infty}$ is
required in this setting.

\begin{theorem}[Abstract Regularity Principle]
  \label{thm:abstract_reg} Let $\class \subseteq H^1 (B_1,
  \RR^N) \times L^{\jqq} (B_1, \RR^N)$ defined on $B_1 \subset \RR^2$ be a
  class of admissible pairs satisfying \text{{\bfseries{{{\em {\jaxs~\ref{ax:closure},~\ref{ax:decay}}}}}}}.
  
  If $(\bm{u}, \bm{f}) \in \class$,
  then $\bm{u}$ is locally Hölder continuous in $B_1$.
  Specifically, $\bm{u} \in C^{0,
  \jeta}_{\ensuremath{\operatorname{loc}}} (B_1)$ for every Hölder exponent
  $\jeta$ satisfying :
  \begin{equation}
    \jeta < \min \left( \frac{\ln \jgam}{\pHold \ln \jth}, 2 - \frac{2}{\jqq}
    \right) .
  \end{equation}
  For any compact subset $K \Subset B_1$, the corresponding Hölder seminorm
  depends only on the chosen exponent $\jeta$, the structural parameters
  $\jth, \jgam, \kappa$, the exponent $\pHold$, the initial energy $\Energy
  (\bm{u}, B_1)$, the source norm
  $\|\bm{f}\|_{L^{\jqq} (B_1)}$, and the distance
  $\mathrm{dist} (K, \partial B_1)$.
\end{theorem}

The proof of Theorem~\ref{thm:abstract_reg} is presented in the next
subsection. However, the same structural argument extends naturally to domains
of arbitrary dimension $m \geqslant 2$, provided one imposes an additional
analytic assumption. Indeed, when $m \geqslant 3$, the $H^1$-Dirichlet energy
alone no longer controls the $L^{\pHold}$-oscillation for arbitrarily large
exponents $\pHold$. This dimensional obstruction can be circumvented by
restricting the admissible class to maps that are uniformly bounded in
$L^{\infty}$. Under this assumption, the initial mean oscillation is trivially
controlled by the amplitude $\| \bm{u} \|_{L^{\infty}}$.
More precisely, if $\bm{u} \in H^1 (B_1, E)$, where $E
\subset \RR^N$ is compact, then
\begin{equation}
  \Osc (\bm{u}, B_1) \leqslant
  \ensuremath{\operatorname{diam}}  (E)^p .
\end{equation}
Consequently, one may avoid the dimension-dependent Sobolev embeddings used in
the final part of the proof of Theorem~\ref{thm:abstract_reg} (see
\eqref{eq:usePS2D}). We omit the proof, as it is essentially identical to that
of Theorem~\ref{thm:abstract_reg}, modulo the modifications discussed above;
see also the proof of the abstract regularity principle in arbitrary
dimension, Theorem~\ref{thm:gen_abstract_reg}.

\begin{theorem}[Abstract Regularity Principle in general dimension]
  \label{thm:abstract_reg_nD} Let $B_1 \subset \RR^m$ with
  $m \geqslant 2$. Let $\class \subseteq (H^1 \cap L^{\infty}) (B_1, \RR^N)
  \times L^{\jqq} (B_1, \RR^N)$ be a class of admissible pairs satisfying
  \text{{\bfseries{{{\em \jaxs~\ref{ax:closure},
  \ref{ax:decay}}}}}}, with the source integrability condition $\jqq > m /
  2$.
  
  If $(\bm{u}, \bm{f}) \in \class$,
  then $\bm{u}$ is locally Hölder continuous in $B_1$.
  Specifically, $\bm{u} \in C^{0,
  \jeta}_{\ensuremath{\operatorname{loc}}} (B_1)$ for every Hölder exponent
  $\jeta$ satisfying:
  \begin{equation}
    \jeta < \min \left( \frac{\ln \jgam}{\pHold \ln \jth}, 2 - \frac{m}{\jqq}
    \right) .
  \end{equation}
  For any compact subset $K \Subset B_1$, the corresponding Hölder seminorm
  depends only on the chosen exponent $\jeta$, the structural parameters
  $\jth, \jgam, \kappa$, the exponent $\pHold$, the amplitude
  $\|\bm{u}\|_{L^{\infty} (B_1)}$, the initial energy
  $\Energy (\bm{u}, B_1)$, the source norm
  $\|\bm{f}\|_{L^{\jqq} (B_1)}$, and the distance
  $\mathrm{dist} (K, \partial B_1)$.
\end{theorem}

\begin{remark}[Optimality of the H\"older Exponent]\label{rmk:optimalityholder}
  The two competing mechanisms defining this upper bound are the homogeneous
  decay rate $\left( \ln \jgam \right) / (\pHold \ln \jth)$ from the decay
  axiom and the inhomogeneous contribution $\left( 2 - m / \jqq \right)$
  coming from the $L^{\jqq}$-integrability of the source via the
  $m$-dimensional scaling.
  
  The upper bound for the H\"older exponent $\eta$ in Theorem~\ref{thm:abstract_reg_nD} is sharp. The inhomogeneous threshold $2 - m/\jqq$ coincides exactly with the optimal Morrey--Campanato regularity for the linear Poisson equation driven by an $L^{\jqq}$ source. Furthermore, the strict inequality $\eta < \min \left( \frac{\ln \jgam}{\pHold \ln \jth}, 2 - \frac{m}{\jqq} \right)$ reflects an intrinsic analytic obstruction rather than a suboptimal estimate. In the resonant case, where the homogeneous decay rate matches the inhomogeneous source scaling, the iteration unavoidably accumulates logarithmic penalties of the form $|\ln \rr|$. Conceding an arbitrarily small $\varepsilon$-loss absorbs these factors, allowing the application of Campanato's theorem and yielding the supremum as a strict limit.
\end{remark}

\subsection{Proof of Theorem~\ref{thm:abstract_reg}} For the proof of
Theorem~\ref{thm:abstract_reg} we need to obtain a continuous decay rate from
discrete scaling estimates. For that, a refined version of the standard
algebraic iteration lemma is required. While similar results are common in the
regularity literature (see, for example, Lemma 2.1 in Chapter III of
Giaquinta~{\cite{Giaquinta1983}} or Lemma 7.3 in Giusti~{\cite{Giusti_2003}}),
they often absorb the iteration constant into a generic, unspecified
parameter. For our purposes, it is essential to explicitly track how this
constant depends on the interplay between the homogeneous and inhomogeneous
decay rates. These rates are encoded in the parameters $\jgam$ and $\kappa$,
as well as by the scaling factor associated with the source norm $\|
\bm{f} \|_{L^{\jqq} (B_1)}^{\pHold}$ introduced in
{\jax}~\text{{\bfseries{\ref{ax:decay}}}}.

{\myprop{\begin{lemma}[Algebraic Iteration Lemma]
  \label{lem:giusti_iteration}Let $\psi : (0, \rr_0] \to [0, \infty)$ be a
  non-decreasing function. Suppose there exist constants $\jth \in (0, 1)$,
  $\jgam \in (0, 1)$, $b \geqslant 0$, and $\beta > 0$ such that
  \begin{equation}
    \psi (\jth \rr) \leqslant \jgam \psi (\rr) + b \rr^{\beta}
  \end{equation}
  holds for all $\rr \in (0, \rr_0]$. Let $\jalfa = \left( \ln \jgam \right) /
  \left( \ln \jth \right)$ (so that $\jgam = \jth^{\jalfa}$) and assume
  $\jalfa < \beta$. Then there exists a structural constant
  \begin{equation}
    c := \frac{1}{\jth^{2 \jalfa}  (1 - \jth^{\beta - \jalfa})} > 0,
    \label{eq:structconstc}
  \end{equation}
  depending only on $\jth$, $\jalfa$, and $\beta$, such that
  \begin{equation}
    \psi (\rr) \leqslant c \left[ \frac{\psi (\rr_0)}{\rr_0^{\jalfa}} + b
    \rr_0^{\beta - \jalfa} \right]  \rr^{\jalfa}
  \end{equation}
  for all $\rr \in (0, \rr_0]$.
\end{lemma}}}

The proof of the Algebraic Iteration Lemma \ref{lem:giusti_iteration} is
postponed to the end of the section.

\begin{proof}[Proof of Theorem~\ref{thm:abstract_reg}]
  For convenience, we divide the proof into five steps.{\smallskip}
  
  {\noindent}{{\em Step 1: The Master Iteration Inequality.}} Let
  $(\bm{u}, \bm{f}) \in \class$. Our
  goal is to verify the Campanato condition for $\bm{u}$
  uniformly on $B_{1 / 2}$. Fix an arbitrary center point $x_0 \in B_{1 / 2}$
  and a macroscopic starting radius $\rr_0 := 1 / 2$ so that $B_{\rr} (x_0)
  \subset B_1$ for every $x_0 \in B_{1 / 2}$ and every $\rr \in (0, \rr_0]$.
  For any radius $\rr \in (0, \rr_0]$, we set $\psi (\rr) := \rr^2 \Osc
  (\bm{u}, B_{\rr} (x_0))$. We wish to establish an
  algebraic iteration inequality for $\psi$.
  
  By {\jax} \text{{\bfseries{\ref{ax:closure}}}}, the rescaled pair
  $(\bm{u}_{x_0, \rr},
  \tilde{\bm{f}}_{x_0, \rr})$ belongs to $\class$.
  Applying the contractive decay property
  ({\jax}~\text{{\bfseries{\ref{ax:decay}}}}) to this rescaled pair yields:
  \begin{equation}
    \Osc (\bm{u}_{x_0, \rr}, B_{\jth}) \leqslant \jgam
    \cdot \Osc (\bm{u}_{x_0, \rr}, B_1) + \kappa \|
    \tilde{\bm{f}}_{x_0, \rr} \|_{L^{\jqq}
    (B_1)}^{\pHold} . \label{eq:decaypropforOsc0}
  \end{equation}
  We translate the oscillations back to the physical coordinates on $B_1$.
  Using the scaling identity \eqref{eq:scalingPsi} for the oscillation (mean
  integral), we get:
  \begin{eqnarray}
    \text{LHS:} \quad \Osc (\bm{u}_{x_0, \rr}, B_{\jth})
    & = & \Osc (\bm{u}, B_{\jth \rr} (x_0)) = \left( \jth
    \rr \right)^{- 2} \psi (\jth \rr), \\
    \text{RHS:\quad$\Osc (\bm{u}_{x_0, \rr_0}, B_1)$} & =
    & \Osc (\bm{u}, B_{\rr} (x_0)) = \rr^{- 2} \psi (\rr)
    . 
  \end{eqnarray}
  Therefore, \eqref{eq:decaypropforOsc0} can be rewritten as
  \begin{equation}
    \psi (\jth \rr) \leqslant \jth^2 \jgam \cdot \psi (\rr) + \kappa \jth^2
    \rr^2 \| \tilde{\bm{f}}_{x_0, \rr} \|_{L^{\jqq}
    (B_1)}^{\pHold} . \label{eq:decaypropforOsc}
  \end{equation}
  For the source term, we observe that $\|
  \tilde{\bm{f}}_{x_0, \rr} \|_{L^{\jqq} (B_1)}^{\jqq} =
  \rr^{2 \jqq - 2} \|\bm{f}\|_{L^{\jqq} (B_{\rr}
  (x_0))}^{\jqq} \leqslant \rr^{2 \jqq - 2}
  \|\bm{f}\|_{L^{\jqq} (B_1)}^{\jqq}$. Hence,
  \begin{equation}
    \| \tilde{\bm{f}}_{x_0, \rr} \|_{L^{\jqq}
    (B_1)}^{\pHold} \leqslant \rr^{\pHold \left( 2 - 2 / \jqq \right)}
    \|\bm{f}\|_{L^{\jqq} (B_1)}^{\pHold} .
  \end{equation}
  Substituting the previous bound into the decay property
  \eqref{eq:decaypropforOsc}, we obtain the master iteration inequality:
  \begin{equation}
    \psi (\jth \rr) \leqslant \jth^2 \jgam \cdot \psi (\rr) + b \rr^{\beta} 
    \quad \text{for all } \rr \in (0, \rr_0],
  \end{equation}
  where $\beta := 2 + \pHold \left( 2 - 2 / \jqq \right)$ and $b := \kappa
  \jth^2 \|\bm{f}\|_{L^{\jqq}
  (B_1)}^{\pHold}$.{\smallskip}
  
  {\noindent}\text{{\itshape{Step 2: Continuous Algebraic Decay via the
  Iteration Lemma.}}} To extract a continuous algebraic decay rate from this
  discrete bound, we wish to apply the algebraic iteration lemma
  (Lemma~\ref{lem:giusti_iteration}). This is possible because $\psi$ is
  non-decreasing. The natural algebraic decay exponent associated with the
  homogeneous term is given by $\ln (\jth^2 \jgam) / (\ln \jth)$. The
  iteration lemma requires this exponent to be strictly less than the source
  exponent $\beta = 2 + \pHold \left( 2 - 2 / \jqq \right)$. To achieve the
  sharpest possible decay while satisfying this strict inequality,
  we fix an arbitrarily small $\varepsilon > 0$ and consider the slightly less
  sharp master iteration inequality
  \begin{equation}
    \psi (\jth \rr) \leqslant \jth^2 \jgam_{\varepsilon} \cdot \psi (\rr) + b
    \rr^{\beta} \quad \text{for all } \rr \in (0, \rr_0],
  \end{equation}
  where we define $\jgam_{\varepsilon} := \max (\jgam, \jth^{\beta - 2 -
  \varepsilon})$. Indeed, with this choice of $\jgam_{\varepsilon}$ we have
  $\jgam_{\varepsilon} \in (0, 1)$ because $\jgam, \jth \in (0, 1)$ and we can
  choose $\varepsilon$ small enough such that $\beta - 2 - \varepsilon > 0$.
  We now define our operational exponent $\jalfa_{\varepsilon} := \ln (\jth^2
  \jgam_{\varepsilon}) / (\ln \jth)$ so that $\jth^2 \jgam_{\varepsilon} =
  \jth^{\jalfa_{\varepsilon}}$. By construction, we have
  \begin{equation}
    \frac{\ln \jgam_{\varepsilon}}{\ln \jth} = \frac{| \ln \jgam_{\varepsilon}
    |}{| \ln \jth |} = \frac{\min (| \ln \jgam_{\varepsilon} |, (\beta - 2 -
    \varepsilon) | \ln \jth |)}{| \ln \jth |} = \min \left( \frac{\ln
    \jgam_{\varepsilon}}{\ln \jth}, \beta - 2 - \varepsilon \right) .
    \label{eq:alfaeps1}
  \end{equation}
  Therefore,
  \begin{equation}
    \jalfa_{\varepsilon} := 2 + \frac{| \ln \jgam_{\varepsilon} |}{| \ln \jth
    |} = 2 + \min \left( \frac{\ln \jgam}{\ln \jth}, \beta - 2 - \varepsilon
    \right) . \label{eq:alfaeps2}
  \end{equation}
  In particular, $\beta - \jalfa_{\varepsilon} = \beta - 2 - \min \left(
  \frac{\ln \jgam}{\ln \jth}, \beta - 2 - \varepsilon \right) \geqslant (\beta
  - 2) - (\beta - 2 - \varepsilon) = \varepsilon > 0$.
  
  Because $\jalfa_{\varepsilon} < \beta$, the iteration lemma
  (Lemma~\ref{lem:giusti_iteration}) applies, guaranteeing the
  existence of a purely structural constant $c_{\varepsilon} > 0$ such that
  for all $\rr \in (0, \rr_0]$:
  \begin{equation}
    \psi \left( \rr \right) \leqslant c_{\varepsilon}  \left[ \frac{\psi
    \left( \rr_0 \right)}{\rr_0^{\jalfa_{\varepsilon}}} + b \rr_0^{\beta -
    \jalfa_{\varepsilon}} \right] \rr^{\jalfa_{\varepsilon}} .
  \end{equation}
  Dividing both sides by $\rr^2$, we recover the uniform algebraic decay for
  the mean oscillation:
  \begin{equation}
    \Osc (\bm{u}, B_{\rr} (x_0)) \leqslant
    c_{\varepsilon}  \left[ \frac{\Osc (\bm{u}, B_{\rr_0}
    (x_0))}{\rr_0^{\jalfa_{\varepsilon} - 2}} + b \rr_0^{\beta -
    \jalfa_{\varepsilon}} \right] \rr^{\jalfa_{\varepsilon} - 2},
    \label{eq:usePS2D0}
  \end{equation}
  where, we recall, $b := \kappa \jth^2
  \|\bm{f}\|_{L^{\jqq} (B_1)}^{\pHold}$.
  
  {\noindent}{{\em Step 3: Uniform Bound via the Poincaré-Sobolev
  Inequality.}} To conclude the proof, we must uniformly bound the
  right-hand side of \eqref{eq:usePS2D0} independent of $x_0$. We invoke the
  Poincaré-Sobolev inequality on $B_{\rr_0} (x_0)$.
Proposition applied with $\jqq=2$ and the integrability exponent $\pHold$ (any finite real number, since $q=m$) yields the scale-invariant estimate  
  \begin{equation}
    \label{eq:poincare_sobolev} \left( \jav{B_{\rr_0} (x_0) }
    |\bm{u}- \langle \bm{u}
    \rangle_{B_{\rr_0} (x_0)} |^{\pHold} \right)^{1 / \pHold} \leqslant K_S
    \left( \int_{B_{\rr_0} (x_0)} | \nabla \bm{u}|^2
    \right)^{1 / 2} .
  \end{equation}
  Note that the integral on the right is the total energy (not the mean),
  which is scale-invariant in 2D. That is, $K_S$ does not depend on the radius
  $\rr_0$. Thus, applying \eqref{eq:poincare_sobolev} to our specific ball
  $B_{\rr_0} (x_0)$ (where $\rr_0 = 1 / 2$):
  \begin{equation}
    \Osc (\bm{u}, B_{\rr_0} (x_0)) \leqslant
    \jav{B_{\rr_0} (x_0)} |\bm{u}- \langle
    \bm{u} \rangle |^{\pHold} \leqslant K_S^{\pHold}
    \left(  \int_{B_{\rr_0} (x_0)} | \nabla \bm{u}|^2
    \right)^{\pHold / 2} \leqslant K_S^{\pHold}  \hspace{0.17em} 2^{\pHold /
    2} \hspace{0.17em} (\Energy (\bm{u}, B_{\rr_0}
    (x_0)))^{\pHold / 2} . \label{eq:usePS2D}
  \end{equation}
  Using the monotonicity of the energy ($\Energy (\bm{u},
  B_{\rr_0} (x_0)) \leqslant \Energy (\bm{u}, B_1)$), we
  obtain the uniform bound:
  \begin{equation}
    \Osc (\bm{u}, B_{\rr_0} (x_0)) \leqslant C_S  \Energy
    (\bm{u}, B_1)^{\pHold / 2},
  \end{equation}
  where $C_S := K_S^{\pHold}  \hspace{0.17em} 2^{\pHold / 2}$ is a structural
  constant independent of $x_0$. Substituting this back into
  \eqref{eq:usePS2D0}, we obtain the estimate:
  \begin{equation}
    \Osc (\bm{u}, B_{\rr} (x_0)) \leqslant M
    \rr^{\jalfa_{\varepsilon} - 2} \label{eq:unifcampnato},
  \end{equation}
  where $M$ is a universal constant (independent of $x_0 \in B_{1 / 2}$ and
  $\rr \in \left( 0, \rr_0 \right]$) and
  \begin{equation}
    \jalfa_{\varepsilon} - 2 = \min \left( \frac{\ln \jgam}{\ln \jth}, \beta -
    2 - \varepsilon \right) = \min \left( \frac{\ln \jgam}{\ln \jth}, \pHold
    \left( 2 - 2 / \jqq \right) - \varepsilon \right) . \label{eq:alfaepsm2}
  \end{equation}
  Finally, observe that since $\jqq < 2$, we have $\left( 2 - 2 / \jqq \right)
  < 1$. Consequently, for sufficiently small $\varepsilon > 0$,
  $\jalfa_{\varepsilon} - 2 < \pHold$.
  
  {\noindent}{{\em Step 4: The Campanato Condition and Hölder Continuity on
  $B_{1 / 2}$.}} Since inequality \eqref{eq:unifcampnato} holds uniformly
  for every $x_0 \in B_{1 / 2}$, it satisfies precisely the Campanato
  condition with exponent $\left( \jalfa_{\varepsilon} - 2 \right)$. Therefore, $\bm{u} \in
  C^{0, \jeta_{\varepsilon}} (B_{1 / 2})$ with $\jeta_{\varepsilon} = \left(
  \jalfa_{\varepsilon} - 2 \right) / \pHold$. Substituting our exact choice
  for $\left( \jalfa_{\varepsilon} - 2 \right)$ directly yields:
  \begin{equation}
    \jeta_{\varepsilon} = \min \left( \frac{\ln \jgam}{\pHold \, \ln \jth}, 2
    - \frac{2}{\jqq} - \frac{\varepsilon}{\pHold} \right),
  \end{equation}
  matching the formula in the theorem statement as soon as we pass
  to the limit for $\varepsilon \rightarrow 0$.
  
  The Hölder seminorm $[\bm{u}]_{C^{0,
  \jeta_{\varepsilon}} (B_{1 / 2})}$ depends only on $\jth, \jgam, \kappa,
  \pHold, \Energy (\bm{u}, B_1)$, and
  $\|\bm{f}\|_{L^{\jqq} (B_1)}$ through the constants
  appearing in the iteration lemma and the Poincaré-Sobolev constant $K_S$.
  
  {\noindent}{{\em Step 5: Extension to arbitrary compact subsets of
  $B_1$.}} To conclude the proof, we extend this regularity from $B_{1 / 2}$
  to the entirety of $B_1$ and explicitly track the seminorm dependency. Let
  $K \subset B_1$ be an arbitrary compact subset, and let $R := \mathrm{dist}
  (K, \partial B_1) > 0$ be its distance to the boundary. For any point $x_0
  \in K$, the ball $B_R (x_0)$ is strictly contained within $B_1$. We define
  the rescaled pair $(\bm{v},
  \bm{g}) := (\bm{u}_{x_0, R},
  \tilde{\bm{f}}_{x_0, R})$ on the unit ball $B_1$. By
  {\jax}~\text{{\bfseries{\ref{ax:closure}}}}, this rescaled pair belongs to
  $\class$.
  
  Applying the exact uniform bound established in the previous steps to this
  rescaled pair, we know that $\bm{v} \in C^{0, \jeta}
  (B_{1 / 2})$. Furthermore, its Hölder seminorm
  $[\bm{v}]_{C^{0, \jeta} (B_{1 / 2})}$ is bounded
  uniformly by the structural constants, the energy $\Energy
  (\bm{v}, B_1)$, and the source norm
  $\|\bm{g}\|_{L^{\jqq} (B_1)}$. By the monotonicity of
  the energy and the 2D source scaling, these are universally bounded by the
  initial global quantities: $\Energy (\bm{v}, B_1)
  \leqslant \Energy (\bm{u}, B_1)$ and
  $\|\bm{g}\|_{L^{\jqq} (B_1)} \leqslant R^{2 - 2 / \jqq}
  \|\bm{f}\|_{L^{\jqq} (B_1)} \leqslant
  \|\bm{f}\|_{L^{\jqq} (B_1)}$.
  
  Finally, we translate this back to the original map
  $\bm{u}$ using the equivalence of scaled Hölder
  seminorms: the restriction of $\bm{v}$ to $B_{1 / 2}$
  corresponds directly to the restriction of $\bm{u}$ to
  $B_{R / 2} (x_0)$. The seminorms obey the scaling identity
  \begin{equation}
    [\bm{u}]_{C^{0, \jeta_{\varepsilon}} (B_{R / 2}
    (x_0))} = R^{- \jeta_{\varepsilon}} [\bm{v}]_{C^{0,
    \jeta} (B_{1 / 2})} .
  \end{equation}
  Because $x_0$ was an arbitrary point in $K$, a standard finite covering
  argument guarantees that $\bm{u} \in C^{0,
  \jeta_{\varepsilon}} (K)$. The explicit appearance of the scaling factor
  $R^{- \jeta_{\varepsilon}} = \mathrm{dist} (K, \partial B_1)^{-
  \jeta_{\varepsilon}}$ dictates precisely why the final Hölder seminorm on
  $K$ depends on the distance to the boundary.
  
  Finally, since $\varepsilon > 0$ can be chosen arbitrarily small, this
  guarantees that $\bm{u} \in C^{0, \eta} (K)$ for every
  exponent $\eta$ strictly less than the critical threshold, concluding the
  proof.
\end{proof}

\begin{proof}[Proof of the Algebraic Iteration Lemma \ref{lem:giusti_iteration}]
  We first derive the bound along the discrete sequence $\rr_j = \jth^j \rr_0$
  for $j \in \NN_0$. For $j = 1$ the given inequality yields $\psi (\jth 
  \rr_0) \leqslant \jgam \psi (\rr_0) + b \rr_0^{\beta}$. Iterating, using
  induction on $j \geqslant 1$, gives
  \begin{equation}
    \psi (\jth^j \rr_0) \leqslant \jgam^j \psi (\rr_0) + b \sum_{i = 0}^{j -
    1} \jgam^i  (\jth^{j - 1 - i}  \rr_0)^{\beta} .
  \end{equation}
  Substituting $\jgam = \jth^{\jalfa}$ and rewriting the general term in the
  sum, $\jgam^i  (\jth^{j - 1 - i})^{\beta} = \jth^{\jalfa i} \cdot
  \jth^{\beta (j - 1 - i)} = \jth^{\jalfa  (j - 1)} \cdot \jth^{(\beta -
  \jalfa)  (j - 1 - i)}$, we obtain
  \begin{eqnarray}
    \psi (\jth^j \rr_0) & \leqslant & \jth^{\jalfa j} \psi (\rr_0) + b
    \rr_0^{\beta}  \jth^{\jalfa  (j - 1)}  \sum_{i = 0}^{j - 1} \jth^{(\beta -
    \jalfa)  (j - 1 - i)} \nonumber\\
    & = & \jth^{\jalfa j} \left[ \psi (\rr_0) + \frac{b
    \rr_0^{\beta}}{\jth^{\jalfa}} \sum_{k = 0}^{j - 1} (\jth^{\beta -
    \jalfa})^k \right], \nonumber
  \end{eqnarray}
  where for the last equality we re-indexed \ the sum by setting $k := j - 1 -
  i$. Since $\beta > \alpha$ and $\jth \in (0, 1)$, the ratio $\jth^{\beta -
  \jalfa} < 1$. Thus the finite geometric sum is bounded by the sum of the
  geometric series:
  \begin{equation}
    \psi (\jth^j \rr_0) \leqslant \jth^{\jalfa j}  \left[ \psi (\rr_0) +
    \frac{b \rr_0^{\beta}}{\jth^{\jalfa} (1 - \jth^{\beta - \jalfa})} \right]
    . \label{eq:discbound1}
  \end{equation}
  To extend the estimate to arbitrary $\rr \in (0, \rr_0]$, choose the unique
  integer $j \geqslant 0$ such that $\jth^{j + 1}  \rr_0 < \rr \leqslant
  \jth^j  \rr_0$. Since $\psi$ is non-decreasing and $\rr \leqslant \jth^j 
  \rr_0$, we have $\psi (\rr) \leqslant \psi (\jth^j \rr_0)$. From the lower
  bound on $\rr$ we deduce $\jth^j < \jth^{- 1} (\rr / \rr_0)$. Raising to the
  power $\jalfa > 0$ (which preserves the inequality) gives
  \begin{equation}
    \jth^{\jalfa j} < \frac{1}{\jth^{\jalfa}} \left( \frac{\rr}{\rr_0}
    \right)^{\jalfa} . \label{eq:discbound2}
  \end{equation}
  Inserting this into the discrete bound produces
  \begin{equation}
    \psi (\rr) \leqslant \psi (\jth^j \rr_0)
    \overset{\eqref{eq:discbound1}}{\leqslant} \jth^{\jalfa j}  \left[ \psi
    (\rr_0) + \frac{b \rr_0^{\beta}}{\jth^{\jalfa} (1 - \jth^{\beta -
    \jalfa})} \right] \overset{\eqref{eq:discbound2}}{\leqslant}
    \frac{1}{\jth^{\jalfa}} \left( \frac{\rr}{\rr_0} \right)^{\jalfa}  \left[
    \psi (\rr_0) + \frac{b \rr_0^{\beta}}{\jth^{\jalfa} (1 - \jth^{\beta -
    \jalfa})} \right] .
  \end{equation}
  Distributing terms,
  \begin{equation} \psi (\rr) \leqslant \frac{1}{\jth^{\jalfa}}  \frac{\psi
     (\rr_0)}{\rr_0^{\jalfa}}  \rr^{\jalfa} + \frac{1}{\jth^{2 \jalfa}  (1 -
     \jth^{\beta - \jalfa})} b \rr_0^{\beta - \jalfa}  \rr^{\jalfa} . \end{equation}
  Because $\jth \in (0, 1)$ and $\jalfa > 0$, we have $1 / \jth^{\jalfa}
  \leqslant 1 / \jth^{2 \jalfa}$ and $1 / \jth^{2 \jalfa} < 1 / [\jth^{2
  \jalfa} (1 - \jth^{\beta - \jalfa})]$ (since $1 - \jth^{\beta - \jalfa} <
  1$). Therefore both coefficients are bounded by the structural constant $c
  := 1 / [\jth^{2 \jalfa} (1 - \jth^{\beta - \jalfa})]$, which completes the
  proof.
\end{proof}

\subsection{Hölder regularity of almost harmonic maps: proof of
Theorem~\ref{thm:regtheoremharmonic}} The proof of
Theorem~\ref{thm:regtheoremharmonic} hinges on the following proposition,
which demonstrates that small Dirichlet energy on the unit ball, coupled with
a sufficiently integrable source term, implies a contraction of the
$L^{\pHold}$-oscillation at a smaller scale. Specifically, the content of the
next result is to prove that the class of almost harmonic maps
satisfies the epsilon-decay property formulated in {\jax}~\ref{ax:decay}.

{\myprop{\begin{lemma}[Decay Property]
  \label{lem:decay}For any integrability exponent $\jqq \in (1, 2)$, we denote by
  $\pHold := 2 \jqq / \left( 2 - \jqq \right)$ the corresponding
  oscillation exponent. There exist structural constants $\jth \in (0, 1)$,
  $\jgam \in (0, 1)$, $\kappa > 0$, and an energy threshold
  $\varepsilon_{\ast} > 0$ such that for an arbitrary almost harmonic map
  $\bm{u} \in H^1 \left( B_1, \Stwo^n \right)$ defined on
  the unit ball $B_1 \subset \RR^2$ with associated source term
  $\bm{f} \in L^{\jqq} (B_1, \RR^{n + 1})$:
  \begin{equation}
    \text{if} \quad \Energy (\bm{u}, B_1) <
    \varepsilon_{\ast}, \quad \text{then} \quad \Osc
    (\bm{u}, B_{\jth}) \leqslant \jgam \cdot \Osc
    (\bm{u}, B_1) + \kappa \| \bm{f}
    \|_{L^{\jqq} (B_1)}^{\pHold} . \label{eq:CH1.1}
  \end{equation}
  In fact, one can always take $\jgam = 1 / 2$.
\end{lemma}}}

\begin{remark}[The $L^2$ Integrability Ceiling]
  \label{rmk:redsourceintq}
  Let $\jqq_{\mathrm{src}}$ denote the physical integrability of the source term $\bm{f}$. As $B_1$ is bounded, we may without loss of generality restrict our analysis to an operational exponent $1 < \jqq < 2$, even when $\jqq_{\mathrm{src}} \geqslant 2$. This reduction is not merely a matter of convenience: it is a structural necessity imposed by the critical quadratic nonlinearity $|\nabla \bm{u}|^2 \bm{u}$. Specifically, recasting this term as $\divg(\jaOm \bm{u})$ with $\jaOm \in L^2$ imposes a strict analytic threshold: estimating this flux in $L^{\jqq}$ via H\"older's inequality requires measuring the oscillation of $\bm{u}$ in $L^{\pHold}$, which couples the exponents via $1/\jqq = 1/2 + 1/\pHold$. Simultaneously, closing the Campanato iteration via the Sobolev--Poincar\'e inequality dictates that $\pHold$ must equal the two-dimensional Sobolev conjugate $\jqq^\ast = 2\jqq / (2 - \jqq)$. This algebraic rigidity fundamentally limits the operational scheme to $\jqq < 2$.
  Because the oscillation exponent $\pHold \to \infty$ as $\jqq \to 2^-$, optimal H\"older regularity is achieved by maximizing $\jqq$. If $\jqq_{\mathrm{src}} < 2$, the optimal choice is exactly $\jqq = \jqq_{\mathrm{src}}$. Conversely, if $\jqq_{\mathrm{src}} \geqslant 2$, any source integrability beyond $L^2$ is analytically inaccessible to this iteration architecture; optimal regularity is instead captured by taking $\jqq$ arbitrarily close to $2$. We therefore unify our subsequent analysis under the assumption that $\bm{f} \in L^{\jqq}(B_1)$ for a fixed exponent $1 < \jqq < 2$.
\end{remark}

\begin{remark}[Interdependence of Structural Constants]
  The parameters $\jth, \jgam, \kappa$, and $\varepsilon_{\ast}$ introduced in Lemma~\ref{lem:decay} are strictly interdependent. The constructive proof explicitly fixes $\jgam = 1/2$ and, via H\"older's inequality, rigidly locks the exponent to $\pHold = 2\jqq / (2 - \jqq)$. This algebraic coupling determines the scaling factor $\jth \equiv \jth(C_{m, \jqq}, \pHold)$, which ultimately dictates the critical energy threshold $\varepsilon_{\ast} = \varepsilon_{\ast}(\jth)$.
\end{remark}

\begin{remark}[The Structural Coincidence of Dimension Two]
  The success of the fundamental decay estimate \eqref{eq:decay_master} hinges on a structural coincidence specific to dimension $m = 2$. To close the Campanato iteration, the integrability gained via the Sobolev embedding must perfectly match the integrability demanded by H\"older's inequality. 
  Specifically, bounding the error $\bm{w} := \bm{u} - \bm{h}$ via the Sobolev embedding $W^{1, \jqq} \hookrightarrow L^{\pSob}$ yields the dimensional relation $1/\pSob = 1/\jqq - 1/m$. Conversely, controlling the right-hand side $\divg(\jaOm(\bm{u} - \jc))$ in $L^{\jqq}$ given the fixed energy density $\jaOm \in L^2$ requires measuring the deviation in $L^{\pHold}$ such that $1/\jqq = 1/2 + 1/\pHold$. The factor of $2$ in this latter relation originates entirely from the natural energy space and is strictly independent of the spatial dimension. 
  For the iteration to close, the available Sobolev output must serve as the required H\"older input ($\pSob = \pHold$). Equating these relations yields $1/\jqq - 1/m = 1/\jqq - 1/2$,  forcing $m = 2$. Thus, dimension two is the unique setting where the regularizing effect of the Laplacian exactly balances the integrability loss induced by the critical quadratic nonlinearity. We note that a similar observation regarding this critical algebraic balance was previously highlighted in \cite{Chang_1999} in their study of harmonic maps.
\end{remark}

We postpone the proof of Lemma~\ref{lem:decay} to the next section. Here
instead, we show how we can apply the abstract regularity theorem
(Theorem~\ref{thm:abstract_reg}) to infer the local Hölder regularity of
weakly harmonic maps in dimension two stated in
Theorem~\ref{thm:regtheoremharmonic}.

\begin{proof}[Proof of Theorem~\ref{thm:regtheoremharmonic}]
  We apply the abstract regularity framework. Let $\varepsilon_{\ast}$ be the
  critical energy threshold provided by Lemma~\ref{lem:decay}. We define the
  admissible class $\class$ as the set of all $\Stwo^n$-valued almost harmonic
  maps whose Dirichlet energy  respects this threshold:
  \begin{equation} \class = \left\{ (\bm{u},
     \bm{f}) \in H^1 (B_1, \Stwo^n) \times L^{\jqq} (B_1,
     \RR^{n + 1}) \st  - \Delta
     \bm{u}= | \grad \bm{u}|^2
     \bm{u}+\bm{f} \quad \text{and}
     \quad \Energy (\bm{u}, B_1) < \varepsilon_{\ast}
     \right\} . \end{equation}
  We need to verify that this class satisfies
  {\jax}~\ref{ax:closure}  and
  {\jax}~\ref{ax:decay}.
  
  {\noindent}{\em Step 1: Verification of} {\jax}~\ref{ax:closure}
  (closure under rescaling). Let $(\bm{u},
  \bm{f}) \in \class$ and let
  $B_\rr (x_0) \subset B_1$. Consider the rescaled pair
  defined for $x \in B_1$ by $\bm{u}_{x_0, \rr} (x) :=
  \bm{u} (x_0 + \rr x)$ and
  $\tilde{\bm{f}}_{x_0, \rr} (x) := \rr^2
  \bm{f} (x_0 + \rr x)$.
  \smallskip
  
  \begin{itemize}
    \item \text{{\itshape{Constraint:}}} Since $\bm{u}$
    takes values in $\Stwo^n$ almost everywhere, clearly
    $\bm{u}_{x_0, \rr}$ also takes values in $\Stwo^n$
    almost everywhere.
    \smallskip
    
    \item \text{{\itshape{Equation:}}} In dimension $m = 2$, the Laplacian
    scales as $\lapl \bm{u}_{x_0, \rr} (y) = \rr^2 \lapl
    \bm{u} \left( x_0 + \rr y \right)$. Similarly, the
    quadratic term scales as $| \grad \bm{u}_{x_0, \rr}
    (y) |^2 = \rr^2  | \grad \bm{u} \left( x_0 + \rr y
    \right) |^2$. Thus, multiplying the original equation evaluated at $x_0 +
    \rr x$ by $\rr^2$, we obtain:
    \begin{equation} - \Delta \bm{u}_{x_0, \rr} = | \nabla
       \bm{u}_{x_0, \rr} |^2
       \bm{u}_{x_0, \rr} +
       \tilde{\bm{f}}_{x_0, \rr} . \end{equation}
    Thus, the structural equation is invariant.
\smallskip
    
    \item \text{{\itshape{Energy Threshold:}}} Because we are in
    dimension two, the Dirichlet energy is  invariant under this
    rescaling. By the monotonicity of the energy integral over subdomains, we
    have:
    \begin{equation} \Energy (\bm{u}_{x_0, \rr}, B_1) = \Energy
       (\bm{u}, B_{\rr} (x_0)) \leqslant \Energy
       (\bm{u}, B_1) < \varepsilon_{\ast} . \end{equation}
  \end{itemize}
  Thus, the rescaled pair satisfies all structural conditions and preserves
  the strict energy bound, ensuring $(\bm{u}_{x_0, \rr},
  \tilde{\bm{f}}_{x_0, \rr}) \in \class$.{\smallskip}
  
  {\noindent}{{\em Step 2: Verification of} {\jax}~\ref{ax:decay}} (the decay property). Let
  $(\bm{u}, \bm{f}) \in \class$. By
  the explicit definition of our class $\class$, the pair satisfies the global
  bound $\Energy (\bm{u}, B_1) < \varepsilon_{\ast}$.
  Therefore, the hypothesis of Lemma~\ref{lem:decay} is met, which immediately
  yields:
  \begin{equation} \Osc (\bm{u}, B_{\jth}) \leqslant \jgam \cdot \Osc
     (\bm{u}, B_1) + \kappa
     \|\bm{f}\|_{L^{\jqq} (B_1)}^{\pHold} . \end{equation}
  This guarantees {\jax}~\ref{ax:decay} is satisfied for all elements in
  $\class$.
  
  {\noindent}{{\em Step 3: Conclusion and Removal of the Small-Energy
  Hypothesis.}} Having verified the two axioms, the abstract regularity
  principle (Theorem~\ref{thm:abstract_reg}) dictates that any map
  $\bm{u}$ belonging to a pair in $\class$ is locally
  Hölder continuous in $B_1$.
  
  To conclude the proof of the main theorem, we must explain why this
  regularity holds for {{\em any}} weakly almost harmonic map, not just
  those with small initial energy. Let $\bm{u} \in H^1
  (B_1, \Stwo^n)$ be an arbitrary almost harmonic map with source
  $\bm{f} \in L^{\jqq} (B_1, \RR^{n + 1})$. This map
  $\bm{u}$ might have a large total Dirichlet energy,
  meaning $(\bm{u}, \bm{f}) \not\in
  \class$.
  
  However, let $x_0 \in B_1$ be an arbitrary point, and let $R := 1 - |x_0 | >
  0$ denote its distance to the boundary. Because $\nabla
  \bm{u} \in L^2 (B_1)$, by the absolute continuity of
  the Lebesgue integral we can always find a sufficiently small radius $\rr
  \in (0, R)$ such that $\Energy (\bm{u}, B_{\rr} (x_0))
  < \varepsilon_{\ast}$. If we now construct the rescaled pair
  $(\bm{u}_{x_0, \rr},
  \tilde{\bm{f}}_{x_0, \rr})$, it is defined on the unit
  ball $B_1$, solves the required equation, and  satisfies the
  small-energy threshold $\Energy (\bm{u}_{x_0, \rr},
  B_1) < \varepsilon_{\ast}$. Consequently, this rescaled pair belongs to our
  class $\class$. By our abstract theorem, $\bm{u}_{x_0,
  \rr}$ is locally Hölder continuous in $B_1$. Scaling back to the original
  spatial coordinates, this implies that the original map
  $\bm{u}$ is Hölder continuous in the neighborhood
  $B_{\rr / 2} (x_0)$. Since the point $x_0 \in B_1$ was chosen arbitrarily,
  the map $\bm{u}$ is locally Hölder continuous
  everywhere in $B_1$, completing the proof.
\end{proof}

\subsection{Proof of the Fundamental Lemma~\ref{lem:decay}}
The rest of these subsection provide the complete proof of the fundamental
decay estimate stated in Lemma~\ref{lem:decay}. The core objective is to demonstrate that if the Dirichlet energy of an almost harmonic map
is sufficiently small on a given ball, its local oscillation decays
geometrically as the observation radius shrinks, thus triggering the Campanato
regularity iteration.

\subsection*{Notation and Algebraic Setup} We begin by establishing our
notational conventions. Although the underlying linear algebra is 
elementary, the specific notation for operations such as the wedge product
between vectors and matrices of vector fields is not  standardized
across the literature. To prevent any ambiguity and ensure that our subsequent
algebraic manipulations remain completely transparent, we explicitly define
these operations below.

{\mydef{\begin{definition}[Convention on gradient and divergence]
  Let $\bm{u}: B_1 \subset \RR^m \to \RR^{n + 1}$. We
  adopt the convention that the gradient $\nabla \bm{u}$
  is the transpose of the Jacobian matrix. Thus, $\nabla
  \bm{u}$ is an $m \times (n + 1)$ matrix whose columns
  consist of the gradients of the scalar components of
  $\bm{u}$:
  \begin{equation}
    \nabla \bm{u} := \left(\begin{array}{ccc}
      | &  & |\\
      \nabla u^1 & \ldots & \nabla u^{n + 1}\\
      | &  & |
    \end{array}\right) .
  \end{equation}
  Consequently, the notation $\nabla u^{\beta}$ refers to the $\beta$-th
  column of this matrix, which is a vector in $\RR^m$. The squared norm $|
  \nabla \bm{u}|^2$ is the Frobenius norm, given by $|
  \nabla \bm{u}|^2 = \sum_{\beta = 1}^{n + 1} | \nabla
  u^{\beta} |^2$.
  
  For a standard vector field $\bm{v}: \RR^m \rightarrow
  \RR$, we denote its classical scalar divergence by $\divr
  \bm{v}$. We extend this operator componentwise using
  the bold notation $\divg$ for higher-order tensor fields. Specifically, if
  $\ensuremath{\boldsymbol{\omega}}$ is a vector of $n + 1$ vector fields
  (i.e., an element of $\RR^{n + 1} \otimes \RR^m$), $\divg
  \ensuremath{\boldsymbol{\omega}}$ is the vector in $\RR^{n + 1}$ whose
  entries are the classical divergences of the individual vector fields.
  Similarly, if $\jaOm$ is an $(n + 1) \times (n + 1)$ matrix whose entries
  are vector fields in $\RR^m$ (an element of $\RR^{(n + 1) \times (n + 1)}
  \otimes \RR^m$), then $\divg  \jaOm$ is the $(n + 1) \times (n + 1)$ matrix
  of their respective divergences.
  
  Under these conventions, by identifying $\nabla \bm{u}$
  as an element of $\RR^{n + 1} \otimes \RR^m$, we naturally recover the
  classical componentwise relation $\divg (\nabla \bm{u})
  = \Delta \bm{u}$.
\end{definition}}}

{\mydef{\begin{definition}[The generalized wedge product]
  For any two vectors $\bm{v},
  \bm{w} \in \RR^{n + 1}$, their exterior product
  $\bm{v} \wedge \bm{w}$ is
  canonically identified with the skew-symmetric matrix (or bivector) in
  $\mathfrak{s}\mathfrak{o} (n + 1)$ whose entries are given by:
  \begin{equation}
    (\bm{v} \wedge \bm{w})^{\alpha
    \beta} := v^{\alpha} w^{\beta} - v^{\beta} w^{\alpha}, \quad \text{for } 1
    \leqslant \alpha, \beta \leqslant n + 1.
  \end{equation}
  We extend this algebraic operation to objects that carry
  additional spatial structure over the domain $\RR^m$. If
  $\ensuremath{\boldsymbol{\Phi}}$ is a standard vector in $\RR^{n + 1}$ and
  $\ensuremath{\boldsymbol{\Psi}} \in \RR^{n + 1} \otimes \RR^m$ is a
  collection of spatial vector fields, their wedge product
  $\ensuremath{\boldsymbol{\Phi}} \wedge \ensuremath{\boldsymbol{\Psi}}$ acts
   on the target space components. The result is a skew-symmetric
  matrix of spatial vector fields taking values in $\mathfrak{s}\mathfrak{o}
  (n + 1) \otimes \RR^m$, with entries defined identically:
  \begin{equation}
    (\ensuremath{\boldsymbol{\Phi}} \wedge
    \ensuremath{\boldsymbol{\Psi}})^{\alpha \beta} := \Phi^{\alpha}
    \Psi^{\beta} - \Phi^{\beta} \Psi^{\alpha} .
  \end{equation}
\end{definition}}}

{\mydef{\begin{definition} [The antisymmetric potential $\jaOm$]
  Let
  $\bm{u} \in \RR^{n + 1}$ and let $\nabla
  \bm{u} \in \RR^{m \times (n + 1)}$ be defined as above.
  We define the \emph{antisymmetric potential} $\jaOm$ as the
  matrix of vector fields of size $(n + 1) \times (n + 1)$:
  \begin{equation}
    \jaOm := \bm{u} \wedge \nabla
    \bm{u} \in \mathfrak{s}\mathfrak{o} (n + 1) \otimes
    \RR^m . \label{eq:defantsympot}
  \end{equation}
  Explicitly, the entry of $\jaOm$ at {{\em row}} $\alpha$ and {{\em
  column}} $\beta$, denoted by $\Omega^{\alpha \beta}$, is the
  $\RR^m$-valued vector field given by:
  \begin{equation}
    \Omega^{\alpha \beta} := u^{\alpha} \nabla u^{\beta} - u^{\beta} \nabla
    u^{\alpha} .
  \end{equation}
  Note that $\nabla u^{\beta}$ is the $\beta$-th column of the matrix $\nabla
  \bm{u}$ and that $\jaOm$ is a \emph{skew--symmetric
  matrix} consisting of $\RR^m$-valued vector fields: $\Omega^{\alpha \beta} = -
  \Omega^{\beta \alpha}$. In other words, $\jaOm$ is interpreted geometrically
  as a connection 1-form taking values in $\mathfrak{s}\mathfrak{o} (n + 1)$,
  i.e., $\jaOm \in \mathfrak{s}\mathfrak{o} (n + 1) \otimes \RR^m \subset
  \RR^{(n + 1) \times (n + 1)} \otimes \RR^m$.
\end{definition}}}

\begin{remark}[Norm compatibility]
  Let $\bm{w} := \jaOm \bm{v}$. Here
  $\bm{v} \in \RR^{n + 1}$, while $\jaOm$ is an $(n + 1)
  \times (n + 1)$ matrix where each entry $\Omega^{\alpha \beta}$ is a vector
  in $\RR^m$. The resulting $\bm{w}$ is a vector of size
  $n + 1$, where each component $w^{\alpha} \in \RR^m$ is an $m$-dimensional
  vector given by $w^{\alpha} = \sum_{\beta = 1}^{n + 1} \Omega^{\alpha \beta}
  v^{\beta}$. We use the standard Euclidean norm for
  $\bm{v}$ and the Frobenius norm for $\jaOm$:
  \begin{equation}
    |\bm{v}|^2 := \sum_{\beta = 1}^{n + 1} |v^{\beta}
    |^2, \qquad | \jaOm |^2 := \sum_{\alpha, \beta = 1}^{n + 1} |
    \Omega^{\alpha \beta} |^2_{\RR^m} .
  \end{equation}
  A straightforward application of the Cauchy-Schwarz inequality gives the norm compatibility relation
  \begin{equation}
    | \jaOm \bm{v}| \leqslant | \jaOm | \cdot
    \hspace{0.17em} |\bm{v}|,
    \label{eq:normcompatibility}
  \end{equation}
  which holds for every tensor $\jaOm \in \RR^{(n + 1) \times (n + 1)} \otimes
  \RR^m$.

Moreover, if $\jaOm =\bm{u}
  \wedge \nabla \bm{u}$ then
  \begin{equation}
    \left| \jaOm \right| = \sqrt{2} \left| \grad \bm{u}
    \right| . \label{eq:normasympot}
  \end{equation}
  Indeed, $| \jaOm |^2 = 2 | \nabla
    \bm{u}|^2 - 2 | (\nabla \bm{u})
    \bm{u} |^2 $,
  and $(\nabla \bm{u})
  \bm{u}=\ensuremath{\boldsymbol{0}}$ because of $|
  \bm{u} | = 1$.
\end{remark}

{\myprop{\begin{lemma}
  A function $\bm{u} \in H^1 \left( B_1, \Stwo^n \right)$
  is an almost harmonic map if, and only if, for every constant vector $\jc
  \in \RR^{n + 1}$, the following identity holds in the sense of
  distributions:
  \begin{equation}
    - \Delta \bm{u}= \divg \left( \jaOm
    (\bm{u}- \jc) \right) + (\bm{u}
    \wedge \bm{f}) \left( \bm{u}-
    \jc \right) +\bm{f}. \label{eq:helein_div}
  \end{equation}
  In coordinates:
  \begin{equation}
    - \Delta u^{\alpha} \eqs \ensuremath{\operatorname{div}} \left[
    \sum_{\beta = 1}^{n + 1} \Omega^{\alpha \beta} (u^{\beta}
    -c^{\beta}) \right] + \sum_{\beta = 1}^{n + 1}
    (\bm{u} \wedge \bm{f})^{\alpha
    \beta} (u^{\beta} -c^{\beta}) + f^{\alpha}, \label{eq:CH1.9}
  \end{equation}
  where, we recall, $\Omega^{\alpha \beta} := u^{\alpha} \nabla u^{\beta} -
  u^{\beta} \nabla u^{\alpha}$ is an $\RR^m$-valued vector field and
  $(\bm{u} \wedge \bm{f})^{\alpha
  \beta} := u^{\alpha} f^{\beta} - u^{\beta} f^{\alpha} .$
\end{lemma}}}

\begin{remark} [Notation]
    The expression $\bm{w} := \jaOm
    (\bm{u}- \jc)$ denotes a matrix-vector product. The
    result, $\bm{w}$, is a vector consisting of $(n + 1)$
    vector fields in $\RR^m$. Its $\alpha$-th component is:
    \begin{equation} w^{\alpha} = \sum_{\beta = 1}^{n + 1} \Omega^{\alpha \beta} (u^{\beta}
       -c^{\beta}) \in \RR^m . 
    \end{equation}
   The operator $\divg$ acts on this vector
    $\bm{w}$ component-wise. That is, the $\alpha$-th
    component of the RHS of \eqref{eq:helein_div} is $\mathrm{div}
    (w^{\alpha})$, where $\divr$ is the standard divergence in $\RR^m$. This
    gives the expression in \eqref{eq:CH1.9}.
    For any generic vector $\bm{v}$ consisting of
    $(n + 1)$ vector fields in $\RR^m$, the following Leibniz rule holds:
    \begin{equation}
      \divg (\jaOm \bm{v}) = (\divg  \jaOm)
      \bm{v}+ \jaOm \cdot \nabla
      \bm{v} \label{eq:Leibniz}
    \end{equation}
    where $\divg  \jaOm$ is taken entrywise and the dot product denotes the
    contraction of the second index $\beta$, that is, $\left( \jaOm \cdot
    \nabla \bm{v} \right)^{\alpha} := \sum_{\beta = 1}^{n
    + 1} \Omega^{\alpha \beta} \cdot \nabla v^{\beta}$. Recall that $\jaOm$ is
    an $(n + 1) \times (n + 1)$ matrix of $\RR^m$-valued vector fields
    (entries in $\RR^m$) and the divergence operator $\divg$ gives an $(n + 1)
    \times (n + 1)$ matrix whose entries are scalars, the divergence of the
    $\RR^m$-valued vector fields. Applying $\jaOm$ to a vector
    $\bm{u} \in \RR^{n + 1}$ yields $\RR^m$-valued vector
    fields, in number of $(n + 1)$. Again, $\divg (\jaOm
    \bm{v})$ is the vector obtained by applying the
    $m$-dimensional divergence operator to these $(n + 1)$ vector fields.
    Therefore, the result will be an element of $\RR^{(n + 1)}$.
\end{remark}

\begin{proof}
  We break the proof into three easily verifiable steps.
  
  {\noindent}{{\em Step 1: Rewriting the non-linearity.}} Starting from the
  almost harmonic map equation $- \Delta u^{\alpha} = u^{\alpha} | \nabla
  \bm{u}|^2 + f^{\alpha}$ we compute the nonlinear term
  using the orthogonality $(\nabla \bm{u})
  \bm{u}=\ensuremath{\boldsymbol{0}}$:
  \begin{equation} u^{\alpha} | \nabla \bm{u}|^2 = \sum_{\beta = 1}^{n
     + 1} \sum_{k = 1}^m (u^{\alpha} \partial_k u^{\beta} - u^{\beta}
     \partial_k u^{\alpha}) \partial_k u^{\beta} = \sum_{\beta = 1}^{n + 1}
     \Omega^{\alpha \beta} \cdot \nabla u^{\beta} = \left( \jaOm \cdot \nabla
     \bm{u} \right)^{\alpha} \end{equation}
  Hence, the original almost-harmonic map can be rewritten as
  \begin{equation}
    - \Delta \bm{u}= \jaOm \cdot \nabla
    \bm{u}+\bm{f} \label{eq:step1}
  \end{equation}
  {{\em Step 2: Applying the Leibniz rule.}} Applying the Leibniz rule
  \eqref{eq:Leibniz} with
  $\bm{v}=\bm{u}- \jc$ gives:
  \begin{equation} \jaOm \cdot \nabla \bm{u} \eqs \jaOm \cdot \nabla
     \left( \bm{u}- \jc \right) \eqs \divg \left( \jaOm
     \left( \bm{u}- \jc \right) \right) - (\divg  \jaOm)
     \left( \bm{u}- \jc \right) \end{equation}
  Substituting this back into \eqref{eq:step1} yields:
  \begin{equation}
    - \Delta \bm{u}= \divg \left( \jaOm \left(
    \bm{u}- \jc \right) \right) - (\divg  \jaOm) \left(
    \bm{u}- \jc \right) +\bm{f}.
    \label{eq:step2}
  \end{equation}
  {{\em Step 3: Evaluating the divergence of $\jaOm$.}} For an {{\em
  almost}} harmonic map $\bm{u}$, which satisfies
  \eqref{eq:step1}, i.e., $\Delta u^{\alpha} = - u^{\alpha} | \nabla
  \bm{u}|^2 - f^{\alpha}$, the divergence of
  $\Omega^{\alpha \beta}$ is
  \begin{eqnarray}
    \ensuremath{\operatorname{div}} (\Omega^{\alpha \beta}) & \eqs & (\nabla
    u^{\alpha} \cdot \nabla u^{\beta} - \nabla u^{\beta} \cdot \nabla
    u^{\alpha}) + (u^{\alpha} \Delta u^{\beta} - u^{\beta} \Delta u^{\alpha})
    \nonumber\\
    & \eqs & u^{\alpha} \Delta u^{\beta} - u^{\beta} \Delta u^{\alpha}
    \nonumber\\
    & \overset{\text{\eqref{eq:step1}}}{\eqs} & - \left| \grad
    \bm{u} \right|^2 u^{\alpha} u^{\beta} + \left| \grad
    \bm{u} \right|^2 u^{\alpha} u^{\beta} - u^{\alpha}
    f^{\beta} + u^{\beta} f^{\alpha} \nonumber\\
    & = & f^{\alpha} u^{\beta} - f^{\beta} u^{\alpha} \nonumber\\
    & = & (\bm{f} \wedge
    \bm{u})^{\alpha \beta}, \nonumber
  \end{eqnarray}
  with $\bm{f} \wedge
  \bm{u}=\bm{f} \otimes
  \bm{u}-\bm{u} \otimes
  \bm{f}$. Thus $\divg  \jaOm
  =\bm{f} \wedge \bm{u}$, which is
  the claimed identity.
\end{proof}

\subsection{Harmonic Approximation and Caccioppoli-type estimates}
From \eqref{eq:CH1.9} we know that if $\bm{u}$ is an
\text{{\itshape{almost harmonic map}}} (i.e., satisfying $- \Delta
\bm{u}=\bm{u}| \nabla
\bm{u}|^2 +\bm{f}$), then for every
constant vector $\jc \in \RR^{n + 1}$, \ the following identity holds in the
ball $B_{\rr}$:
\begin{equation}
  - \Delta \bm{u}= \divg (\jaOm
  (\bm{u}- \jc)) + (\bm{u} \wedge
  \bm{f}) \left( \bm{u}- \jc \right)
  +\bm{f} \label{eq:CH1.11}
\end{equation}
The following result plays a central role in our argument. It establishes a
control of the $L^{\jqq}$-norm of the gradient deviation in terms of the
$L^{\pHold}$-norm of the map itself. This ``reverse'' nature (bounding
derivatives by values) is typical of elliptic systems with critical growth.

\begin{lemma} [Caccioppoli-type
  estimate] \label{lemma:reverse_sobolev} Let $\jqq
  \in (1, 2)$ be the fixed integrability exponent of the source term
  $\bm{f}$, and let $\pHold \in (2, \infty)$ be the
  corresponding oscillation exponent determined by the relation
  \begin{equation}
    \frac{1}{\jqq} = \frac{1}{2} + \frac{1}{\pHold}, \quad \text{that is,}
    \quad \text{} \pHold = \frac{2 \jqq}{2 - \jqq} .
  \end{equation}
  Let $\bm{u} \in H^1 (B_1, \Stwo^n)$ be a solution to
  the almost harmonic map equation:
  \begin{equation}
    - \Delta \bm{u}= \divg \left( \jaOm
    (\bm{u}- \jc) \right) + (\bm{u}
    \wedge \bm{f}) \left( \bm{u}-
    \jc \right) +\bm{f}. \label{eq:newAHME}
  \end{equation}
  Let $\rr \in (0, 1)$ be a fixed radius. Let
  $\ensuremath{\boldsymbol{h}}_{\rr} \in H^1 (B_{\rr}, \RR^{n + 1})$ denote
  the unique harmonic extension of $\bm{u}_{\restr
  \partial B_{\rr}}$ to the ball $B_{\rr}$:
  \begin{equation}
    - \Delta \ensuremath{\boldsymbol{h}}_{\rr} =\ensuremath{\boldsymbol{0}}
    \hspace{1.2em} \text{ in } B_{\rr}, \qquad
    \ensuremath{\boldsymbol{h}}_{\rr} =\bm{u}_{\restr
    \partial B_{\rr}} \hspace{1.2em} \text{ on } \partial B_{\rr} .
  \end{equation}
  With these exponents, the following estimate holds:
  \begin{equation}
    \label{eq:reverse_sobolev_bound} \| \grad
    (\bm{u}-\ensuremath{\boldsymbol{h}}_{\rr})\|_{L^{\jqq}
    (B_{\rr})} \leqslant C_{\jqq} \cdot \| \nabla
    \bm{u}\|_{L^2 (B_{\rr})}  \left( \rr^m
    \hspace{0.17em} \Osc \left( \bm{u}, B_{\rr} \right)
    \right)^{1 / \pHold} + C_{\jqq}  \rr
    \|\bm{f}\|_{L^{\jqq} (B_{\rr})},
  \end{equation}
  where $C_{\jqq}$ depends only on $\jqq$ {\opt}and, in general, on the base
  dimension $m$; here $m = 2${\cpt}. In particular $C_{\jqq}$ is independent
  of $\rr$ and $\bm{u}$.
\end{lemma}

\begin{proof}
  Put $\bm{w}_{\rr} :=
  \bm{u}-\ensuremath{\boldsymbol{h}}_{\rr}$. Since $-
  \Delta \ensuremath{\boldsymbol{h}}_{\rr} = 0$, the function
  $\bm{w}_{\rr}$ satisfies the same Poisson equation as
  $\bm{u}$ but with zero boundary data
  ({cf.}~\eqref{eq:CH1.11}):
  \begin{equation}
    \left\{\begin{array}{ll}
      - \Delta \bm{w}_{\rr} = \divg \left( \jaOm
      (\bm{u}- \jc) \right) +
      (\bm{u} \wedge \bm{f}) \left(
      \bm{u}- \jc \right) +\bm{f} &
      \text{in } B_{\rr},\\
      \bm{w}_{\rr} =\ensuremath{\boldsymbol{0}} &
      \text{on } \partial B_{\rr},
    \end{array}\right.
  \end{equation}
  for any fixed $\jc \in \RR^{n + 1}$. Here, we recall, $\jaOm = \jaOm
  (\bm{u}) := \bm{u} \wedge \grad
  \bm{u}$ is the antisymmetric potential defined in
  \eqref{eq:defantsympot}.
  
  If we set $\bm{g}_1 := \divg \left( \jaOm
  (\bm{u}- \jc) \right)$, $\bm{g}_2
  := (\bm{u} \wedge \bm{f}) \left(
  \bm{u}- \jc \right)$, and
  $\bm{g}_3 := \bm{f}$, by the
  linearity of the Laplacian operator, we can decompose the error into three
  parts, $\bm{w}=\bm{w}_1
  +\bm{w}_2 +\bm{w}_3$, with $-
  \lapl \bm{w}_i =\bm{g}_i$ in
  $B_{\rr}$ and $\bm{w}_i =\ensuremath{\boldsymbol{0}}$
  on $\partial B_{\rr}$.
  
  {\noindent}{{\em Step 1: Elliptic Estimates}}. By standard elliptic
  estimates scaled to the ball $B_{\rr}$, for every $1 < \jqq < 2$, we have
  \begin{equation}
    \| \nabla \bm{w}_2 \|_{L^{\jqq} (B_{\rr})} \leqslant
    C_{\jqq} \rr \| (\bm{u} \wedge
    \bm{f}) \left( \bm{u}- \jc
    \right) \|_{L^{\jqq} (B_{\rr})}, \quad \| \nabla
    \bm{w}_3 \|_{L^{\jqq} (B_{\rr})} \leqslant C_{\jqq}
    \rr \|\bm{f}\|_{L^{\jqq} (B_{\rr})} .
  \end{equation}
  Because $\bm{u}$ takes values in the unit sphere,
  $|\bm{u}| \equiv 1$, we can safely restrict our
  attention to constant vectors $\jc$ belonging to the closed unit ball in
  $\RR^{n + 1}$ (see Lemma~\ref{lemma:osc_infimum_bound}). Therefore, we may
  assume $\left| \jc \right| \leqslant 1$. Thus, we have the universal
  pointwise bound $|\bm{u}- \jc | \leqslant 2$.
  Furthermore, the operator norm of the wedge product gives $|
  (\bm{u} \wedge \bm{f})
  (\bm{u}- \jc) | \leqslant |\bm{f}|
  |\bm{u}- \jc | \leqslant 2
  |\bm{f}|$. Therefore, the source terms are trivially
  bounded directly in $L^{\jqq}$:
  \begin{equation}
    \| \nabla \bm{w}_2 \|_{L^{\jqq} (B_{\rr})} + \|
    \nabla \bm{w}_3 \|_{L^{\jqq} (B_{\rr})} \leqslant 3
    C_{\jqq} \rr \|\bm{f}\|_{L^{\jqq} (B_{\rr})} .
    \label{eq:tempfirstestforHolder0}
  \end{equation}
  It remains to bound $\| \nabla \bm{w}_1 \|_{L^{\jqq}
  (B_{\rr})}$. For that, observe that the term $\jaOm
  (\bm{u}- \jc)$ belongs to $L^2 \left( B_{\rr}, \RR^{n + 1}
  \otimes \RR^m \right)$. In particular, for every $1 < \jqq < 2$, applying
  Calderón--Zygmund estimates~\cite[Theorems 7.1 and 7.2]{Giaquinta_2012} we get that
  \begin{equation}
    \| \nabla \bm{w}_1 \|_{L^{\jqq} (B_{\rr})} \leqslant
    C_{\jqq} \| \jaOm (\bm{u}- \jc)\|_{L^{\jqq}
    (B_{\rr})}, \label{eq:tempfirstestforHolder}
  \end{equation}
  valid for the stated range of $\jqq$. The constant $C_{\jqq}$ depends only
  on the base dimension $m$, the exponent $\jqq$, and on the shape of the
  domain; because the domain is a ball the estimate is scale invariant, so
  $C_{\jqq}$ is independent of $\rr, \bm{u},
  \jc$.{\smallskip}
  
  {\noindent}{{\em Step 2: Hölder's Inequality.}} We estimate the terms in
  $L^{\jqq}$ using Hölder's inequality. By the definition of the matrix-vector
  action, $| \jaOm (\bm{u}- \jc) | \leqslant | \jaOm | 
  |\bm{u}- \jc |$. Since $\jaOm,
  \bm{f} \in L^2 \left( B_{\rr} \right)$, for $\pHold >
  1$ such that $1 / \jqq = 1 / 2 + 1 / \pHold$, Hölder inequality gives
  \begin{equation}
    \| \jaOm (\bm{u}- \jc)\|_{L^{\jqq} (B_{\rr})}
    \leqslant \| \jaOm \|_{L^2 (B_{\rr})}  \hspace{0.17em}
    \|\bm{u}- \jc \|_{L^{\pHold} (B_{\rr})} .
    \label{eq:tempfirstestforHolder2}
  \end{equation}
  {\noindent}{{\em Step 3: Conclusion.}} Overall, by
  \eqref{eq:tempfirstestforHolder0} and \eqref{eq:tempfirstestforHolder2} we
  get that
  \begin{eqnarray}
    \| \grad \left(
    \bm{u}-\ensuremath{\boldsymbol{h}}_{\rr} \right)
    \|_{L^{\jqq} (B_{\rr})} & \leqslant & \| \nabla
    \bm{w}_1 \|_{L^{\jqq} (B_{\rr})} + \| \nabla
    \bm{w}_2 \|_{L^{\jqq} (B_{\rr})} + \| \nabla
    \bm{w}_3 \|_{L^{\jqq} (B_{\rr})} \nonumber\\
    & \leqslant & C_{\jqq} \| \jaOm \|_{L^2 (B_{\rr})} \cdot
    \|\bm{u}- \jc \|_{L^{\pHold} (B_{\rr})} + 3 C_{\jqq}
    \rr \|\bm{f}\|_{L^{\jqq} (B_{\rr})} . \nonumber
  \end{eqnarray}
  Recalling that $\jaOm =\bm{u} \wedge \nabla
  \bm{u}$ and $|\bm{u}| \equiv 1$,
  as shown in \eqref{eq:normasympot}, we have the pointwise estimate $| \jaOm
  | \leqslant \sqrt{2} | \nabla \bm{u}|$. Hence
  \begin{equation} \| \grad \left(
     \bm{u}-\ensuremath{\boldsymbol{h}}_{\rr} \right)
     \|_{L^{\jqq} (B_{\rr})} \leqslant C_{\jqq} \sqrt{2} \| \nabla
     \bm{u}\|_{L^2 (B_{\rr})} \cdot
     \|\bm{u}- \jc \|_{L^{\pHold} (B_{\rr})} + 3 C_{\jqq}
     \rr \|\bm{f}\|_{L^{\jqq} (B_{\rr})} . \end{equation}
  The previous estimate holds for every $\left| \jc \right| \leqslant 1$.
  Taking the infimum over $\left| \jc \right| \leqslant 1$ yields (see
  Lemma~\ref{lemma:osc_infimum_bound})
  \begin{equation}
    \| \nabla
    (\bm{u}-\ensuremath{\boldsymbol{h}}_{\rr})\|_{L^{\jqq}
    (B_{\rr})} \leqslant C_{\jqq} \cdot \| \nabla
    \bm{u}\|_{L^2 (B_{\rr})}  \left( \rr^m
    \hspace{0.17em} \Osc \left( \bm{u}, B_{\rr} \right)
    \right)^{1 / \pHold} + C_{\jqq} \rr
    \|\bm{f}\|_{L^{\jqq} (B_{\rr})},
  \end{equation}
  with the harmless abuse of notation of redefining $C_{\jqq} := \max \left(
  \sqrt{2} \omega_m^{1 / \pHold}, 3 C_{\jqq} \right)$. This concludes the
  proof.
\end{proof}

The preceding PDE estimate and the following purely real-analytic bound for
harmonic extensions will play a key role in our analysis. We now establish
this harmonic approximation, which relies  on the properties of the
Poisson kernel and requires no PDE hypotheses.

{\myprop{\begin{lemma}[Harmonic approximation]
  \label{lemma:harmonic_approx}Let $\bm{u} \in H^1 (B_1,
  \RR^{n + 1})$. For any $\pHold > 1$, there exists a radius $\rr \in [1 / 2,
  1]$ and a harmonic function $\ensuremath{\boldsymbol{h}}: B_{\rr} \to \RR^{n
  + 1}$ satisfying the boundary trace
  $\ensuremath{\boldsymbol{h}}=\bm{u}$ on $\partial
  B_{\rr}$, such that the interior gradient satisfies:
  \begin{equation}
    \sup_{x \in B_{1 / 4}} | \nabla \ensuremath{\boldsymbol{h}}(x) | \leqslant
    C_m \cdot \Osc^{1 / \pHold} (\bm{u}, B_1),
    \label{eq:harmapproxosc}
  \end{equation}
  where $C_m > 0$ is a dimensional constant.
\end{lemma}}}

\begin{proof}
  {\noindent}{{\em Step 1: Selection of a good radius.}} By the coercivity and strict convexity of $\jc \mapsto \int_{B} |\bm{u}-\jc|^\pHold$ for $\pHold>1$, the infimum in the definition of the
  oscillation is attained. Let $\jc \in \RR^{n + 1}$ be a constant vector that
  realizes this infimum in $\Osc (\bm{u}, B_1)$
  \begin{equation}
    \jav{B_1} |\bm{u}- \jc |^{\pHold} = \Osc
    (\bm{u}, B_1) .
  \end{equation}
  By a standard consequence of Fubini's theorem (or the mean value theorem
  applied to radial slices), there exists a radius $\rr \in [1 / 2, 1]$, which
  depends on the function $\bm{u}- \jc$, such that the
  boundary integral is controlled by the bulk integral:
  \begin{equation}
    \int_{\zeta \in \partial B_{\rr}} \left| \bm{u}
    (\zeta) - \jc \right| \mathrm{d} \zeta \leqslant 2 \int_{B_1} \left|
    \bm{u} (x) - \jc \right| \mathrm{d} x.
    \label{eq:boundtodom}
  \end{equation}
  Indeed, integrating in polar coordinates we have
  \begin{equation}
    \int_{1 / 2}^1 \left( \int_{\partial B_{\rr}}
    |\bm{u}(\zeta) - \jc | \hspace{0.17em} \mathrm{d}
    \zeta \right) \mathrm{d} \rr = \int_{B_1 \setminus B_{1 / 2}}
    |\bm{u}- \jc | \hspace{0.17em} \mathrm{d} x \leqslant
    \int_{B_1} |\bm{u}- \jc | \hspace{0.17em} \mathrm{d}
    x,
  \end{equation}
  so that the mean value theorem yields a radius $\rr \in [1 / 2, 1]$ for
  which \eqref{eq:boundtodom} holds.
  
  Note that $\rr$ depends on the optimal $\jc$, and that the optimal $\jc$
  depends only on $\bm{u}$ and $\pHold$.
  
  {\noindent}{{\em Step 2: Harmonic extension and interior estimates.}} We
  fix the radius $\rr$ found in Step~1. Let $\ensuremath{\boldsymbol{h}}$ be
  the unique harmonic function in $B_{\rr}$ satisfying the boundary condition
  $\ensuremath{\boldsymbol{h}}=\bm{u}$ on $\partial
  B_{\rr}$. Since $\jc$ is a constant, the function
  $\ensuremath{\boldsymbol{h}}- \jc$ is also harmonic in $B_{\rr}$. We
  estimate $\ensuremath{\boldsymbol{h}}- \jc$ inside the ball using its
  Poisson Integral Formula (see Proposition~\ref{prop:Poissongradest} in the
  Appendix in Section \ref{sec:app1}). By standard estimates on the gradient
  of the Poisson kernel $P_r (x, \zeta)$, for any $x \in B_{\rr / 2}$, we have
  (since $\ensuremath{\boldsymbol{h}}=\bm{u}$ on the
  boundary $\partial B_{\rr}$):
  \begin{equation}
    | \nabla \ensuremath{\boldsymbol{h}}(x) | = | \nabla
    (\ensuremath{\boldsymbol{h}}(x) - \jc) | \leqslant \frac{\tilde{C}_m
    }{\rr^m}  \int_{\zeta \in \partial B_{\rr}}
    |\ensuremath{\boldsymbol{h}}(\zeta) - \jc |  \hspace{0.17em} \mathrm{d}
    \zeta = \frac{\tilde{C}_m }{\rr^m}  \int_{\zeta \in \partial B_{\rr}}
    \left| \bm{u} (\zeta) - \jc \right| \mathrm{d} \zeta
    . \label{eq:Poissonbound}
  \end{equation}
  We now restrict the domain. Since $\rr \geqslant 1 / 2$, we have $B_{1 / 4}
  \subset B_{\rr / 2}$ and $1 / \rr^m \leqslant 2^m$. Combining this with the
  radius estimate \eqref{eq:boundtodom} we get:
  \begin{eqnarray}
    \sup_{x \in B_{1 / 4}} | \nabla \ensuremath{\boldsymbol{h}}(x) | &
    \leqslant & \sup_{x \in B_{\rr / 2}} | \nabla
    \ensuremath{\boldsymbol{h}}(x) | \;
    \overset{\eqref{eq:Poissonbound}}{\leqslant} \; \frac{\tilde{C}_m }{\rr^m}
    \int_{\zeta \in \partial B_{\rr}} \left| \bm{u}
    (\zeta) - \jc \right| \mathrm{d} \zeta \nonumber\\
    & \overset{\eqref{eq:boundtodom}}{\leqslant} & \tilde{C}_m 2^m  \left( 2
    \int_{B_1} |\bm{u}- \jc | \right) \leqslant 2^{m + 1}
    \tilde{C}_m \omega_m  \left( \jav{B_1} \left| \bm{u}-
    \jc \right|^{\pHold} \right)^{1 / \pHold}, 
  \end{eqnarray}
  where $\omega_m$ is the volume of the unit ball in $\RR^m$. Because $\jc$
  was specifically chosen to realize the minimal oscillation, the right-hand
  side is  proportional to $\Osc^{1 / \pHold}
  (\bm{u}, B_1)$. Setting $C_m := 2^{m + 1}  \tilde{C}_m
  \omega_m$ yields the gradient estimate \eqref{eq:harmapproxosc}.{\smallskip}
\end{proof}

Having established the existence of a Caccioppoli-type estimate
(Lemma~\ref{lemma:reverse_sobolev}) and a controlled harmonic approximation
(Lemma~\ref{lemma:harmonic_approx}), we can now combine these results to
obtain the following result which will be the main tool in the derivation of
the decay estimate for almost harmonic maps.

{\myprop{\begin{lemma}[Coupled Caccioppoli-type estimate]
  \label{lemma:reverse_sobolev_bound}Let $\bm{u} \in H^1
  (B_1, \Stwo^n)$ be a solution to the almost harmonic map equation:
  \begin{equation}
    - \Delta \bm{u}= \divg \left( \jaOm
    (\bm{u}- \jc) \right) + (\bm{u}
    \wedge \bm{f}) \left( \bm{u}-
    \jc \right) +\bm{f},
  \end{equation}
  with source term $\bm{f} \in L^{\jqq} (B_1, \RR^{n +
  1})$ having integrability exponent $\jqq \in (1, 2)$. Let
  $\pHold := 2 \jqq / \left( 2 - \jqq \right) \in (2,
  \infty)$ be the corresponding oscillation exponent.
  
  Then, there exists a harmonic function $\ensuremath{\boldsymbol{h}}: B_{1 /
  4} \to \RR^{n + 1}$ such that the following gradient estimate holds on $B_{1
  / 4}$:
  \begin{equation}
    \| \nabla
    (\bm{u}-\ensuremath{\boldsymbol{h}})\|_{L^{\jqq}
    (B_{1 / 4})} \leqslant C_{\jqq}  \| \nabla
    \bm{u}\|_{L^2 (B_1)} \Osc^{1 / \pHold}
    (\bm{u}, B_1) + C_{\jqq}
    \|\bm{f}\|_{L^{\jqq} (B_1)},
    \label{eq:reverse_sobolev_bound2}
  \end{equation}
  where $C_{\jqq} > 0$ depends only on the dimension $m = 2$ and the exponent
  $\jqq$. Moreover,
  \begin{equation}
    \sup_{x \in B_{1 / 4}} | \nabla \ensuremath{\boldsymbol{h}}(x) | \leqslant
    C_m \cdot \Osc^{1 / \pHold} (\bm{u}, B_1),
    \label{eq:harmapproxosc2}
  \end{equation}
  where $C_m > 0$ is a dimensional constant.
\end{lemma}}}

\begin{proof}
  By Lemma \ref{lemma:harmonic_approx}, there exists a good radius $\rr \in [1
  / 2, 1]$ and a harmonic approximation $\ensuremath{\boldsymbol{h}}$ of the
  almost harmonic map $\bm{u}$ on $B_{\rr}$, such that
  \eqref{eq:harmapproxosc2} holds. Because $\ensuremath{\boldsymbol{h}}$ is a
  solution of the problem
  \begin{equation}
    - \Delta \ensuremath{\boldsymbol{h}}=\ensuremath{\boldsymbol{0}}
    \hspace{1.2em} \text{ in } B_{\rr}, \qquad
    \ensuremath{\boldsymbol{h}}=\bm{u} \hspace{1.2em}
    \text{ on } \partial B_{\rr},
  \end{equation}
  we can apply Lemma~\ref{lemma:reverse_sobolev} over the ball $B_{\rr}$ to
  get (recall that $B_{1 / 4} \subset B_{\rr}$) from
  \eqref{eq:reverse_sobolev_bound}
  \begin{equation}
    \| \nabla
    (\bm{u}-\ensuremath{\boldsymbol{h}})\|_{L^{\jqq}
    (B_{1 / 4})} \leqslant C_{\jqq} \cdot \| \nabla
    \bm{u}\|_{L^2 (B_{\rr})} \cdot \left( \rr^m
    \hspace{0.17em} \Osc \left( \bm{u}, B_{\rr} \right)
    \right)^{1 / \pHold} + C_{\jqq} \rr
    \|\bm{f}\|_{L^{\jqq} (B_1)} .
  \end{equation}
  Using the fact that $\rr \leqslant 1$, extending the domains of integration
  to $B_1$, and by the monotonicity of the oscillation functional $\rr^m \Osc
  \left( \bm{u}, B_{\rr} \right) \leqslant \Osc
  (\bm{u}, B_1)$, we obtain
  \eqref{eq:reverse_sobolev_bound2}. Overall, the restriction of
  $\ensuremath{\boldsymbol{h}}$ to $B_{1 / 4}$ gives
  \eqref{eq:reverse_sobolev_bound2} and \eqref{eq:harmapproxosc2}. 
\end{proof}

\subsection{Conclusion of the Proof of the Fundamental
{Lemma~\ref{lem:decay}}: The Decay Estimate}
We are now in the position to show that for the specific exponent $\pHold = 2
\jqq / \left( 2 - \jqq \right)$, arising from the Hölder inequality, the decay
estimate ({\jax}~\ref{ax:decay}) holds. Specifically, we want to prove that if
the gradient energy is small enough, then for a fixed small radius $\jth \in
(0, 1 / 4)$:
\begin{equation}
  \Osc (\bm{u}, B_{\jth}) \leqslant \jgam \Osc
  (\bm{u}, B_1) + \kappa \| \bm{f}
  \|_{L^{\jqq} (B_1)}^{\pHold}, \quad \jgam := \frac{1}{2}
\end{equation}
where $\kappa$ is a constant depending on $\jth, \pHold$, and the dimensions,
but independent of $\bm{u}$. Explicitly, we want to prove
that if the energy is small enough, then for a fixed small radius $\jth \in
(0, 1 / 4)$:
\begin{equation} \inf_{\jc \in \RR^{n + 1}} \left( \jav{B_{\jth}}
   |\bm{u}- \jc |^{\pHold} \right) \leqslant \frac{1}{2}
   \inf_{\jc \in \RR^{n + 1}} \left( \jav{B_1} |\bm{u}-
   \jc |^{\pHold} \right) + \kappa \left( \int_{B_1} |
   \bm{f} |^{\jqq} \right)^{\pHold / \jqq} . \end{equation}
{\noindent}\text{{\bfseries{Proof.}}} By
Lemma~\ref{lemma:reverse_sobolev_bound}, there exists a harmonic function
$\ensuremath{\boldsymbol{h}}: B_{1 / 4} \to \RR^{n + 1}$ satisfying the
gradient estimates \eqref{eq:reverse_sobolev_bound2} and
\eqref{eq:harmapproxosc2}. Let $\jth \in (0, 1 / 4)$. We estimate the
oscillation of $\bm{u}$ on $B_{\jth}$ by comparing
$\bm{u}$ with $\ensuremath{\boldsymbol{h}}$. Using the
triangle inequality for the $L^{\pHold}$-oscillation seminorm, we split the
error:
\begin{equation}
  \Osc^{1 / \pHold} (\bm{u}, B_{\jth}) \leqslant \Osc^{1
  / \pHold} (\bm{u}-\ensuremath{\boldsymbol{h}},
  B_{\jth}) + \Osc^{1 / \pHold} (\ensuremath{\boldsymbol{h}}, B_{\jth}) = : I
  + I  I. \label{eq:splitPsi}
\end{equation}
{\noindent}{{\em Step 1: Estimate of term $I$.}} Recall that
$\bm{u}-\ensuremath{\boldsymbol{h}}$ was constructed to
vanish on the larger boundary $\partial B_{\rr}$ \ for some $\rr \geqslant 1 /
2$. Consequently, it does not necessarily vanish on the much smaller internal
boundary $\partial B_{\jth}$. Because
$\bm{u}-\ensuremath{\boldsymbol{h}} \notin W^{1, \jqq}_0
\left( B_{\jth} \right)$, we cannot rely on the zero-trace Sobolev embedding.
Instead, we apply the standard Poincaré-Sobolev inequality on the ball
$B_{\jth}$:
\begin{align}
  I := \Osc^{1 / \pHold}
  (\bm{u}-\ensuremath{\boldsymbol{h}}, B_{\jth}) &
  \leqslant  \left( \jav{B_{\jth}} |
  (\bm{u}-\ensuremath{\boldsymbol{h}}) - \langle
  \bm{u}-\ensuremath{\boldsymbol{h}} \rangle_{B_{\jth}}
  |^{\pHold} \right)^{1 / \pHold} \\
  & \eqs  \frac{1}{\left( \omega_m  \jth^m \right)^{1 / \pHold}}  \left\|
  (\bm{u}-\ensuremath{\boldsymbol{h}}) - \langle
  \bm{u}-\ensuremath{\boldsymbol{h}} \rangle_{B_{\jth}}
  \right\|_{L^{\pSob} \left( B_{\jth} \right)} \nonumber\\
  &   \qquad \qquad \qquad \qquad \qquad \leqslant \frac{C_S}{\left(
  \omega_m  \jth^m \right)^{1 / \pHold}} \| \nabla
  (\bm{u}-\ensuremath{\boldsymbol{h}})\|_{L^{\jqq}
  (B_{\jth})} . 
\end{align}
Now we substitute the Caccioppoli-type estimate
\eqref{eq:reverse_sobolev_bound2} in the previous expression: Since $B_{\jth}
\subset B_{1 / 4}$, the $L^{\jqq}$ norm over $B_{\jth}$ is bounded by the
$L^{\jqq}$ norm over $B_{1 / 4}$:
\begin{equation}
  \| \nabla
  (\bm{u}-\ensuremath{\boldsymbol{h}})\|_{L^{\jqq}
  (B_{\jth})} \leqslant \| \nabla
  (\bm{u}-\ensuremath{\boldsymbol{h}})\|_{L^{\jqq} (B_{1
  / 4})} \overset{\eqref{eq:reverse_sobolev_bound2}}{\leqslant} C_{\jqq} \cdot
  \hspace{0.17em} \| \nabla \bm{u}\|_{L^2 (B_1)}  \Osc^{1
  / \pHold} (\bm{u}, B_1) + C_{\jqq}
  \|\bm{f}\|_{L^{\jqq} (B_1)} .
\end{equation}
Substituting this back into $I$ we get:
\begin{equation}
  \label{eq:term1_bound} I := \Osc^{1 / \pHold}
  (\bm{u}-\ensuremath{\boldsymbol{h}}, B_{\jth})
  \leqslant \frac{C_{\jqq}}{\jth^{m / \pHold}} \cdot \| \nabla
  \bm{u}\|_{L^2 (B_1)} \cdot \Osc^{1 / \pHold}
  (\bm{u}, B_1) + \frac{C_{\jqq}}{\jth^{m / p}}
  \|\bm{f}\|_{L^{\jqq} (B_1)} .
\end{equation}
Here, by an harmless abuse of notation, the constant $C_{\jqq}$ has been
redefined to absorb the Poincaré-Sobolev constant $C_S$ and geometric
constants. It remains to estimate $I  I := \Osc^{1 / \pHold}
(\ensuremath{\boldsymbol{h}}, B_{\jth})$.

{\noindent}{{\em Step 2: Estimate of Term $I  I$.}} Since the oscillation is
an infimum over all constants, we can bound it from above by choosing the
specific constant $\jc =\ensuremath{\boldsymbol{h}} (0)$. Also, since
$\ensuremath{\boldsymbol{h}}$ is harmonic and smooth, we apply the mean value
theorem inside $B_{1 / 4}$ to infer:
\begin{eqnarray}
  I  I := \Osc^{1 / \pHold} (\ensuremath{\boldsymbol{h}}, B_{\jth}) &
  \leqslant & \left( \jav{B_{\jth}} |\ensuremath{\boldsymbol{h}}(x)
  -\ensuremath{\boldsymbol{h}}(0) |^{\pHold}  \hspace{0.17em} \mathrm{d} x
  \right)^{1 / \pHold} \; \leqslant \; \left( \jav{B_{\jth}} |x|^{\pHold}
  \left( \sup_{y \in B_{\jth}} | \nabla \ensuremath{\boldsymbol{h}}(y) |^p
  \right)  \hspace{0.17em} \mathrm{d} x \right)^{1 / \pHold} \nonumber\\
  &  & \qquad \qquad \qquad \qquad \qquad \qquad \qquad \qquad \qquad
  \leqslant \jth \sup_{x \in B_{1 / 4}} | \nabla \ensuremath{\boldsymbol{h}}
  (x) | . 
\end{eqnarray}
Applying the gradient bound \eqref{eq:harmapproxosc2}, we conclude that
\begin{equation}
  \label{eq:term2_bound} I  I \leqslant C_m  \jth  \Osc^{1 / \pHold}
  (\bm{u}, B_1) .
\end{equation}
{\noindent}{{\em Step 3: Collecting the estimates of $I$ and $I  I$.}}
Combining the estimates \eqref{eq:term1_bound} and \eqref{eq:term2_bound} for
$I$ and $I  I$, and factoring out $\Osc^{1 / \pHold}
(\bm{u}, B_1)$, with the assignment $C_{m, \jqq} := \max
\left( C_m, C_{\jqq} \right)$, we obtain:
\begin{equation}
  \Osc^{1 / \pHold} (\bm{u}, B_{\jth}) \leqslant C_{m,
  \jqq} \left( \frac{\| \nabla \bm{u}\|_{L^2
  (B_1)}}{\jth^{m / p}} + \jth \right)  \Osc^{1 / \pHold}
  (\bm{u}, B_1) + \frac{C_{m, \jqq}}{\jth^{m / p}}
  \|\bm{f}\|_{L^{\jqq} (B_1)} .
\end{equation}
Raising both sides to the power $\pHold$ (using the standard algebraic
inequality $(a + b)^{\pHold} \leqslant 2^{\pHold - 1}  (a^{\pHold} +
b^{\pHold})$), we rewrite this purely in terms of the Dirichlet energy
$\Energy (\bm{u}) := \| \nabla
\bm{u}\|_{L^2 (B_1)}^2$:
\begin{equation}
  \label{eq:decay_master} \Osc (\bm{u}, B_{\jth})
  \leqslant C_{m, \jqq}^{\pHold} \left( \frac{\Energy
  (\bm{u})^{\pHold / 2}}{\jth^m} + \jth^{\pHold} \right)
  \Osc (\bm{u}, B_1) + \frac{C_{m,
  \jqq}^{\pHold}}{\jth^m} \|\bm{f}\|_{L^{\jqq}
  (B_1)}^{\pHold},
\end{equation}
where, as usual, the constant $C_{m, \jqq}^{\pHold}$ has been harmlessly
redefined to absorb the $2^{\pHold - 1}$ factors that arise.{\smallskip}

{\noindent}{{\em Step 4: Choice of the parameters.}} To achieve the decay
factor $\jgam = 1 / 2$, we choose $\jth$ and $\varepsilon_{\ast}$
sequentially. First, choose the radius $\jth \in (0, 1 / 4)$ small enough such
that the harmonic term dominates:
\begin{equation}
  C_{m, \jqq}^{\pHold}  \jth^{\pHold} \leqslant \frac{1}{4} \quad
  \Longleftrightarrow \quad \jth \leqslant \left( \frac{1}{4} \right)^{1 /
  \pHold} \frac{1}{C_{m, \jqq}} .
\end{equation}
Second, with $\jth$ fixed, we choose the Dirichlet energy threshold
$\varepsilon_{\ast}$ small enough such that if $\Energy
(\bm{u}) < \varepsilon_{\ast}$, the nonlinear error term
is  controlled:
\begin{equation}
  C_{m, \jqq}^{\pHold}  \frac{\varepsilon_{\ast}^{\pHold / 2}}{\jth^m}
  \leqslant \frac{1}{4} .
\end{equation}
Notice that this smallness condition depends \text{{\itshape{only}}} on the
Dirichlet energy; no size restriction is placed on
$\bm{f}$. With these choices, the scaling bracket in
\eqref{eq:decay_master} \ sums to at most $1 / 4 + 1 / 4 = 1 / 2$. Defining
$\kappa (\jth) := C_{m, \jqq}^{\pHold} / \jth^m$, we arrive at the final
Campanato iteration step:
\begin{equation}
  \Osc (\bm{u}, B_{\jth}) \leqslant \frac{1}{2} \Osc
  (\bm{u}, B_1) + \kappa
  \|\bm{f}\|_{L^{\jqq} (B_1)}^{\pHold} .
\end{equation}
This establishes the exact decay estimate formulated in
{\jax}{\textbf{~\ref{ax:decay}}}, concluding the proof of
Lemma~\ref{lem:decay}.{\hspace*{\fill}}$\square$

\section{Generalization: Regularity of non-antisymmetric systems}\label{sec:sec4}In the classical regularity theory of harmonic maps, the assumption that the field $\bm{u}$ takes values in a closed target manifold $E \subset\RR^{n+1}$ naturally yields a connection matrix $\jaOm$  that is geometrically antisymmetric. However, a fundamental consequence of the classical elliptic approach developed in this paper is that this strict geometric requirement can be entirely dispensed with.

By replacing the geometric hypothesis of antisymmetry with an analytic integrability condition on the divergence, we seamlessly extend our Campanato and Caccioppoli-type framework to generic, non-antisymmetric systems. While the landmark results of Rivière \cite{Riviere2007} and Müller and Schikorra \cite{Mueller2009} fundamentally rely on the compensated compactness afforded by an antisymmetric $\jaOm$ , our method demonstrates that classical linear elliptic estimates are robust enough to handle general systems, provided the divergence is analytically controlled.

The primary objective of this section is to rigorously establish this generalization. We emphasize throughout that $\jaOm$  is \emph{not} assumed to be antisymmetric, distinguishing our setting from the frameworks of \cite{Riviere2007} and \cite{Mueller2009}:

{\noindent}\textbf{Theorem~\ref{thm:generalized_regularity}.}
(Regularity of generalized systems) {\itshape Let
$B_1 \subset \RR^2$. Let $\bm{u} \in H^1 \cap L^{\infty}
(B_1, \RR^{n + 1})$ be a weak solution to the generalized system
\begin{equation}
  - \Delta \bm{u}= \jaOm \cdot \nabla
  \bm{u}+\bm{f} \quad \text{in } B_1
  .
\end{equation}
Assume that the structural matrix and the source terms satisfy the purely
analytic hypotheses for some $\jqq > 1$:
\begin{itemize}
  \item $\jaOm \in L^2 (B_1, \RR^{(n + 1) \times (n + 1)} \otimes \RR^2)$,
  \item $\bm{f} \in L^{\jqq} (B_1, \RR^{n + 1})$,
  \item $\divg \jaOm \in L^{\jqq} (B_1, \RR^{(n + 1) \times (n + 1)})$.
\end{itemize}
Then, $\bm{u}$ is locally Hölder continuous in $B_1$.
Specifically, there exists an exponent $\jeta \in (0, 1)$ such that
$\bm{u} \in C^{0,
\jeta}_{\ensuremath{\operatorname{loc}}}  (B_1, \RR^{n +1})$.}

\begin{remark}
  In what follows, we may assume without loss of generality that $1 < \jqq <
  2$. The reason for this has already been explained in
  Remark~\ref{rmk:redsourceintq}.
\end{remark}

The proof of Theorem~\ref{thm:generalized_regularity} will be given in
subsection~\ref{sec:genregproof}. To this end, we first introduce a slight
generalization of the abstract regularity principle stated in
Theorem~\ref{thm:abstract_reg} by defining an abstract class of
configurations. We then formulate and prove the second abstract
regularity principle (Theorem~\ref{thm:gen_abstract_reg}) in arbitrary base
dimension $m \geqslant 2$, although the framework will ultimately be applied
to the regularity theory of generalized systems in the two-dimensional case $m
= 2$ to conclude the proof of Theorem~\ref{thm:generalized_regularity}.

\subsection{The second abstract regularity principle}
Let $\class$ be a class of triplets $(\bm{u}, \jaOm,
\bm{f})$ belonging to the space
\begin{equation}
  H^1 \cap L^{\infty} (B_1, \RR^{n + 1}) \times L^2 (B_1, \RR^{(n + 1) \times
  (n + 1)} \otimes \RR^m) \times L^{\jqq} (B_1, \RR^{n + 1}),
\end{equation}
which obey the following two axioms:

{\myprop{\begin{axiom}[Locality and Closure]
  \label{ax:closuregen}The class $\class$ is closed under rescaling.
  Specifically, if $(\bm{u}, \jaOm,
  \bm{f}) \in \class$, then for any ball $B_{\rr} (x_0)
  \subset B_1$, the rescaled triplet $(\bm{u}_{x_0, \rr},
  \overline{\jaOm}_{x_0, \rr}, \tilde{\bm{f}}_{x_0,
  \rr})$ defined by
  \begin{equation}
    \bm{u}_{x_0, \rr} (y) := \bm{u}
    (x_0 + \rr  y), \quad \overline{\jaOm}_{x_0, \rr} (y) := \rr \jaOm (x_0 +
    \rr  y), \quad \tilde{\bm{f}}_{x_0, \rr} (y) := \rr^2
    \bm{f} (x_0 + \rr  y),
  \end{equation}
  is still in $\class$ over the unit ball $B_1 \subseteq \RR^m$.
\end{axiom}}}

\begin{remark}[Dimensional Scaling and Closure]
  Unlike the $m = 2$ case, where the Dirichlet energy is conformally
  invariant, the energy in dimension $m > 2$ scales super-critically: $\Energy
  (\bm{u}_{x_0, \rr}, B_1) = \rr^{2 - m} \Energy
  (\bm{u}, B_{\rr} (x_0))$. Because $\rr < 1$ and $2 - m
  < 0$, the energy of the rescaled map can arbitrarily magnify, meaning a
  ``small energy'' hypothesis cannot be preserved under inward scaling. The
  generalized framework bypasses this entirely by anchoring the closure to the
  amplitude bound. The uniform bound is  scale-invariant:
  $\|\bm{u}_{x_0, \rr} \|_{L^{\infty} (B_1)} =
  \|\bm{u}\|_{L^{\infty} (B_{\rr} (x_0))} \leqslant M$.
  Furthermore, the forcing terms scale sub-critically: The $L^{\jqq}$ norm of
  the rescaled source yields:
  \begin{equation}
    \| \tilde{\bm{f}}_{x_0, \rr} \|_{L^{\jqq} (B_1)} =
    \left( \int_{B_1} \rr^{2 \jqq} |\bm{f}(x_0 + \rr y)
    |^{\jqq} \mathrm{d} y \right)^{1 / \jqq} = \rr^{2 - {m}/{\jqq}}
    \|\bm{f}\|_{L^{\jqq} (B_{\rr} (x_0))} .
  \end{equation}
  Because we  assume $\jqq > m / 2$, the exponent $2 - m / \jqq > 0$.
  This guarantees that the rescaled source data inherently decays as $\rr \to
  0$. Thus, the zooming process intrinsically preserves the axioms without any
  reliance on the $L^2$ gradient energy.
\end{remark}

{\myprop{\begin{axiom}[The Decay Property]
  \label{ax:decaygen}There exist structural constants $\jth, \jgam \in (0,
  1)$, $\pHold \geqslant 1$, an amplitude bound $M > 0$, and $\kappa > 0$,
  such that for any triplet $(\bm{u}, \jaOm,
  \bm{f}) \in \class$ satisfying
  $\|\bm{u}\|_{L^{\infty} (B_1)} \leqslant M$, the
  oscillation strictly contracts at the interior scale:
  \begin{equation}
    \Osc (\bm{u}, B_{\jth}) \leqslant \jgam \cdot \Osc
    (\bm{u}, B_1) + \kappa \left( \|
    \bm{f} \|_{L^{\jqq} (B_1)} + M \left\| \divg \jaOm
    \right\|_{L^{\jqq} (B_1)} \right)^{\pHold} .
  \end{equation}
\end{axiom}}}

Proceeding as we did for Theorem~\ref{thm:abstract_reg} one gets

{\myprop{\begin{theorem}[Second Abstract Regularity Principle in general
dimension]
  \label{thm:gen_abstract_reg}Let $\class$ be a class of triplets satisfying
  {{\em {\jaxs}~{\textbf{\ref{ax:closuregen}}}}} and {{\em
  {\textbf{\ref{ax:decaygen}}}}}, on the unit ball $B_1 \subset \RR^m$. If
  $(\bm{u}, \jaOm, \bm{f}) \in
  \class$ with $\|\bm{u}\|_{L^{\infty} (B_1)} \leqslant
  M$, and the integrability exponent  satisfies $\jqq > m / 2$, then
  $\bm{u}$ is locally Hölder continuous in $B_1$.
  Specifically, $\bm{u} \in C^{0,
  \jeta}_{\ensuremath{\operatorname{loc}}} (B_1)$ for every Hölder exponent
  $\jeta$ satisfying:
  \begin{equation}
    \jeta < \min \left( \frac{\ln \jgam}{\pHold \ln \jth}, 2 - \frac{m}{\jqq}
    \right) .
  \end{equation}
  For any compact subset $K \subset B_1$, the corresponding Hölder seminorm
  depends only on the chosen exponent $\jeta$, the structural parameters
  $\jth, \jgam, \kappa, \pHold, M, \| \divg \jaOm \|_{L^{\jqq} (B_1)}$,
  $\|\bm{f}\|_{L^{\jqq} (B_1)}$, and the distance
  $\mathrm{dist} (K, \partial B_1)$.
\end{theorem}}}

\begin{proof}[Proof of the Second Abstract Regularity Principle,
Theorem~\ref{thm:gen_abstract_reg}]
  We divide the proof in five steps.{\smallskip}
  
  {\noindent}{{\em Step 1: The Master Iteration Inequality.}} Let
  $(\bm{u}, \jaOm, \bm{f}) \in
  \class$. Our goal is to verify the Campanato condition for
  $\bm{u}$ uniformly on $B_{1 / 2}$. Fix an arbitrary
  center point $x_0 \in B_{1 / 2}$ and a macroscopic starting radius $\rr_0 :=
  1 / 2$ so that $B_{\rr} (x_0) \subset B_1$ for every $x_0 \in B_{1 / 2}$ and
  every $\rr \in (0, \rr_0]$. For any radius $\rr \in (0, \rr_0]$, we set
  $\psi (\rr) := \rr^m \Osc (\bm{u}, B_{\rr} (x_0))$. We
  wish to establish an algebraic iteration inequality for $\psi$.
  
  By {\jax}~\ref{ax:closuregen}, the rescaled triplet
  $(\bm{u}_{x_0, \rr}, \overline{\jaOm}_{x_0, \rr},
  \tilde{\bm{f}}_{x_0, \rr})$ belongs to $\class$. We
  apply the contractive decay property ({\jax}~\ref{ax:decaygen}) to this
  rescaled triplet:
  \begin{equation}
    \Osc (\bm{u}_{x_0, \rr}, B_{\jth}) \leqslant \jgam
    \cdot \Osc (\bm{u}_{x_0, \rr}, B_1) + \kappa \left(
    \| \tilde{\bm{f}}_{x_0, \rr} \|_{L^{\jqq} (B_1)} +
    M\| \divg_y \overline{\jaOm}_{x_0, \rr} \|_{L^{\jqq} (B_1)}
    \right)^{\pHold} . \label{eq:ittemp}
  \end{equation}
  We translate the oscillations back to the physical coordinates on $B_1$.
  Using the scaling identity \eqref{eq:scalingPsi} for the oscillation (mean
  integral), we get:
  \begin{eqnarray}
    \text{LHS:} \quad \Osc (\bm{u}_{x_0, \rr}, B_{\jth})
    & \overset{\eqref{eq:scalingPsi}}{=} & \Osc (\bm{u},
    B_{\jth \rr} (x_0)) = \left( \jth \rr \right)^{- m} \psi (\jth \rr), \\
    \text{RHS:\quad$\Osc (\bm{u}_{x_0, \rr_0}, B_1)$} &
    \overset{\eqref{eq:scalingPsi}}{=} & \Osc (\bm{u},
    B_{\rr} (x_0)) = \rr^{- m} \psi (\rr) . 
  \end{eqnarray}
  Therefore, \eqref{eq:ittemp} can be rewritten as
  \begin{equation}
    \psi (\jth \rr) \leqslant \jth^m \jgam \cdot \psi (\rr) + \kappa \jth^m
    \rr^m \left( \| \tilde{\bm{f}}_{x_0, \rr}
    \|_{L^{\jqq} (B_1)} + M \| \divg_y \overline{\jaOm}_{x_0, \rr}
    \|_{L^{\jqq} (B_1)} \right)^{\pHold} . \label{eq:decaypropforOscgen}
  \end{equation}
  We calculate the $L^{\jqq}$ norms of the rescaled data. Since
  $\tilde{\bm{f}}_{x_0, \rr} (y) := \rr^2
  \bm{f} (x_0 + \rr  y)$, we have $\|
  \tilde{\bm{f}}_{x_0, \rr} \|_{L^{\jqq} (B_1)}^{\jqq} =
  \rr^{2 \jqq - m} \|\bm{f}\|_{L^{\jqq} (B_{\rr}
  (x_0))}^{\jqq} \leqslant \rr^{2 \jqq - m}
  \|\bm{f}\|_{L^{\jqq} (B_1)}^{\jqq}$. Hence
  \begin{equation}\label{eq:scale_ftemp}
    \| \tilde{\bm{f}}_{x_0, \rr} \|_{L^{\jqq}
    (B_1)}^{\pHold} \leqslant \rr^{\pHold \left( 2 - m / \jqq \right)}
    \|\bm{f}\|_{L^{\jqq} (B_1)}^{\pHold} .
  \end{equation}
  Similarly, because $\divg_y \overline{\jaOm}_{x_0, \rr} (y) = \rr^2 (\divg_x
  \jaOm)  (x_0 + \rr y)$, we obtain:
  \begin{equation}\label{eq:scale_divOmegatemp}
    \| \divg_y \overline{\jaOm}_{x_0, \rr} \|_{L^{\jqq} (B_1)} = \rr^{2 - m /
    \jqq} \| \divg_x \jaOm \|_{L^{\jqq} (B_r (x_0))} \leqslant \rr^{2 - m /
    \jqq} \| \divg_x \jaOm \|_{L^{\jqq} (B_1)} .
  \end{equation}
  Substituting the previous bound into the decay property
  \eqref{eq:decaypropforOscgen}, we obtain the master iteration inequality:
  \begin{equation}
    \psi (\jth \rr) \leqslant \jth^m \jgam \cdot \psi (\rr) + b \rr^{\beta} 
    \quad \text{for all } \rr \in (0, \rr_0],
  \end{equation}
  where $\beta := m + \pHold \left( 2 - m / \jqq \right)$, and $b := \kappa
  \jth^m \left( \|\bm{f}\|_{L^{\jqq} (B_1)} + M\| \divg_x
  \jaOm \|_{L^{\jqq} (B_1)} \right)^{\pHold}$ is a constant acting as a
  universal bound: Indeed, since $b$ depends  on the fixed global data
  over the full unit ball $B_1$, it remains completely independent of the
  moving center $x_0$ and the shrinking radius $\rr$.
  
  {\noindent}{{\em Step 2: Continuous Algebraic Decay via the Iteration
  Lemma.}} To extract a continuous algebraic decay rate from this discrete
  bound, we wish to apply the algebraic iteration lemma
  (Lemma~\ref{lem:giusti_iteration}). This is possible because $\psi$ is
  non-decreasing. Arguing as in the proof of Theorem~\ref{thm:abstract_reg},
  we fix an arbitrary small $\varepsilon > 0$ and consider the iteration
  inequality
  \begin{equation}
    \psi (\jth \rr) \leqslant \jth^m \jgam_{\varepsilon} \cdot \psi (\rr) + b
    \rr^{\beta}  \quad \text{for all } \rr \in (0, \rr_0],
  \end{equation}
  with $\jgam_{\varepsilon} := \max (\jgam, \jth^{\beta - m - \varepsilon})$,
  so that, by construction, for $\varepsilon$ sufficiently small, we have
  ({cf.} \eqref{eq:alfaeps1}) $\jgam_{\varepsilon} \in (0, 1)$, $\beta - m -
  \varepsilon > 0$, and ({cf.} \eqref{eq:alfaeps2})
  \begin{equation}
    \beta - \jalfa_{\varepsilon} \geqslant \varepsilon > 0 \quad \text{with}
    \quad \jalfa_{\varepsilon} := m + \min \left( \frac{\ln \jgam}{\ln \jth},
    \beta - m - \varepsilon \right) .
  \end{equation}
  Because $\jalfa_{\varepsilon} < \beta$, the iteration lemma
  (Lemma~\ref{lem:giusti_iteration}) gives the existence of a purely
  structural constant $c_{\varepsilon} > 0$ such that for all $\rr \in (0,
  \rr_0]$ there holds ({cf.}~\eqref{eq:usePS2D0}):
  \begin{equation}
    \Osc (\bm{u}, B_{\rr} (x_0)) = \rr^{- m} \psi \left(
    \rr \right) \leqslant c_{\varepsilon}  \left[ \frac{\Osc
    (\bm{u}, B_{\rr_0}
    (x_0))}{\rr_0^{\jalfa_{\varepsilon} - m}} + b \rr_0^{\beta -
    \jalfa_{\varepsilon}} \right]  \rr^{\jalfa_{\varepsilon} - m},
    \label{eq:usePS2D00}
  \end{equation}
  where, we recall, $b := \kappa \jth^m \left(
  \|\bm{f}\|_{L^{\jqq} (B_1)} + M\| \divg_x \jaOm
  \|_{L^{\jqq} (B_1)} \right)^{\pHold}$.{\smallskip}
  
  {\noindent}{{\em Step 3: Uniform bound via the condition
  $\|\bm{u}\|_{L^{\infty} (B_1)} \leqslant M$.}} Given
  that $\|\bm{u}\|_{L^{\infty} (B_1)} \leqslant M$,
  choosing the test vector $\jc =\ensuremath{\boldsymbol{0}}$ in the infimum
  trivially yields $\Osc (\bm{u}, B_{\rr_0} (x_0))
  \leqslant M^{\pHold}$. Therefore, \eqref{eq:usePS2D00} gives the uniform
  bound
  \begin{equation}
    \Osc (\bm{u}, B_{\rr} (x_0)) \leqslant M_{\ast} 
    \rr^{\jalfa_{\varepsilon} - m}, \label{eq:unifcampnatogen}
  \end{equation}
  where $M_{\ast}$ is a universal constant (independent of $x_0 \in B_{1 / 2}$
  and $\rr \in \left( 0, \rr_0 \right]$) and ({cf.}~\eqref{eq:alfaepsm2})
  \begin{equation}
    \jalfa_{\varepsilon} - m = \min \left( \frac{\ln \jgam}{\ln \jth}, \pHold
    \left( 2 - m / \jqq \right) - \varepsilon \right) .
  \end{equation}
  {\noindent}{{\em Step 4: The Campanato Condition and Hölder Continuity on
  $B_{1 / 2}$.}} Since inequality \eqref{eq:unifcampnatogen} holds uniformly
  for every $x_0 \in B_{1 / 2}$, it satisfies precisely the Campanato
  condition with exponent $\left( \jalfa_{\varepsilon} - m \right)$. Therefore, $\bm{u} \in
  C^{0, \jeta_{\varepsilon}} (B_{1 / 2})$ with $\jeta_{\varepsilon} = \left(
  \jalfa_{\varepsilon} - m \right) / \pHold$. Substituting our exact choice
  for $\left( \jalfa_{\varepsilon} - m \right)$ directly yields:
  \begin{equation}
    \jeta_{\varepsilon} = \min \left( \frac{\ln \jgam}{\pHold \, \ln \jth}, 2
    - \frac{m}{\jqq} - \frac{\varepsilon}{\pHold} \right),
  \end{equation}
  matching the formula in the theorem statement for $\varepsilon
  \rightarrow 0$. Finally, note that the Hölder seminorm $[\bm{u}]_{C^{0,
  \jeta_{\varepsilon}} (B_{1 / 2})}$ depends only on $\jth, \jgam, \kappa,
  \pHold, M, \| \divg_x \jaOm \|_{L^{\jqq} (B_1)}$, and
  $\|\bm{f}\|_{L^{\jqq} (B_1)}$.
  
  {\noindent}{{\em Step 5: Extension to arbitrary compact subsets of
  $B_1$.}} Step 5 follows the same scheme as in the proof of
Theorem~\ref{thm:abstract_reg}, with the energy bound
$\mathcal{E}(\bm{v }, B_1) \leqslant
\mathcal{E}(\bm{u}, B_1)$ replaced by the
(scale-invariant) amplitude bound $\| \bm{v}
\|_{L^{\infty} (B_1)} \leqslant M$ and the sub-critical scaling identities
\eqref{eq:scale_ftemp}--\eqref{eq:scale_divOmegatemp}.
\end{proof}

\subsection{Caccioppoli-type estimate for generalized systems in dimension \texorpdfstring{$m
= 2$}{two}}
To apply the second abstract regularity principle, we must now
demonstrate that our generalized system \eqref{eq:generalized_riviere}
inherently controls the oscillation of the solution at small scales. We
achieve this by locally restructuring the equation in divergence form,
allowing us to tame the critical gradient term and establish a rigorous
Caccioppoli-type estimate specific to dimension $m = 2$.

{\myprop{\begin{lemma}[Generalized Caccioppoli-type estimate in dimension $m =
2$]
  \label{lem:reverse_sobolev}Let $\jqq \in (1, 2)$ be the fixed integrability
  exponent of the source term $\bm{f}$, and let $\pHold
  \in (2, \infty)$ be the corresponding oscillation exponent determined by the
  relation $\text{} \pHold = 2 \jqq / \left( 2 - \jqq \right)$. Assume that
  $\jaOm \in L^2 (B_1, \RR^{(n + 1) \times (n + 1)} \otimes \RR^2)$, $\divg
  \jaOm \in L^{\jqq} (B_1, \RR^{(n + 1) \times (n + 1)})$, and
  $\bm{f} \in L^{\jqq} (B_1, \RR^{n + 1})$. Let
  $\bm{u} \in H^1 \cap L^{\infty} (B_1, \RR^{n + 1})$ be
  a bounded weak solution to
  \begin{equation}
    - \Delta \bm{u}= \jaOm \cdot \nabla
    \bm{u}+\bm{f}
    \label{eq:newRiviere}
  \end{equation}
  with global bound $\|\bm{u}\|_{L^{\infty} (B_1)}
  \leqslant M$.
  
  Let $\rr \in (0, 1)$ be a fixed radius, and let
  $\ensuremath{\boldsymbol{h}}_{\rr} \in H^1 (B_{\rr}, \RR^{n + 1})$ denote
  the unique harmonic extension of $\bm{u}_{\restr
  \partial B_{\rr}}$ to the ball $B_{\rr}$.
  
  With these exponents, the following estimate holds:
  \begin{equation}
    \| \nabla
    (\bm{u}-\ensuremath{\boldsymbol{h}}_{\rr})\|_{L^{\jqq}
    (B_{\rr})} \leqslant C_{\jqq} \| \jaOm \|_{L^2 (B_{\rr})} \left( \rr^m
    \Osc (\bm{u}, B_{\rr}) \right)^{1 / \pHold} +
    C_{\jqq}  \rr \left( \|\bm{f}\|_{L^{\jqq} (B_{\rr})}
    + M\| \divg \jaOm \|_{L^{\jqq} (B_{\rr})} \right),
  \end{equation}
  where $C_{\jqq}$ depends only on $\jqq$ {\opt}and, in general, on the base
  dimension $m$; here $m = 2${\cpt}. In particular $C_{\jqq}$ is independent
  of $\rr$ and $\bm{u}$.
\end{lemma}}}

\begin{proof}
  Since $\jaOm \in L^2$, $\divg \jaOm \in L^{\jqq}$, and
  $\bm{u} \in L^{\infty}$, we invoke the standard
  distributional Leibniz rule. For any constant vector $\jc \in \RR^{n + 1}$,
  the system \eqref{eq:generalized_riviere} can be rewritten locally as:
  \begin{equation}
    - \Delta \bm{u}= \divg (\jaOm
    (\bm{u}- \jc)) +
    \underbrace{\bm{f}- (\divg \jaOm) 
    (\bm{u}- \jc)}_{:= \bm{F}} .
  \end{equation}
  Due to the global bound $\|\bm{u}\|_{L^{\infty}}
  \leqslant M$, by Lemma~\ref{lemma:osc_infimum_bound} we can focus on
  constants $\jc \in \RR^{n + 1}$ such that $\left| \jc \right| \leqslant M$,
  i.e., to constants ensuring pointwise $|\bm{u}- \jc |
  \leqslant 2 M$. By the triangle inequality, the local effective source
  $\bm{F}$ is bounded by:
  \begin{equation}
    \|\bm{F}\|_{L^{\jqq} (B_{\rr})} \leqslant
    \|\bm{f}\|_{L^{\jqq} (B_{\rr})} + 2 M \| \divg \jaOm
    \|_{L^{\jqq} (B_{\rr})} .
  \end{equation}
  We proceed like in Lemma~\ref{lemma:reverse_sobolev}. We decompose
  $\bm{w}_{\rr} :=
  \bm{u}-\ensuremath{\boldsymbol{h}}_{\rr}
  =\bm{w}_1 +\bm{w}_2 \in H^1_0
  (B_{\rr})$, where $- \Delta \bm{w}_1 = \divg (\jaOm
  (\bm{u}- \jc))$ and $- \Delta
  \bm{w}_2 =\bm{F}$. Applying
  standard elliptic estimates and Hölder's inequality ($1 / \jqq = 1 / 2 + 1 /
  \pHold$) we get
  \begin{equation}
    \| \nabla \bm{w}_1 \|_{L^{\jqq} (B_{\rr})} \leqslant
    C_{\jqq} \| \jaOm \|_{L^2 (B_{\rr})}  \|\bm{u}- \jc
    \|_{L^{\pHold} (B_{\rr})} .
  \end{equation}
  For the subcritical source term involving $\bm{F}$,
  classical estimates scaled to $B_{\rr}$ yield:
  \begin{equation}
    \| \nabla \bm{w}_2 \|_{L^{\jqq} (B_{\rr})} \leqslant
    C_{\jqq}  \rr \|\bm{F}\|_{L^{\jqq} (B_{\rr})} .
  \end{equation}
  The conclusion follows by the triangle inequality, summing the bounds, and
  absorbing the involved absolute constants into $C_{\jqq}$.
\end{proof}

Having established (in Lemma~\ref{lem:reverse_sobolev})  the existence of a generalized Caccioppoli-type estimate  and the controlled harmonic
approximation (Lemma~\ref{lemma:harmonic_approx}), we can now combine these
results to obtain the following result which will be the main tool in the
derivation of the decay estimate for almost harmonic maps.

{\myprop{\begin{lemma}[Generalized coupled Caccioppoli-type estimate]
  \label{lemma:reverse_sobolev_bound3}Let $\jqq \in (1, 2)$ be the fixed
  integrability exponent of the source term $\bm{f}$, and
  let $\pHold \in (2, \infty)$ be the corresponding oscillation exponent
  determined by the relation $\text{} \pHold = 2 \jqq / \left( 2 - \jqq
  \right)$. Assume that $\jaOm \in L^2 (B_1, \RR^{(n + 1) \times (n + 1)}
  \otimes \RR^2)$, $\divg \jaOm \in L^{\jqq} (B_1, \RR^{(n + 1) \times (n +
  1)})$, and $\bm{f} \in L^{\jqq} (B_1, \RR^{n + 1})$.
  Let $\bm{u} \in H^1 \cap L^{\infty} (B_1, \RR^{n + 1})$
  be a bounded weak solution to
  \begin{equation}
    - \Delta \bm{u}= \jaOm \cdot \nabla
    \bm{u}+\bm{f},
  \end{equation}
  with global bound $\|\bm{u}\|_{L^{\infty} (B_1)}
  \leqslant M$.
  
  Then, there exists a harmonic function $\ensuremath{\boldsymbol{h}}: B_{1 /
  4} \to \RR^{n + 1}$ such that the following gradient estimate holds on $B_{1
  / 4}$:
  \begin{equation}
    \| \nabla
    (\bm{u}-\ensuremath{\boldsymbol{h}})\|_{L^{\jqq}
    (B_{1 / 4})} \leqslant C_{\jqq}  \| \jaOm \|_{L^2 (B_1)} \Osc^{1 / \pHold}
    (\bm{u}, B_1) + C_{\jqq} \left(
    \|\bm{f}\|_{L^{\jqq} (B_1)} + M\| \divg \jaOm
    \|_{L^{\jqq} (B_1)} \right), \label{eq:reverse_sobolev_bound3}
  \end{equation}
  where $C_{\jqq} > 0$ depends only on the dimension $m = 2$ and the exponent
  $\jqq$. Moreover,
  \begin{equation}
    \sup_{x \in B_{1 / 4}} | \nabla \ensuremath{\boldsymbol{h}}(x) | \leqslant
    C_m \cdot \Osc^{1 / \pHold} (\bm{u}, B_1),
    \label{eq:harmapproxosc3}
  \end{equation}
  where $C_m > 0$ is a dimensional constant.
\end{lemma}}}

\begin{proof}
  The argument is structurally identical to the proof of
  Lemma~\ref{lemma:reverse_sobolev_bound}, with the following two
  substitutions. First, Lemma~\ref{lem:reverse_sobolev} (rather than
  Lemma~\ref{lemma:reverse_sobolev}) is invoked on the good radius $\rr \in [1
  / 2, 1]$, providing
  \begin{equation}
    \| \nabla
    (\bm{u}-\ensuremath{\boldsymbol{h}}_{\rr})\|_{L^q
    (B_{\rr})} \leqslant C_{\jqq}  \| \jaOm \|_{L^2 (B_{\rr})} (\rr^m \Osc
    (\bm{u}, B_{\rr}))^{1 / \pHold} + C_{\jqq}  \rr
    \left( \|\bm{f}\|_{L^{\jqq} (B_{\rr})} + M\| \divg
    \jaOm \|_{L^{\jqq} (B_{\rr})} \right) .
  \end{equation}
  Second, the bound on the harmonic extension $\ensuremath{\boldsymbol{h}}$
  (Lemma~\ref{lemma:harmonic_approx}) is unchanged. Restricting the LHS to
  $B_{1 / 4} \subset B_{\rr}$ and using $\rr \leqslant 1$ together with the
  monotonicity of $\rr \mapsto \rr^m \Osc$
  yields~\eqref{eq:reverse_sobolev_bound3} and~\eqref{eq:harmapproxosc3}.
\end{proof}

\subsection{The Generalized Decay Estimate}
We now prove that the structural matrix $\jaOm$ dictates the decay of the
oscillation provided its $L^2$ energy is sufficiently small. Specifically, the
content of the next result is to prove that the class of almost harmonic maps
 satisfies the decay property formulated in {\jax}~\ref{ax:decaygen}.

{\myprop{\begin{lemma}[Generalized Decay Property]
  \label{lem:gen_decay}For any integrability exponent $\jqq \in (1, 2)$, let
  $\pHold := 2 \jqq / \left( 2 - \jqq \right)$ be its corresponding
  oscillation exponent. There exist structural constants $\jth, \jgam \in (0,
  1)$, $\kappa > 0$, and an $\jaOm$-energy threshold $\varepsilon_{\ast} > 0$
  such that the following holds. Let $\bm{u}$ be a
  bounded weak solution on the unit ball $B_1$ to the equation {\opt}see {{\em
  \eqref{eq:generalized_riviere}}}{\cpt}:
  \begin{equation}
    - \Delta \bm{u}= \jaOm \cdot \nabla
    \bm{u}+\bm{f} \quad \text{in }
    B_1,
  \end{equation}
  satisfying the global bound $\|\bm{u}\|_{L^{\infty}
  (B_1)} \leqslant M$ for some $M > 0$. If the coefficient $\jaOm$ satisfies
  the smallness condition $\| \jaOm \|_{L^2 (B_1)} < \varepsilon_{\ast}$, then
  the oscillation of $\bm{u}$ decays according to the
  estimate:
  \begin{equation}
    \Osc (\bm{u}, B_{\jth}) \leqslant \jgam \cdot \Osc
    (\bm{u}, B_1) + \kappa \left(  \|
    \bm{f} \|_{L^{\jqq} (B_1)} + M \left\| \divg \jaOm
    \right\|_{L^{\jqq} (B_1)} \right)^{\pHold} .
  \end{equation}
  In fact, one can always take $\jgam = 1 / 2$.
\end{lemma}}}

\begin{proof}
  To distinguish the exponent originating from $L^2$-integrability from the
  spatial dimension, we will temporarily denote the domain dimension by $m$
  (where $m = 2$).
  
  By Lemma~\ref{lemma:reverse_sobolev_bound3}, there exists a harmonic
  function $\ensuremath{\boldsymbol{h}}: B_{1 / 4} \rightarrow \RR^{n + 1}$
  satisfying the gradient estimates \eqref{eq:reverse_sobolev_bound3} and
  \eqref{eq:harmapproxosc3}. For $\jth \in (0, 1 / 4)$, the triangle
  inequality gives:
  \begin{equation}
    \Osc^{1 / \pHold} (\bm{u}, B_{\jth}) \leqslant
    \Osc^{1 / \pHold}
    (\bm{u}-\ensuremath{\boldsymbol{h}}, B_{\jth}) +
    \Osc^{1 / \pHold} (\ensuremath{\boldsymbol{h}}, B_{\jth}) = : I + I  I.
  \end{equation}
  For $I$, applying the Poincaré-Sobolev inequality (noting the exact matching
  of the Sobolev conjugate exponent $\pHold$ in dimension two), the coupled
  Caccioppoli-type estimate (see Lemma~\ref{lemma:reverse_sobolev_bound3}), and
  using that $\rr^2 \Osc (\bm{u}, B_{\rr}) \leqslant \Osc
  (\bm{u}, B_1)$ because of the monotonicity of $\rr^2
  \Osc (\bm{u}, B_{\rr})$, we deduce that:
  \begin{eqnarray}
    I & \leqslant & \frac{C_S}{\left( \omega_m  \jth^m \right)^{1 / \pHold}} 
    \| \nabla
    (\bm{u}-\ensuremath{\boldsymbol{h}})\|_{L^{\jqq}
    (B_{\jth})} \\
    & \leqslant & \frac{C_{\jqq}}{\jth^{m / \pHold}} \| \jaOm \|_{L^2 (B_1)}
    \Osc^{1 / \pHold} (\bm{u}, B_1) +
    \frac{C_{\jqq}}{\jth^{m / \pHold}} \left(
    \|\bm{f}\|_{L^{\jqq} (B_1)} + M\| \divg \jaOm
    \|_{L^{\jqq} (B_1)} \right) . 
  \end{eqnarray}
  Here, the constant $C_{\jqq}$ has been harmlessly redefined to absorb the
  Poincaré-Sobolev constant $C_S$ and geometric constants.
  
  It remains to estimate $I  I := \Osc^{1 / \pHold}
  (\ensuremath{\boldsymbol{h}}, B_{\jth})$. For that, the mean value theorem
  and the gradient bound for $\ensuremath{\boldsymbol{h}}$ yield
  \begin{equation}
    I  I \leqslant \jth \sup_{B_{1 / 4}} | \nabla \ensuremath{\boldsymbol{h}}|
    \leqslant C_m  \jth \Osc^{1 / \pHold} (\bm{u}, B_1)
  \end{equation}
  Combining the estimates for $I$ and $I  I$, and raising to the power
  $\pHold$, we obtain the master decay inequality purely in terms of the
  original PDE data:
  \begin{equation}
    \label{eq:gen_decay_master} \Osc (\bm{u}, B_{\jth})
    \leqslant C_{m, \jqq}^{\pHold}  \left( \frac{\| \jaOm \|_{L^2
    (B_1)}^p}{\jth^m} + \jth^{\pHold} \right) \Osc
    (\bm{u}, B_1) + \frac{C_{m, \jqq}^{\pHold}}{\jth^m}
    \left( \|\bm{f}\|_{L^{\jqq} (B_1)} + M\| \divg \jaOm
    \|_{L^{\jqq} (B_1)} \right)^{\pHold},
  \end{equation}
  with $C_{m, \jqq}$ a positive constant that depends only on $m$ and $\jqq$.
  
  We fix the geometric decay factor $\jgam = 1 / 2$. First, select $\jth \in
  (0, 1 / 4)$ sufficiently small such that $C_{m, \jqq}^{\pHold} 
  \jth^{\pHold} \leqslant 1 / 4$. With $ \jth$ fixed, we  define the
  threshold parameter $\varepsilon_{\ast} > 0$ such that $C_{m, \jqq}^{\pHold}
  \varepsilon_{\ast}^p / \jth^m \leqslant 1 / 4$. Provided $\| \jaOm \|_{L^2
  (B_1)} < \varepsilon_{\ast}$, the scaling brackets in
  \eqref{eq:gen_decay_master} sum to at most $1 / 2$. Setting $\kappa := C_{m,
  \jqq}^{\pHold} / \jth^m$ concludes the proof.
\end{proof}

\subsection{Proof of the Main Regularity Theorem
(Theorem~\ref{thm:generalized_regularity})}\label{sec:genregproof}
Let $\bm{u}$ be an arbitrary bounded weak solution to the
generalized system. We fix its global amplitude bound $M :=
\|\bm{u}\|_{L^{\infty} (B_1)}$. Let $\varepsilon_{\ast}$
be the critical threshold from Lemma~\ref{lem:gen_decay}. To invoke
Theorem~\ref{thm:gen_abstract_reg}, we define the admissible class $\class$ as
the set of triplets $(\bm{u}, \jaOm,
\bm{f}) \in \mathcal{X}$ satisfying the original system,
the amplitude bound $M$, and the strict energy threshold:
\begin{equation}
  \class := \left\{ (\bm{u}, \jaOm,
  \bm{f}) \in \mathcal{X} \st - \Delta
  \bm{u}= \jaOm \cdot \nabla
  \bm{u}+\bm{f}, \; \;
  \|\bm{u}\|_{L^{\infty} (B_1)} \leqslant M \; \;
  \text{and} \; \; \| \jaOm \|_{L^2 (B_1)} < \varepsilon_{\ast} \right\},
\end{equation}
where the ambient space explicitly includes the divergence regularity:
\begin{equation}
  \mathcal{X} := \left( H^1 \cap L^{\infty} (B_1, \RR^{n + 1}) \right) \times
  \left\{ \jaOm \in L^2 (B_1) \st \divg \jaOm \in L^{\jqq} (B_1) \right\}
  \times L^{\jqq} (B_1, \RR^{n + 1}) .
\end{equation}
We need to verify {\jaxs}~\ref{ax:closuregen} and \ref{ax:decaygen} required
by Theorem~\ref{thm:gen_abstract_reg}.

{\noindent}{{\em Step 1: Verification of {{\em {\jax}~\ref{ax:closuregen}}}
{\opt}closure under rescaling{\cpt}.}} Let $(\bm{u},
\jaOm, \bm{f}) \in \class$ and let $B_{\rr} (x_0) \subset
B_1$. Consider the rescaled triple $(\bm{u}_{x_0, \rr},
\overline{\jaOm}_{x_0, \rr}, \tilde{\bm{f}}_{x_0, \rr})$.
We need to show that this triple remains in $\class$. Simple scaling shows
that
\begin{equation}
  - \lapl \bm{u}_{x_0, \rr} = \overline{\jaOm}_{x_0, \rr}
  \cdot \nabla \bm{u}_{x_0, \rr} +
  \tilde{\bm{f}}_{x_0, \rr} .
\end{equation}
Furthermore, the global bound is trivially preserved
($\|\bm{u}_{x_0, \rr} \|_{L^{\infty} (B_1)} \leqslant
M$), and the conformal invariance of the $L^2$ norm in dimension two
implies $\| \overline{\jaOm}_{x_0, r} \|_{L^2 (B_1)} = \| \jaOm \|_{L^2
(B_{\rr} (x_0))} \leqslant \| \jaOm \|_{L^2 (B_1)} < \varepsilon_{\ast}$.
Thus, the class $\class$ is closed under rescaling.

{\noindent}{{\em Step 2: Verification of {{\em {\jax} \ref{ax:decaygen}}}
{\opt}the decay property{\cpt}. }}{\jax}~\ref{ax:decaygen} is immediately
satisfied because of Lemma~\ref{lem:gen_decay}.

{\noindent}{{\em Step 3: Conclusion and Removal of the small $\jaOm$-energy
Hypothesis.}} Having verified the two axioms, the abstract regularity
principle (Theorem~\ref{thm:gen_abstract_reg}) dictates that any map
$\bm{u}$ belonging to a triple in $\class$ is locally
Hölder continuous in $B_1$. To conclude the proof of
Theorem~\ref{thm:generalized_regularity}, it remains to explain why this
regularity holds for {{\em any}} weak solution of the generalized system
\eqref{eq:generalized_riviere}, not just those with small initial
$\jaOm$-energy.

Let $\bm{u} \in H^1 (B_1, \RR^{n + 1})$ be an arbitrary
solution of \eqref{eq:generalized_riviere}. This map
$\bm{u}$ might have a large total $\jaOm$-energy, meaning
$(\bm{u}, \jaOm, \bm{f}) \not\in
\class$. However, let $x_0 \in B_1$ be an arbitrary point, and let $R := 1 -
|x_0 | > 0$ denote its distance to the boundary. Because $\jaOm \in L^2
(B_1)$, by the absolute continuity of the Lebesgue integral we can always find
a sufficiently small radius $\rr \in (0, R)$ such that $\left\| \jaOm
\right\|_{L^2 (B_{\rr} (x_0))} < \varepsilon_{\ast}$.

If we now construct the rescaled pair $(\bm{u}_{x_0,
\rr}, \overline{\jaOm}_{x_0, \rr}, \tilde{\bm{f}}_{x_0,
\rr})$, it is defined on the unit ball $B_1$, solves the required equation,
and  satisfies the small-energy threshold $\left\|
\overline{\jaOm}_{x_0, \rr} \right\|_{L^2 (B_1)} < \varepsilon_{\ast}$.
Consequently, this rescaled triple belongs to our class $\class$. By
Theorem~\ref{thm:gen_abstract_reg}, the rescaled map
$\bm{u}_{x_0, \rr}$ is locally Hölder continuous in
$B_1$. Scaling back to the original spatial coordinates, this implies that the
original map $\bm{u}$ is Hölder continuous in the
neighborhood $B_{\rr / 2} (x_0)$. Since the point $x_0 \in B_1$ was chosen
arbitrarily, the map $\bm{u}$ is locally Hölder
continuous everywhere in $B_1$, completing the proof.

\section{Regularity of anisotropic systems with general connection
forms}\label{sec:aniso}
In the preceding section, we showed that the antisymmetry condition on the
connection matrix $\jaOm$ can be replaced by the analytic assumption
$\divg \jaOm \in L^{\jqq}$. The principal operator was still the Laplacian,
corresponding to the identity coefficient matrix. We now study the anisotropic
case and determine the regularity assumptions on the coefficients required for
the Campanato decay argument.
Anisotropic Dirichlet energies of the form
\begin{equation}
  \Energy (\bm{u}, B_1) := \frac{1}{2}  \int_{B_1} A
  \nabla \bm{u} : \nabla \bm{u}
  \hspace{0.17em} \mathrm{d} x,
\end{equation}
arise in models of inhomogeneous media. Their critical points satisfy the equation $- \divg (A \nabla \bm{u}) =
  \jaOm \cdot \nabla \bm{u}+\bm{f}$, where $A(x)$ is symmetric and uniformly elliptic, and $\jaOm$ describes the
geometric structure of the target.
For constant $A$, the system reduces, after a linear change of variables, to
the previous setting. If $\det A$ is constant, a uniformization argument
reduces the problem to a harmonic-map-type system. This approach does not
apply when $\det A$ varies, which requires a direct analysis.

The goal of this section is to prove the following generalization, where
$\jaOm$ is not assumed to be antisymmetric and the principal operator is not
restricted to the Laplacian.

{\noindent}\textbf{Theorem~\ref{thm:aniso_regularity}.} (Regularity of anisotropic systems) 
{\itshape{ Let
$B_1 \subset \RR^2$. Let $\bm{u} \in H^1 \cap L^{\infty}
(B_1, \RR^{n + 1})$ be a weak solution of
\begin{equation}
  \label{eq:aniso_main} - \divg (A \nabla \bm{u}) = \jaOm
  \cdot \nabla \bm{u}+\bm{f} \qquad
  \text{in } B_1 .
\end{equation}
Assume the structural data satisfy, for some $\jqq > 1$ and some $\gamma_0 \in
(0, 1]$:
\begin{itemize}
  \item[\text{{\upshape{(A1)}}}] $A : B_1 \to \RR^{2 \times 2}$ is symmetric,
  uniformly elliptic, and Hölder continuous:
  \begin{equation} \lambda | \xi |^2 \leqslant A (x) \xi \cdot \xi \leqslant \Lambda | \xi
     |^2 \qquad \forall x \in B_1, \xi \in \RR^2, \quad 0 < \lambda \leqslant
     \Lambda < + \infty, \end{equation}
  and $A \in C^{0, \gamma_0} (\overline{B_1})$;
  \item[\text{{\upshape{(A2)}}}] $\jaOm \in L^2 (B_1, \RR^{(n + 1) \times (n +
  1)} \otimes \RR^2)$;
  \item[\text{{\upshape{(A3)}}}] $\bm{f} \in L^{\jqq}
  (B_1, \RR^{n + 1})$;
  \item[\text{{\upshape{(A4)}}}] $\divg \jaOm \in L^{\jqq} (B_1, \RR^{(n + 1)
  \times (n + 1)})$.
\end{itemize}
Then $\bm{u}$ is locally Hölder continuous in $B_1$.
Specifically, there exists an exponent $\eta \in (0, 1)$ such that
$\bm{u} \in C^{0, \eta}_{\ensuremath{\operatorname{loc}}}
(B_1, \RR^{n + 1})$.}}

\begin{remark}
  In what follows, we may assume without loss of generality that $1 < \jqq <
  2$. The reason is identical to that explained in
  Remark~\ref{rmk:redsourceintq}.
\end{remark}
\begin{remark}[Necessity of H\"older Continuous Coefficients]
The H\"older continuity of the principal part $A$ is used to control the freezing error $\tilde{A}\nabla \bm{u}$ in the master iteration inequality \eqref{eq:master_aniso}. To apply the algebraic iteration lemma (Lemma~\ref{lem:giusti_iteration}) and obtain local H\"older continuity of $\bm{u}$, the coefficient error $\|\tilde{A}\|_{L^\infty(B_r)}$ must decay at a power rate as $r\to0$. The assumption $A\in C^{0,\gamma_0}$ gives precisely this estimate:
\begin{equation}
  \label{eq:tildeA_bound} 
  \| \tilde{A} \|_{L^{\infty} (B_r)} \leqslant r^{\gamma_0} [A]_{C^{0, \gamma_0} (\overline{B_1})}
  \qquad \forall r \in (0,1].
\end{equation}
Hence, the freezing contribution decays as $r^{\gamma_0}$ and can be incorporated into the iteration. Uniform continuity of $A$ would be sufficient for continuity of $\bm{u}$, but the H\"older assumption provides the algebraic decay needed for the H\"older estimate.
\end{remark}

The proof of Theorem~\ref{thm:aniso_regularity} is given in
Subsection~\ref{ssec:aniso_proof_main}. To this end, we first introduce a
slight extension of the abstract regularity principle of
Theorem~\ref{thm:gen_abstract_reg} to accommodate the additional structural
datum $A$.

\subsection{The anisotropic abstract regularity
principle}\label{ssec:aniso_abstract} We now introduce a second, distinct
abstract regularity framework designed specifically for anisotropic systems.
It is important to emphasize that this result does not subsume the generalized
abstract regularity principle of the previous sections as a special case;
rather, it stands as a  complementary result.

%

Let $\class$ be a class of quadruples $(\bm{u}, \jaOm,
\bm{f}, A)$ belonging to the space
\begin{eqnarray}
  \mathcal{Y} & := & \left( H^1 \cap L^{\infty} (B_1, \RR^{n + 1}) \right)
  \times L^2 (B_1, \RR^{(n + 1) \times (n + 1)} \otimes \RR^m) \nonumber\\
  &  & \qquad \qquad \qquad \qquad \times L^{\jqq} (B_1, \RR^{n + 1}) \times
  C^{0, \gamma_0} (\overline{B_1}, \RR^{m \times m}_{\mathrm{sym}}), 
  \label{eq:defmathcalY}
\end{eqnarray}
where $A$ is uniformly elliptic with constants $0 < \lambda \leqslant \Lambda
< \infty$ and $\divg \jaOm \in L^{\jqq} (B_1)$. For any ball $B_{\rr} (x_0)
\subset B_1$, we define the rescaled quadruple on $B_1$ by
\begin{eqnarray}
  \bm{u}_{x_0, \rr} (y) := \bm{u}
  (x_0 + \rr y), &  & \overline{\jaOm}_{x_0, \rr} (y) := \rr \jaOm (x_0 + \rr
  y), \\
  \tilde{\bm{f}}_{x_0, \rr} (y) := \rr^2
  \bm{f} (x_0 + \rr y), &  & A_{x_0, \rr} (y) := A (x_0 +
  \rr y) .  \label{eq:aniso_rescaling}
\end{eqnarray}
\begin{remark}[Scaling identities]
  \label{rmk:aniso_scaling}The rescaling \eqref{eq:aniso_rescaling} preserves
  the ellipticity bounds of $A$ identically and obeys the algebraic
  identities:
  \begin{eqnarray}
    \|\bm{u}_{x_0, \rr} \|_{L^{\infty} (B_1)} & = &
    \|\bm{u}\|_{L^{\infty} (B_{\rr} (x_0))}, 
    \label{eq:scale_uinf}\\
    \| \nabla \bm{u}_{x_0, \rr} \|_{L^2 (B_1)} & = &
    \rr^{(2 - m) / 2} \| \nabla \bm{u}\|_{L^2 (B_{\rr}
    (x_0))},  \label{eq:scale_grad}\\
    \| \overline{\jaOm}_{x_0, \rr} \|_{L^2 (B_1)} & = & \rr^{(2 - m) / 2} \|
    \jaOm \|_{L^2 (B_{\rr} (x_0))},  \label{eq:scale_Omega}\\
    \| \divg_y \overline{\jaOm}_{x_0, \rr} \|_{L^{\jqq} (B_1)} & = & \rr^{2 -
    m / \jqq} \| \divg_x \jaOm \|_{L^{\jqq} (B_{\rr} (x_0))}, 
    \label{eq:scale_divOmega}\\
    \| \tilde{\bm{f}}_{x_0, \rr} \|_{L^{\jqq} (B_1)} & =
    & \rr^{2 - m / \jqq} \|\bm{f}\|_{L^{\jqq} (B_{\rr}
    (x_0))},  \label{eq:scale_f}\\
    {}[A_{x_0, \rr}]_{C^{0, \gamma_0} (\overline{B_1})} & = & \rr^{\gamma_0}
    [A]_{C^{0, \gamma_0} (\overline{B_{\rr} (x_0)})} . 
    \label{eq:scale_Holder}
  \end{eqnarray}
  In dimension $m = 2$, the gradient $L^2$-norm and the connection $L^2$-norm
  are  scale-invariant (\eqref{eq:scale_grad}--\eqref{eq:scale_Omega});
  the source norms decay with exponent $2 - m / \jqq > 0$ (whenever $\jqq > m
  / 2$); and the Hölder seminorm of $A$ contracts with exponent $\gamma_0$.
  These three decay mechanisms are the engine driving the iteration in
  Theorem~\ref{thm:aniso_abstract_reg} below.
\end{remark}

The class is required to obey the two axioms below.

{\mydef{\begin{axiom}[Locality and closure]
  \label{ax:closure_aniso}The class $\class$ is closed under the rescaling
  {{\em \eqref{eq:aniso_rescaling}}}: for any $(\bm{u},
  \jaOm, \bm{f}, A) \in \class$ and any $B_{\rr} (x_0)
  \subset B_1$, the rescaled quadruple $(\bm{u}_{x_0,
  \rr}, \overline{\jaOm}_{x_0, \rr}, \tilde{\bm{f}}_{x_0,
  \rr}, A_{x_0, \rr})$ is again in $\class$ over the unit ball $B_1 \subseteq
  \RR^m$.
\end{axiom}}}

{\mydef{\begin{axiom}[Decay property]
  \label{ax:decay_aniso}There exist structural constants $\jth \in (0, 1)$,
  $\gamma \in (0, 1)$, $\pHold \geqslant 1$, an amplitude bound $M > 0$, and
  $\kappa > 0$, such that for any quadruple $(\bm{u},
  \jaOm, \bm{f}, A) \in \class$ satisfying
  $\|\bm{u}\|_{L^{\infty} (B_1)} \leqslant M$, the
  oscillation contracts at the interior scale:
  \begin{equation}
    \label{eq:axiom_decay_aniso} \Osc (\bm{u}, B_{\jth})
    \leqslant \jgam \cdot \Osc (\bm{u}, B_1) + \kappa
    \left( D_{\mathrm{src}}^{\pHold} + D_{\mathrm{frz}}^{\pHold} \right),
  \end{equation}
  where the source and freezing data are
  \begin{equation}
    \label{eq:def_Dsrc_Dfrz} D_{\mathrm{src}} :=
    \|\bm{f}\|_{L^{\jqq} (B_1)} + M \| \divg \jaOm
    \|_{L^{\jqq} (B_1)}, \qquad D_{\mathrm{frz}} := \| \tilde{A}
    \|_{L^{\infty} (B_1)}  \| \nabla \bm{u}\|_{L^2
    (B_1)},
  \end{equation}
  with $\tilde{A} := A - A (0)$.
\end{axiom}}}

Proceeding as in the proof of Theorem~\ref{thm:gen_abstract_reg}, one obtains
the following abstract result.

{\myprop{\begin{theorem}[Anisotropic abstract regularity principle]
  \label{thm:aniso_abstract_reg}Let $\class$ be a class of quadruples
  satisfying \text{{\upshape{{\jaxs}~\ref{ax:closure_aniso}}}} and
  \text{{\upshape{\ref{ax:decay_aniso}}}} on the unit ball $B_1 \subset
  \RR^m$. Let $(\bm{u}, \jaOm,
  \bm{f}, A) \in \class$ with
  $\|\bm{u}\|_{L^{\infty} (B_1)} \leqslant M$. Assume
  that $A$ is Hölder continuous of exponent $\gamma_0 \in (0, 1]$ satisfying
  $\gamma_0 > \frac{m - 2}{2}$ and that $\jqq > m / 2$. Then
  $\bm{u}$ is locally Hölder continuous in $B_1$:
  $\bm{u} \in C^{0,
  \jeta}_{\ensuremath{\operatorname{loc}}} (B_1)$ for every exponent $\jeta$
  satisfying
  \begin{equation}
    \label{eq:aniso_holder_exponent} \jeta < \min \left( \frac{\ln
    \jgam}{\pHold \ln \jth}, 2 - \frac{m}{\jqq}, \gamma_0 + \frac{2 - m}{2}
    \right) .
  \end{equation}
  For any compact $K \subset B_1$, the corresponding Hölder seminorm depends
  only on $\jeta$, the structural parameters $\jth, \jgam, \kappa, \pHold, M$,
  the data norms $\| \divg \jaOm \|_{L^{\jqq} (B_1)}$,
  $\|\bm{f}\|_{L^{\jqq} (B_1)}$, $\| \nabla
  \bm{u}\|_{L^2 (B_1)}$, $[A]_{C^{0, \gamma_0}
  (\overline{B_1})}$, and the distance $\ensuremath{\operatorname{dist}} (K,
  \partial B_1)$.
\end{theorem}}}

\begin{remark}
  Notice that when $m = 2$, the term $\frac{2 - m}{2}$ vanishes, and
  we get that $\bm{u} \in C^{0,
  \jeta}_{\ensuremath{\operatorname{loc}}} (B_1)$ for every exponent $\jeta$
  satisfying
  \begin{equation}
    \jeta < \min \left( \frac{\ln \jgam}{\pHold \ln \jth}, 2 - \frac{2}{\jqq},
    \gamma_0 \right) .
  \end{equation}
  That is, the third entry reduces  to $\gamma_0$, completely
  independent of the integrability exponent $\pHold$.
\end{remark}


\begin{proof}
  We follow the structure of the proof of Theorem~\ref{thm:gen_abstract_reg}
  in five steps, indicating only the modifications required by the additional
  freezing datum.{\smallskip}
  
  {\noindent}\text{{\itshape{Step~1: Master iteration inequality.}}} Fix $x_0
  \in B_{1 / 2}$ and set $\rr_0 := 1 / 2$ so that $B_{\rr} (x_0) \subset B_1$
  for every $\rr \in (0, \rr_0]$. Define, for $\rr \in (0, \rr_0]$, $\psi
  (\rr) := \rr^m \Osc (\bm{u}, B_{\rr} (x_0))$. By
  {\jax}~\ref{ax:closure_aniso}, the rescaled quadruple
  $(\bm{u}_{x_0, \rr}, \overline{\jaOm}_{x_0, \rr},
  \tilde{\bm{f}}_{x_0, \rr}, A_{x_0, \rr})$ belongs to
  $\class$. Note that $\|\bm{u}_{x_0, \rr} \|_{L^{\infty}
  (B_1)} = \|\bm{u}\|_{L^{\infty} (B_{\rr} (x_0))}
  \leqslant M$. Applying {\jax}~\ref{ax:decay_aniso} to the rescaled
  quadruple,
  \begin{equation}
    \label{eq:aniso_decay_pulled_back} \Osc (\bm{u}_{x_0,
    \rr}, B_{\jth}) \leqslant \jgam \Osc (\bm{u}_{x_0,
    \rr}, B_1) + \kappa \left( \bar{D}_{\mathrm{src}, \rr}^{\pHold} +
    \bar{D}_{\mathrm{frz}, \rr}^{\pHold} \right),
  \end{equation}
  where $\bar{D}_{\mathrm{src}, \rr}$ and $\bar{D}_{\mathrm{frz}, \rr}$ denote
  the source and freezing data computed for the rescaled quadruple.
  
Translating the oscillation estimates back to the original coordinates via the standard scaling identity $\Osc (\bm{u}{x_0, \rr}, B{\jth \rr}) = \Osc (\bm{u}, B_{\jth \rr} (x_0))$, we obtain
  \begin{eqnarray}
    \Osc (\bm{u}_{x_0, \rr}, B_{\jth}) & = & (\jth
    \rr)^{- m} \psi (\jth \rr), \\
    \Osc (\bm{u}_{x_0, \rr}, B_1) & = & \rr^{- m} \psi
    (\rr) . \nonumber
  \end{eqnarray}
  For the source term, by \eqref{eq:scale_f} and \eqref{eq:scale_divOmega},
  \begin{equation}
    \bar{D}_{\mathrm{src}, \rr} \leqslant \rr^{2 - m / \jqq} \left(
    \|\bm{f}\|_{L^{\jqq} (B_1)} + M\| \divg \jaOm
    \|_{L^{\jqq} (B_1)} \right) = \rr^{2 - m / \jqq} D_{\mathrm{src}} .
  \end{equation}
  For the freezing term, by \eqref{eq:scale_grad} and \eqref{eq:tildeA_bound},
  \begin{equation}
    \bar{D}_{\mathrm{frz}, \rr} = \| \tilde{A}_{x_0, \rr} \|_{L^{\infty}
    (B_1)}  \| \nabla \bm{u}_{x_0, \rr} \|_{L^2 (B_1)}
    \leqslant \rr^{\gamma_0 + (2 - m) / 2} [A]_{C^{0, \gamma_0}}  \| \nabla
    \bm{u}\|_{L^2 (B_1)},
  \end{equation}
  where in the rescaled quadruple $\tilde{A}_{x_0, \rr} = A_{x_0, \rr} -
  A_{x_0, \rr} (0) = A (x_0 + \rr \cdot) - A (x_0)$. Substituting the previous
  estimates into \eqref{eq:aniso_decay_pulled_back} and multiplying by $(\jth
  \rr)^m$, we obtain
  \begin{equation}
    \label{eq:aniso_master_psi} \psi (\jth \rr) \leqslant \jth^m \jgam \psi
    (\rr) + b_{\mathrm{src}} \rr^{\beta_{\mathrm{src}}} + b_{\mathrm{frz}}
    \rr^{\beta_{\mathrm{frz}}},
  \end{equation}
  with exponents
  \begin{equation}
    \beta_{\mathrm{src}} := m + \pHold (2 - m / \jqq), \quad
    \beta_{\mathrm{frz}} := m + \pHold \left( \gamma_0 + \frac{2 - m}{2}
    \right),
  \end{equation}
  and constants
  \begin{equation}
    b_{\mathrm{src}} := \kappa \jth^m D_{\mathrm{src}}^{\pHold}, \quad
    b_{\mathrm{frz}} := \kappa \jth^m [A]_{C^{0, \gamma_0}
    (\overline{B_1})}^{\pHold}  \| \nabla \bm{u}\|_{L^2
    (B_1)}^{\pHold},
  \end{equation}
  depending only on the global data $D_{\mathrm{src}}, D_{\mathrm{frz}}$,
  $\kappa$, $\jth$, $M$, $[A]_{C^{0, \gamma_0}}$. Since both exponents satisfy
  $\beta_{\mathrm{src}} > m$ (because $\jqq > m / 2$) and
  $\beta_{\mathrm{frz}} > m$ (because we  assumed $\gamma_0 > (m - 2)
  / 2$; in particular for $m = 2$, this reduces to $\gamma_0 > 0$), setting
  $\beta := \min (\beta_{\mathrm{src}}, \beta_{\mathrm{frz}})$ and $b :=
  b_{\mathrm{src}} + b_{\mathrm{frz}}$ we obtain the master iteration
  inequality
  \begin{equation}
    \label{eq:aniso_master_inequality} \psi (\jth \rr) \leqslant \jth^m \jgam
    \psi (\rr) + b \rr^{\beta} \qquad \forall \rr \in (0, \rr_0] .
  \end{equation}
  {\noindent}\text{{\itshape{Step~2: Continuous algebraic decay via the
  iteration lemma.}}} Identical to Step~2 of the proof of
  Theorem~\ref{thm:gen_abstract_reg}. Fix an arbitrarily small $\varepsilon >
  0$ and set $\jgam_{\varepsilon} := \max (\jgam, \jth^{\beta - m -
  \varepsilon})$. With $\jalfa_{\varepsilon} := m + \min (\ln \jgam / \ln
  \jth, \beta - m - \varepsilon)$, we have $\beta - \jalfa_{\varepsilon}
  \geqslant \varepsilon > 0$, and by Lemma~\ref{lem:giusti_iteration} there
  exists a structural constant $c_{\varepsilon} > 0$ such that for all $\rr
  \in (0, \rr_0]$,
  \begin{equation}
    \label{eq:aniso_continuous_decay} \Osc (\bm{u},
    B_{\rr} (x_0)) \leqslant c_{\varepsilon}  \left[ \frac{\Osc
    (\bm{u}, B_{\rr_0}
    (x_0))}{\rr_0^{\jalfa_{\varepsilon} - m}} + b \rr_0^{\beta -
    \jalfa_{\varepsilon}} \right] \rr^{\jalfa_{\varepsilon} - m} .
  \end{equation}
  {\noindent}\text{{\itshape{Step~3: Uniform bound via
  $\|\bm{u}\|_{L^{\infty}} \leqslant M$.}}} Choosing $\jc
  =\ensuremath{\boldsymbol{0}}$ in the infimum defining the oscillation $\Osc
  (\bm{u}, B_{\rr_0} (x_0))$ yields $\Osc
  (\bm{u}, B_{\rr_0} (x_0)) \leqslant M^{\pHold}$.
  Substituting into~\eqref{eq:aniso_continuous_decay} produces
  \begin{equation}
    \Osc (\bm{u}, B_{\rr} (x_0)) \leqslant M_{\ast}
    \rr^{\jalfa_{\varepsilon} - m},
  \end{equation}
  for a constant $M_{\ast}$ independent of $x_0 \in B_{1 / 2}$ and $\rr \in
  (0, \rr_0]$, with
  \begin{equation}
    \jalfa_{\varepsilon} - m = \min \left( \frac{\ln \gamma}{\ln \jth}, \pHold
    (2 - m / \jqq) - \varepsilon, \pHold (\gamma_0 + (2 - m) / 2) -
    \varepsilon \right) .
  \end{equation}
  {\noindent}\text{{\itshape{Step~4: Campanato condition and Hölder continuity
  on $B_{1 / 2}$.}}} By Campanato's theorem, $\bm{u} \in
  C^{0, \jeta_{\varepsilon}} (B_{1 / 2})$ with $\jeta_{\varepsilon} =
  (\jalfa_{\varepsilon} - m) / \pHold$. Using our formula for
  $(\jalfa_{\varepsilon} - m)$ yields:
  \begin{equation}
    \jeta_{\varepsilon} = \min \left( \frac{\ln \gamma}{\pHold \ln \jth}, 2 -
    \frac{m}{\jqq} - \frac{\varepsilon}{\pHold}, \gamma_0 + \frac{2 - m}{2} -
    \frac{\varepsilon}{\pHold} \right) .
  \end{equation}
  Sending $\varepsilon \to 0^+$ recovers~\eqref{eq:aniso_holder_exponent} with
  the third entry $\gamma_0 + \frac{2 - m}{2}$ for the general-dimension
  statement.
  
  {\noindent}\text{{\itshape{Step~5: Extension to arbitrary compact
  subsets.}}} Identical to Step~5 in the proof of
  Theorem~\ref{thm:gen_abstract_reg}.
\end{proof}

\subsection{Anisotropic Caccioppoli-type estimate in dimension \texorpdfstring{$m
= 2$}{two}}\label{ssec:aniso_caccioppoli} To apply the abstract regularity principle
(Theorem~\ref{thm:aniso_abstract_reg}) we must show that solutions of
\eqref{eq:aniso_main} inherently control the oscillation at small scales. We
achieve this by locally restructuring the equation in divergence form, with a
coefficient-freezing step that converts $- \divg (A \nabla \cdot)$ into the
constant-coefficient operator $- \divg (A_0 \nabla \cdot)$ at the cost of an
explicit freezing flux.

Fix $x_0 \in B_1$ (which we may translate to the origin). Set $A_0 := A (x_0)$
and $\tilde{A} (x) := A (x) - A_0$. Then $A_0$ is symmetric, $\lambda I
\leqslant A_0 \leqslant \Lambda I$, and \eqref{eq:tildeA_bound} holds. For any
constant vector $\jc \in \RR^{n + 1}$, the Leibniz rule gives
\begin{equation}
  \jaOm \cdot \nabla \bm{u}= \jaOm \cdot \nabla
  (\bm{u}- \jc) = \divg \left( \jaOm
  (\bm{u}- \jc) \right) - (\divg \jaOm) 
  (\bm{u}- \jc) .
\end{equation}
Substituting into \eqref{eq:aniso_main} and adding/subtracting $\divg (A_0
\nabla \bm{u})$, we rewrite the equation as
\begin{equation}
  \label{eq:rewritten} - \divg (A_0 \nabla \bm{u}) =
  \divg \underbrace{\left( \tilde{A} \nabla \bm{u}+ \jaOm
  (\bm{u}- \jc) \right)}_{= :
  \bm{J}} + \underbrace{\bm{f}-
  (\divg \jaOm)  (\bm{u}- \jc)}_{= :
  \bm{F}} .
\end{equation}
The flux $\bm{J}$ has two natural sub-contributions: the
\text{{\itshape{convective}}} term $\jaOm (\bm{u}- \jc)$,
already present in the isotropic case, and the new
\text{{\itshape{freezing}}} term $\tilde{A} \nabla
\bm{u}$, which vanishes algebraically as $r \to 0$ by
\eqref{eq:tildeA_bound}.

{\myprop{\begin{lemma}[Anisotropic Caccioppoli-type estimate, dimension $m =
2$]
  \label{lem:reverse_sobolev_aniso}Let $\jqq \in (1, 2)$, and let $\pHold := 2
  \jqq / (2 - \jqq) \in (2, \infty)$ be the corresponding oscillation
  exponent. Assume \text{{\upshape{(A1)}}}--\text{{\upshape{(A4)}}}. Let
  $\bm{u} \in H^1 \cap L^{\infty} (B_1, \RR^{n + 1})$ be
  a weak solution of {{\em \eqref{eq:aniso_main}}} with
  $\|\bm{u}\|_{L^{\infty} (B_1)} \leqslant M$. Fix $\rr
  \in (0, 1]$, and let $\ensuremath{\boldsymbol{h}}_r \in H^1 (B_{\rr}, \RR^{n
  + 1})$ be the unique $A_0$-harmonic extension of
  $\bm{u}_{\restr \partial B_{\rr}}$:
  \begin{equation} - \divg (A_0 \nabla \ensuremath{\boldsymbol{h}}_{\rr})
     =\ensuremath{\boldsymbol{0}} \quad \text{in } B_r, \qquad
     \ensuremath{\boldsymbol{h}}_{\rr} =\bm{u} \quad
     \text{on } \partial B_{\rr} . \end{equation}
  Set $A_0 := A (0)$ and $\tilde{A} := A - A_0$. Then
  \begin{align}
    \| \nabla
    (\bm{u}-\ensuremath{\boldsymbol{h}}_r)\|_{L^{\jqq}
    (B_{\rr})} & \leqslant  C_0 \| \jaOm \|_{L^2 (B_{\rr})}  \rr^{2 / \pHold}
    \Osc^{1 / \pHold} (\bm{u}, B_{\rr}) + C_0 \|
    \tilde{A} \|_{L^{\infty} (B_{\rr})}  \rr^{2 / \pHold}  \| \nabla
    \bm{u}\|_{L^2 (B_{\rr})} \nonumber\\
    &   \qquad + C_0  \rr \left( \|\bm{f}\|_{L^{\jqq}
    (B_{\rr})} + 2 M\| \divg \jaOm \|_{L^{\jqq} (B_{\rr})} \right), 
    \label{eq:rs_aniso}
  \end{align}
  where $C_0 = C_0  (m, \jqq, \lambda, \Lambda)$ depends only on the
  dimension, the integrability exponent, and the ellipticity bounds. In
  particular, $C_0$ does not depend on the modulus of continuity of $A$, on
  $\rr$, or on $\bm{u}$.
\end{lemma}}}

\begin{proof}
  Since $\|\bm{u}\|_{L^{\infty}} \leqslant M$, by
  Lemma~\ref{lemma:osc_infimum_bound} we may restrict to constants $\jc \in
  \RR^{n + 1}$ with $| \jc | \leqslant M$. Fix such an optimal $\jc =
  \jc_{\ast}$ in $\Osc (\bm{u}, B_r)$; then
  $|\bm{u}- \jc | \leqslant 2 M$ pointwise, and
  $\|\bm{u}- \jc \|_{L^{\pHold} (B_r)} = \omega_m^{1 /
  \pHold}  \rr^{2 / \pHold} \Osc^{1 / \pHold} (\bm{u},
  B_r)$.
  
  Set $\bm{w} :=
  \bm{u}-\ensuremath{\boldsymbol{h}}_{\rr} \in H^1_0
  (B_{\rr}, \RR^{n + 1})$. Since $\divg (A_0 \nabla
  \ensuremath{\boldsymbol{h}}_{\rr}) =\ensuremath{\boldsymbol{0}}$ ($A_0$
  constant, $\ensuremath{\boldsymbol{h}}_{\rr}$ is $A_0$-harmonic),
  subtracting from \eqref{eq:rewritten} yields the constant-coefficient
  Dirichlet problem
  \begin{equation}
    \left\{ \begin{array}{lcl}
      - \divg (A_0 \nabla \bm{w}) = \divg
      \bm{J}+\bm{F} &  & \text{in }
      B_{\rr},\\
      \bm{w}=\bm{0} &  & \text{on }
      \partial B_{\rr},
    \end{array} \right.
  \end{equation}
  with $\bm{J}, \bm{F}$ as in
  \eqref{eq:rewritten}.
  By linearity we can decompose
  $\bm{w}=\bm{w}_{\bm{J}}
  +\bm{w}_{\bm{F}}$, with $\bm{w}_{\bm{J}}$ and $
  \bm{w}_{\bm{F}}$ the unique solutions of the Poisson problems $- \divg
  (A_0 \nabla \bm{w}_{\bm{J}}) =
  \divg \bm{J}$ and $- \divg (A_0 \nabla
  \bm{w}_{\bm{F}})
  =\bm{F}$, both with zero boundary data on $\partial
  B_{\rr}$.{\smallskip}
  
  {\noindent}{\itshape{Step 1: Calderón--Zygmund bound for the flux
  part.}} By the divergence-form Calderón--Zygmund estimate for the
  constant-coefficient operator $- \divg (A_0 \nabla \cdot)$ on the ball
  $B_{\rr}$ (which is scale-invariant in the form of the source),
  \begin{equation}
    \label{eq:CZ_J} \| \nabla
    \bm{w}_{\bm{J}} \|_{L^{\jqq}
    (B_{\rr})} \leqslant C_q  (m, \jqq, \lambda, \Lambda)
    \|\bm{J}\|_{L^{\jqq} (B_{\rr})} .
  \end{equation}
  {\noindent}{\itshape{Step 2: Source-to-divergence representation.}}
  We re-express $\bm{F} \in L^{\jqq}$ as a divergence:
  solve $- \Delta
  \bm{\varphi}=\bm{F}$ in $B_{\rr}$
  with $\bm{\varphi}=\ensuremath{\boldsymbol{0}}$ on
  $\partial B_{\rr}$. Standard $W^{2, \jqq}$-theory gives $\|D^2
  \bm{\varphi}\|_{L^{\jqq} (B_{\rr})} \leqslant C
  \|\bm{F}\|_{L^{\jqq} (B_{\rr})}$. Because
  $\bm{\varphi}$ vanishes on the boundary, the divergence
  theorem guarantees that its gradient has a mean of zero over $B_{\rr}$.
  Applying the Poincaré--Wirtinger inequality to $\nabla
  \bm{\varphi}$ yields $\| \nabla
  \bm{\varphi}\|_{L^{\jqq} (B_{\rr})} \leqslant C \rr 
  \|D^2 \bm{\varphi}\|_{L^{\jqq} (B_{\rr})}$.
  Setting
  $\bm{G} := \nabla \bm{\varphi}$,
  we have $\divg \bm{G}= -\bm{F}$
  with
  \begin{equation}
    \label{eq:G_bound} \|\bm{G}\|_{L^{\jqq} (B_{\rr})}
    \leqslant C \rr \|\bm{F}\|_{L^{\jqq} (B_{\rr})} .
  \end{equation}
  Applying~\eqref{eq:CZ_J} to the divergence-form problem $- \divg (A_0 \nabla
  \bm{w}_{\bm{F}}) = \divg
  (-\bm{G})$,
  \begin{equation}
    \label{eq:CZ_F} \| \nabla
    \bm{w}_{\bm{F}} \|_{L^{\jqq}
    (B_{\rr})} \leqslant C_q  \|-\bm{G}\|_{L^{\jqq}
    (B_{\rr})} \leqslant C \rr \|\bm{F}\|_{L^{\jqq}
    (B_{\rr})} .
  \end{equation}
  {\noindent}\text{{\itshape{Step 3: Bounds on the flux pieces.}}} Combining
  \eqref{eq:CZ_J} and \eqref{eq:CZ_F},
  \begin{equation}
    \label{eq:CZ_constcoeff} \| \nabla \bm{w}\|_{L^{\jqq}
    (B_{\rr})} \leqslant C \left[ \| \tilde{A} \nabla
    \bm{u}\|_{L^{\jqq} (B_{\rr})} +\| \jaOm
    (\bm{u}- \jc)\|_{L^{\jqq} (B_{\rr})} + \rr
    \|\bm{F}\|_{L^{\jqq} (B_{\rr})} \right] .
  \end{equation}
  We bound each piece via Hölder's inequality. By Hölder with exponents $(2,
  \pHold)$, $1 / \jqq = 1 / 2 + 1 / \pHold$:
  \begin{equation}
    \label{eq:flux_convective} \| \jaOm (\bm{u}-
    \jc)\|_{L^{\jqq} (B_{\rr})} \leqslant \| \jaOm \|_{L^2 (B_{\rr})} 
    \|\bm{u}- \jc \|_{L^{\pHold} (B_{\rr})} = \omega_m^{1
    / \pHold} \| \jaOm \|_{L^2 (B_{\rr})}  \rr^{2 / \pHold} \Osc^{1 / \pHold}
    (\bm{u}, B_{\rr}) .
  \end{equation}
  A second Hölder with the constant function $1$ and exponents $(2 / \jqq, 2 /
  (2 - \jqq))$ (since $\jqq < 2$), together with $|B_{\rr} |^{1 / \pHold} =
  \omega_m^{1 / \pHold}  \rr^{2 / \pHold}$, gives
  \begin{equation}
    \| \nabla \bm{u}\|_{L^{\jqq} (B_{\rr})} \leqslant \|
    \nabla \bm{u}\|_{L^2 (B_{\rr})} |B_{\rr} |^{1 /
    \pHold} = \omega_m^{1 / \pHold}  \rr^{2 / \pHold}  \| \nabla
    \bm{u}\|_{L^2 (B_{\rr})} .
  \end{equation}
  Hence
  \begin{equation}
    \label{eq:flux_freezing} \| \tilde{A} \nabla
    \bm{u}\|_{L^{\jqq} (B_{\rr})} \leqslant \omega_m^{1 /
    \pHold} \| \tilde{A} \|_{L^{\infty} (B_{\rr})}  \rr^{2 / \pHold}  \|
    \nabla \bm{u}\|_{L^2 (B_{\rr})} .
  \end{equation}
  Pointwise $|\bm{F}| \leqslant
  |\bm{f}| + 2 M | \divg \jaOm |$, hence
  \begin{equation}
    \label{eq:flux_source} \|\bm{F}\|_{L^{\jqq}
    (B_{\rr})} \leqslant \|\bm{f}\|_{L^{\jqq} (B_{\rr})}
    + 2 M \| \divg \jaOm \|_{L^{\jqq} (B_{\rr})} .
  \end{equation}
  Substituting \eqref{eq:flux_convective}--\eqref{eq:flux_source} into
  \eqref{eq:CZ_constcoeff} and absorbing the universal factor $\omega_m^{1 /
  \pHold}$ into the constant $C_0 = C_0  (m, \jqq, \lambda, \Lambda)$ yields
  \eqref{eq:rs_aniso}.
\end{proof}

We now extend the harmonic approximation lemma previously established in the
antisymmetric setting (Lemma~\ref{lemma:harmonic_approx}) to the class of
$A_0$-harmonic functions. To this end, we rely on estimates for the
anisotropic Poisson kernel $P_{A_0} (x, \zeta)$, established in
Proposition~\ref{prop:aniso_poisson_kernel} of the Appendix, Section
\ref{sec:app1}.

{\myprop{\begin{lemma}[Anisotropic harmonic approximation]
  \label{lemma:harmonic_approx_aniso}Let $A_0 \in \RR^{m \times m}$ be a
  constant symmetric matrix satisfying the uniform ellipticity condition
  $\lambda I \leqslant A_0 \leqslant \Lambda I$. Let
  $\bm{u} \in H^1 (B_1, \RR^{n + 1})$.
  
  For any $p > 1$, there exists a radius $\rr \in [1 / 2, 1]$ and an
  $A_0$-harmonic function $\ensuremath{\boldsymbol{h}}: B_{\rr} \to \RR^{n +
  1}$ satisfying the boundary trace
  $\ensuremath{\boldsymbol{h}}=\bm{u}$ on $\partial
  B_{\rr}$, such that the interior gradient satisfies:
  \begin{equation}
    \sup_{x \in B_{1 / 4}} | \nabla \ensuremath{\boldsymbol{h}}(x) | \leqslant
    C_A \cdot \Osc^{1 / p} (\bm{u}, B_1),
    \label{eq:harmapproxosc_aniso}
  \end{equation}
  where $C_A > 0$ is a dimensional constant depending  on $m$,
  $\lambda$, and $\Lambda$.
\end{lemma}}}

\begin{proof}
  {{\em Step 1: Selection of a good radius.}} Let $\jc \in \RR^{n + 1}$ be a
  constant vector that realizes this minimum for $\Osc
  (\bm{u}, B_1)$:
  \begin{equation}
    \jav{B_1} |\bm{u}- \jc |^p = \Osc
    (\bm{u}, B_1) .
  \end{equation}
  As explained in the proof of Lemma~\ref{lemma:harmonic_approx}, Fubini's
  theorem in polar coordinates guarantees the existence of a radius $\rr \in
  [1 / 2, 1]$ such that the boundary integral is controlled by the bulk
  integral:
  \begin{equation}
    \int_{\partial B_{\rr}} \left| \bm{u}(\zeta) - \jc
    \right| \mathrm{d} \zeta \leqslant 2 \int_{B_1} \left|
    \bm{u}(x) - \jc \right| \mathrm{d} x.
    \label{eq:boundtodom_aniso}
  \end{equation}
  Note that $\rr$ depends on the optimal $\jc$, which in turn depends 
  on $\bm{u}$ and $p$.
  
  {\noindent}{{\em Step 2: Anisotropic harmonic extension and interior
  estimates.}} We fix the radius $\rr$ found in Step 1. Let
  $\ensuremath{\boldsymbol{h}}$ be the unique weak solution to $\divg (A_0
  \nabla \ensuremath{\boldsymbol{h}}) = 0$ in $B_{\rr}$ satisfying the
  boundary condition $\ensuremath{\boldsymbol{h}}=\bm{u}$
  on $\partial B_{\rr}$. Since $\jc$ is a constant vector, the translated
  function $\ensuremath{\boldsymbol{h}}- \jc$ is also $A_0$-harmonic in
  $B_{\rr}$. We estimate $\ensuremath{\boldsymbol{h}}- \jc$  inside
  the ball using the Poisson integral representation for the Euclidean ball
  $B_{\rr}$ associated with the anisotropic operator. By elliptic estimates
  for the anisotropic Poisson kernel $P_{A_0} (x, \zeta)$ evaluated at
  interior points $x \in B_{\rr / 2}$  (see Proposition~\ref{prop:aniso_poisson_kernel}) we get
  \begin{equation}
    | \nabla \ensuremath{\boldsymbol{h}}(x) | = | \nabla
    (\ensuremath{\boldsymbol{h}}(x) - \jc) | \leqslant
    \frac{\tilde{C}_A}{\rr^m} \int_{\zeta \in \partial B_{\rr}} \left|
    \bm{u}(\zeta) - \jc \right| \mathrm{d} \zeta,
    \label{eq:Poissonbound_aniso}
  \end{equation}
  with $\tilde{C}_A := C_{\theta, \lambda, \Lambda, m}$ being the constant in
  Proposition~\ref{prop:aniso_poisson_kernel} (applied with $\jth = 1 / 2$).
  
  We now restrict the domain to the fixed interior ball. Since $\rr \geqslant
  1 / 2$, we have $B_{1 / 4} \subset B_{\rr / 2}$ and that $1 / \rr^m
  \leqslant 2^m$. Combining these geometric bounds with the radius estimate
  \eqref{eq:boundtodom_aniso}, we obtain:
  
  \begin{align}
    \sup_{x \in B_{1 / 4}} | \nabla \ensuremath{\boldsymbol{h}}(x) | &
    \leqslant \sup_{x \in B_{\rr / 2}} | \nabla \ensuremath{\boldsymbol{h}}(x)
    | \overset{\eqref{eq:Poissonbound_aniso}}{\leqslant}
    \frac{\tilde{C}_A}{\rr^m}  \int_{\zeta \in \partial B_{\rr}} \left|
    \bm{u}(\zeta) - \jc \right| \mathrm{d} \zeta
    \nonumber\\
    & \overset{\eqref{eq:boundtodom_aniso}}{\leqslant} \tilde{C}_A 2^m 
    \left( 2 \int_{B_1} |\bm{u}- \jc | \hspace{0.17em}
    \mathrm{d} x \right) \leqslant 2^{m + 1}  \tilde{C}_A \omega_m \left(
    \jav{B_1} \left| \bm{u}- \jc \right|^p \mathrm{d} x
    \right)^{1 / p}, 
  \end{align}
  
  where $\omega_m$ denotes the volume of the unit ball in $\RR^m$.
  
  Since the constant $\jc$ was chosen precisely to realize the minimum in the
  definition of the oscillation, the right-hand side is exactly $\Osc^{1 / p}
  (\bm{u}, B_1)$. Setting $C_A := 2^{m + 1}  \tilde{C}_A
  \omega_m$ formally yields the gradient estimate
  \eqref{eq:harmapproxosc_aniso}, concluding the proof.
\end{proof}

We now combine the anisotropic Caccioppoli estimate
(established in Lemma~\ref{lem:reverse_sobolev_aniso}) with the anisotropic harmonic
approximation result (Lemma~\ref{lemma:harmonic_approx_aniso}) to derive the
following estimate.

{\myprop{\begin{lemma}[Anisotropic coupled Caccioppoli-type estimate]
  \label{lem:coupled_aniso}Under the assumptions and notation of Lemma~{{\em
  \ref{lem:reverse_sobolev_aniso}}}, there exists an $A_0$-harmonic function
  $\ensuremath{\boldsymbol{h}}: B_{1 / 4} \to \RR^{n + 1}$ such that
  \begin{align}
    \| \nabla
    (\bm{u}-\ensuremath{\boldsymbol{h}})\|_{L^{\jqq}
    (B_{1 / 4})} & \leqslant  C_0 \| \jaOm \|_{L^2 (B_1)} \Osc^{1 / \pHold}
    (\bm{u}, B_1) + C_0 \| \tilde{A} \|_{L^{\infty}
    (B_1)}  \| \nabla \bm{u}\|_{L^2 (B_1)} \nonumber\\
    &  \quad  \qquad \qquad \qquad + C_0 \left(
    \|\bm{f}\|_{L^{\jqq} (B_1)} + 2 M\| \divg \jaOm
    \|_{L^{\jqq} (B_1)} \right),  \label{eq:coupled_aniso}
  \end{align}
  with $C_0 = C_0  (m, \jqq, \lambda, \Lambda)$. Moreover,
  \begin{equation}
    \label{eq:harm_grad_aniso} \sup_{x \in B_{1 / 4}} | \nabla
    \ensuremath{\boldsymbol{h}}(x) | \leqslant C_A \Osc^{1 / \pHold}
    (\bm{u}, B_1),
  \end{equation}
  where $C_A > 0$ is a constant depending only on $m$, $\lambda$, and
  $\Lambda$.
\end{lemma}}}

\begin{proof}
  By the Anisotropic Harmonic Approximation
  (Lemma~\ref{lemma:harmonic_approx_aniso}) applied on the unit ball with
  exponent $\pHold$, there exists a radius $\rr_{\ast} \in [1 / 2, 1]$ and an
  $A_0$-harmonic extension $\ensuremath{\boldsymbol{h}}: B_{\rr_{\ast}} \to
  \RR^{n + 1}$ such that
  $\ensuremath{\boldsymbol{h}}=\bm{u}$ on $\partial
  B_{\rr_{\ast}}$. This extension satisfies the interior gradient bound
  \eqref{eq:harm_grad_aniso} since $B_{1 / 4} \subset B_{\rr_{\ast} / 2}$.
  Also, we may apply the estimate from Lemma~\ref{lem:reverse_sobolev_aniso}
  directly on the ball $B_{\rr_{\ast}}$. This yields an intermediate constant
  $C_0$ such that:
  
  \begin{align}
    \| \nabla
    (\bm{u}-\ensuremath{\boldsymbol{h}})\|_{L^{\jqq}
    (B_{\rr_{\ast}})} & \leqslant C_0  \rr_{\ast}^{2 / \pHold} \| \jaOm
    \|_{L^2 (B_{\rr_{\ast}})} \Osc^{1 / \pHold} (\bm{u},
    B_{\rr_{\ast}}) + C_0 \| \tilde{A} \|_{L^{\infty} (B_{\rr_{\ast}})}  \|
    \nabla \bm{u}\|_{L^2 (B_{\rr_{\ast}})} \nonumber\\
    & \quad + C_0 \left( \|\bm{f}\|_{L^{\jqq}
    (B_{\rr_{\ast}})} + 2 M\| \divg \jaOm \|_{L^{\jqq} (B_{\rr_{\ast}})}
    \right) .  \label{eq:intermediate_caccioppoli}
  \end{align}
  
  To bridge \eqref{eq:intermediate_caccioppoli} to our target estimate
  \eqref{eq:coupled_aniso}, we systematically bound the terms on the
  right-hand side using the monotonicity of the data norms (e.g., $\| \jaOm
  \|_{L^2 (B_{\rr_{\ast}})} \leqslant \| \jaOm \|_{L^2 (B_1)}$ and likewise
  for $\| \tilde{A} \|_{L^{\infty}}$, $\| \nabla
  \bm{u}\|_{L^2}$,
  $\|\bm{f}\|_{L^{\jqq}}$, $\| \divg \jaOm
  \|_{L^{\jqq}}$), and the monotonicity $\rr^m \Osc
  (\bm{u}, B_{\rr}) \leqslant \Osc
  (\bm{u}, B_1)$ of the rescaled oscillation. Restricting
  the LHS from $\| \nabla
  (\bm{u}-\ensuremath{\boldsymbol{h}})\|_{L^{\jqq}
  (B_{\rr_{\ast}})}$ to $\| \nabla
  (\bm{u}-\ensuremath{\boldsymbol{h}})\|_{L^{\jqq} (B_{1
  / 4})}$ (since $B_{1 / 4} \subset B_{\rr_{\ast}}$) yields
  \eqref{eq:coupled_aniso}.
\end{proof}

\subsection{The anisotropic decay estimate}\label{ssec:aniso_decay} We now show
that the structural matrix $\jaOm$ dictates the decay of the oscillation,
provided its $L^2$-energy is sufficiently small. The freezing flux enters as a
fixed additive contribution and \text{{\itshape{does not}}} require a
smallness threshold; its smallness on small balls is automatic from
\eqref{eq:tildeA_bound} and is exploited in the global iteration of
Subsection~\ref{ssec:aniso_proof_main}.

{\myprop{\begin{lemma}[Anisotropic decay property]
  \label{lem:decay_aniso}For any integrability exponent $\jqq \in (1, 2)$, let
  $\pHold := 2 \jqq / (2 - \jqq)$. There exist structural constants $\jth \in
  (0, 1 / 4)$, $\jgam = 1 / 2$, $\kappa > 0$, and a threshold
  $\varepsilon_{\ast} > 0$ {\opt}depending only on $m, \jqq, \lambda,
  \Lambda$, with $\kappa$ also depending on $M${\cpt} such that the following
  holds. Let $\bm{u} \in H^1 \cap L^{\infty} (B_1, \RR^{n
  + 1})$ be a weak solution of \eqref{eq:aniso_main} under
  \text{{\upshape{(A1)}}}--\text{{\upshape{(A4)}}} with
  $\|\bm{u}\|_{L^{\infty} (B_1)} \leqslant M$ for some $M
  > 0$. If
  \begin{equation}
    \label{eq:smallness_aniso} \| \jaOm \|_{L^2 (B_1)} < \varepsilon_{\ast},
  \end{equation}
  then
  \begin{equation}
    \label{eq:decay_aniso} \Osc (\bm{u}, B_{\jth})
    \leqslant \jgam \Osc (\bm{u}, B_1) + \kappa \left(
    D_{\mathrm{src}}^{\pHold} + D_{\mathrm{frz}}^{\pHold} \right),
  \end{equation}
  where $D_{\mathrm{src}}$ and $D_{\mathrm{frz}}$ are defined in {{\em
  \eqref{eq:def_Dsrc_Dfrz}}}.
\end{lemma}}}

\begin{remark}[Asymmetry between the two smallness mechanisms]
  Only the smallness of the norm $\| \jaOm \|_{L^2 (B_1)}$ is required. The freezing
  data $D_{\mathrm{frz}} = \| \tilde{A} \|_{L^{\infty} (B_1)}  \| \nabla
  \bm{u}\|_{L^2 (B_1)}$ enters as fixed additive
  data: no smallness threshold is imposed. Structurally, in the master
  inequality \eqref{eq:master_aniso} below, $\| \jaOm \|_{L^2}^{\pHold}$
  multiplies $\Osc (\bm{u}, B_1)$ (so it must be small
  for the contraction to operate), whereas $\| \tilde{A}
  \|_{L^{\infty}}^{\pHold}  \| \nabla
  \bm{u}\|_{L^2}^{\pHold}$ appears only as an additive
  contribution. The eventual smallness of $\| \tilde{A} \|_{L^{\infty}}$ on
  small balls, needed in the global iteration of
  Subsection~\ref{ssec:aniso_proof_main}, comes for free from
  \eqref{eq:tildeA_bound} and the rescaling \eqref{eq:scale_Holder}.
\end{remark}

\begin{proof}
  To distinguish the exponent originating from $L^2$-integrability from the
  spatial dimension, we temporarily denote the latter by $m$ (here $m = 2$).
  By Lemma~\ref{lem:coupled_aniso}, there exists a harmonic
  $\ensuremath{\boldsymbol{h}}: B_{1 / 4} \to \RR^{n + 1}$ satisfying
  \eqref{eq:coupled_aniso} and \eqref{eq:harm_grad_aniso}. For $\jth \in (0, 1
  / 4)$, the triangle inequality for the seminorm $\Osc^{1 / \pHold}$ yields
  \begin{equation}
    \label{eq:split_aniso} \Osc^{1 / \pHold} (\bm{u},
    B_{\jth}) \leqslant \Osc^{1 / \pHold}
    (\bm{u}-\ensuremath{\boldsymbol{h}}, B_{\jth}) +
    \Osc^{1 / \pHold} (\ensuremath{\boldsymbol{h}}, B_{\jth}) =: I + I  I.
  \end{equation}
  {\noindent}\text{{\itshape{Bound on $I  I$.}}} Choosing the constant
  $\ensuremath{\boldsymbol{h}} (0)$ in the infimum and using the fundamental
  theorem of calculus on the segment $[0, x] \subset B_{\jth}$,
  \begin{equation} I  I \leqslant \left( \jav{B_{\jth}} |\ensuremath{\boldsymbol{h}}(x)
     -\ensuremath{\boldsymbol{h}}(0) |^{\pHold} \mathrm{d} x \right)^{1 /
     \pHold} \leqslant \jth \sup_{B_{1 / 4}} | \nabla
     \ensuremath{\boldsymbol{h}}|
     \overset{\eqref{eq:harm_grad_aniso}}{\leqslant} C_A \jth  \Osc^{1 /
     \pHold} (\bm{u}, B_1) . \end{equation}
  {\noindent}\text{{\itshape{Bound on $I$.}}} Since
  $\bm{u}-\ensuremath{\boldsymbol{h}}$ does not vanish on
  $\partial B_{\jth}$, we apply the (mean-subtracted) Sobolev--Poincaré
  inequality on $B_{\jth}$ at the critical exponents $1 / \pHold = 1 / \jqq -
  1 / 2$, which is scale-invariant in dimension $m = 2$:
  \begin{equation}
    I \leqslant \frac{C_S}{(\omega_m \jth^m)^{1 / \pHold}}  \| \nabla
    (\bm{u}-\ensuremath{\boldsymbol{h}})\|_{L^{\jqq}
    (B_{\jth})} .
  \end{equation}
  Bounding $\| \nabla
  (\bm{u}-\ensuremath{\boldsymbol{h}})\|_{L^{\jqq}
  (B_{\jth})} \leqslant \| \nabla
  (\bm{u}-\ensuremath{\boldsymbol{h}})\|_{L^{\jqq} (B_{1
  / 4})}$ and invoking \eqref{eq:coupled_aniso}, we obtain (after harmlessly
  redefining the constant $C_0$ to absorb $C_S$ and $\omega_m^{- 1 /
  \pHold}$):
  \begin{eqnarray}
    I & \leqslant & \frac{C_0}{\jth^{m / \pHold}} \| \jaOm \|_{L^2 (B_1)}
    \Osc^{1 / \pHold} (\bm{u}, B_1) + \frac{C_0}{\jth^{m
    / \pHold}} \| \tilde{A} \|_{L^{\infty} (B_1)}  \| \nabla
    \bm{u}\|_{L^2 (B_1)} \nonumber\\
    &  & \qquad \qquad \qquad \qquad \qquad + \frac{C_0}{\jth^{m / \pHold}}
    \left( \|\bm{f}\|_{L^{\jqq} (B_1)} + 2 M\| \divg
    \jaOm \|_{L^{\jqq} (B_1)} \right) .  \label{eq:Ibound_aniso}
  \end{eqnarray}
  {\noindent}\text{{\itshape{Master inequality.}}} Combining
  \eqref{eq:split_aniso} with the bounds on $I$ and $I  I$ and raising to the
  $\pHold$-th power and using that $(a_1 + \ldots + a_j)^{\pHold} \leqslant
  j^{\pHold - 1}  \sum_{i = 1}^j a_i^{\pHold}$, we obtain the master decay
  inequality
  \begin{align}
    \Osc (\bm{u}, B_{\jth}) & \leqslant K_1
    \jth^{\pHold} \Osc (\bm{u}, B_1) + \frac{K_2}{\jth^m}
    \| \jaOm \|_{L^2 (B_1)}^{\pHold} \Osc (\bm{u}, B_1)
    \nonumber\\
    &   \quad + \frac{K_2}{\jth^m}  \left[ \| \tilde{A} \|_{L^{\infty}
    (B_1)}^{\pHold} \| \nabla \bm{u}\|_{L^2
    (B_1)}^{\pHold} + \left( \|\bm{f}\|_{L^{\jqq} (B_1)}
    + 2 M\| \divg \jaOm \|_{L^{\jqq} (B_1)} \right)^{\pHold} \right], 
    \label{eq:master_aniso}
  \end{align}
  with $K_1 = 2^{\pHold - 1} C_A^{\pHold}$ and $K_2 = 2^{\pHold - 1} \cdot
  4^{\pHold - 1} C_0^{\pHold}$ depending only on $m, \jqq, \lambda, \Lambda,
  \pHold$.
  
  {\noindent}\text{{\itshape{Choice of parameters.}}} The contractive
  coefficient of $\Osc (\bm{u}, B_1)$ in
  \eqref{eq:master_aniso} is given by $K_1 \jth^{\pHold} + (K_2 / \jth^m) \| \jaOm
  \|_{L^2 (B_1)}^{\pHold}$. We make it less than or equal to $1 / 2$ in two
  steps. First, choose $\jth \in (0, 1 / 4)$ small enough that $K_1
  \jth^{\pHold} \leqslant 1 / 4$. Second, with this $\jth$ fixed, choose
  $\varepsilon_{\ast} > 0$ such that $K_2 \varepsilon_{\ast}^{\pHold} / \jth^m
  \leqslant 1 / 4$. Under \eqref{eq:smallness_aniso}, the contractive
  coefficient is at most $1 / 2$. The remaining additive terms are bounded by
  $K_2 \jth^{- m} (D_{\mathrm{frz}}^{\pHold} + (2
  D_{\mathrm{src}})^{\pHold})$, recalling
  $\|\bm{f}\|_{L^{\jqq} (B_1)} + 2 M \| \divg \jaOm
  \|_{L^{\jqq} (B_1)} \leqslant 2 D_{\mathrm{src}}$. Setting $\kappa :=
  2^{\pHold} K_2 \jth^{- m}$ yields \eqref{eq:decay_aniso}.
\end{proof}

\subsection{Proof of the main regularity theorem
(Theorem~\ref{thm:aniso_regularity})}\label{ssec:aniso_proof_main} Let
$\bm{u}$ be an arbitrary weak solution of
\eqref{eq:aniso_main} under (A1)--(A4) with $\bm{u} \in
H^1 \cap L^{\infty} (B_1, \RR^{n + 1})$. We fix the global amplitude bound $M
:= \max (1, \|\bm{u}\|_{L^{\infty} (B_1)})$.

Let $\varepsilon_{\ast}$ be the critical threshold from
Lemma~\ref{lem:decay_aniso}. To invoke Theorem~\ref{thm:aniso_abstract_reg},
we define the admissible class $\class$ as the set of quadruples
$(\bm{u}, \jaOm, \bm{f}, A) \in
\mathcal{Y}$, $\mathcal{Y}$ defined in \eqref{eq:defmathcalY}, satisfying the
original system, the amplitude bound $M$, and the smallness threshold:
\begin{equation}
  \class := \left\{ (\bm{u}, \jaOm,
  \bm{f}, A) \in \mathcal{Y} \st
   \begin{array}{l}
    - \divg (A \nabla \bm{u}) = \jaOm \cdot \nabla
    \bm{u}+\bm{f},\\
    \|\bm{u}\|_{L^{\infty} (B_1)} \leqslant M, \| \jaOm
    \|_{L^2 (B_1)} < \varepsilon_{\ast}
  \end{array} \right\} .
\end{equation}
We need to verify {\jaxs}~\ref{ax:closure_aniso} and \ref{ax:decay_aniso}
required by Theorem~\ref{thm:aniso_abstract_reg}.

{\noindent}{{\em Step 1: Verification of {{\em {\jax}
\ref{ax:closure_aniso}}} {\opt}closure under rescaling{\cpt}}}. Let
$(\bm{u}, \jaOm, \bm{f}, A) \in
\class$ and $B_{\rr} (x_0) \subset B_1$. Consider the rescaled tuple
$(\bm{u}_{x_0, \rr}, {\overline{\jaOm}_{x_0, \rr}} ,
\tilde{\bm{f}}_{x_0, \rr}, A_{x_0, \rr})$. We need to
show that this triple remains in $\class$. By the rescaling
\eqref{eq:aniso_rescaling}, we have
\begin{equation}
  - \divg (A_{x_0, \rr} \nabla \bm{u}_{x_0, \rr}) =
  {\overline{\jaOm}_{x_0, \rr}}  \cdot \nabla
  \bm{u}_{x_0, \rr} +
  \tilde{\bm{f}}_{x_0, \rr},
\end{equation}
so the equation is invariant. Furthermore, the amplitude bound is preserved by
\eqref{eq:scale_uinf}; the ellipticity constants of $A$ are preserved
identically; and the $L^2$-norm of $\jaOm$ is conformally invariant in
dimension $m = 2$ (cf. \eqref{eq:scale_Omega}), so
\begin{equation} \| {\overline{\jaOm}_{x_0, \rr}}  \|_{L^2 (B_1)} = \| \jaOm \|_{L^2
   (B_{\rr} (x_0))} \leqslant \| \jaOm \|_{L^2 (B_1)} < \varepsilon_{\ast} .
\end{equation}
Finally, $\divg_y {\overline{\jaOm}_{x_0, \rr}}  \in L^{\jqq}$ and $A_{x_{0,
\rr}} \in C^{0, \gamma_0}$ by \eqref{eq:scale_divOmega} and
\eqref{eq:scale_Holder}. Thus, the class $\class$ is closed under rescaling.
\smallskip

{\noindent}{{\em Step 2: Verification of {{\em {\jax}
\ref{ax:decay_aniso}}}}}. Immediate from Lemma~\ref{lem:decay_aniso}: any
quadruple in $\class$ satisfies the smallness assumption
\eqref{eq:smallness_aniso}, and the decay estimate \eqref{eq:decay_aniso} is
exactly \eqref{eq:axiom_decay_aniso} with $\jgam = 1 / 2$.
\smallskip

{\noindent}{{\em Step 3: Conclusion and Removal of the small $\jaOm$-energy
condition}}. Having verified the two axioms, the abstract regularity
principle (Theorem~\ref{thm:aniso_abstract_reg}) dictates that any map
$\bm{u}$ belonging to a quadruple in $\class$ is locally
Hölder continuous in $B_1$. To conclude the proof of
Theorem~\ref{thm:aniso_regularity}, it remains to explain why this regularity
holds for any weak solution of the generalized anisotropic system, not just
those with small initial $\jaOm$-energy.

The arbitrary solution $\bm{u}$ may not lie in $\class$
(the $L^2$-energy of $\jaOm$ may be too large). Fix $x_0 \in B_1$ and let $R
:= 1 - | x_0 | > 0$ denote its distance to the boundary. Since $| \jaOm |^2
\in L^1 (B_1)$, the absolute continuity of the Lebesgue integral yields a
radius $\rr \in (0, R)$ with $\| \jaOm \|_{L^2 (B_{\rr} (x_0))} <
\varepsilon_{\ast}$. By the conformal invariance of the $L^2$-norm in 2D, the
rescaled quadruple $(\bm{u}_{x_0, \rr},
{\overline{\jaOm}_{x_0, \rr}} , \tilde{\bm{f}}_{x_{0,
\rr}}, A_{x_0, \rr})$ is defined on the unit ball $B_1$, solves the rescaled
equation, and satisfies $\| {\overline{\jaOm}_{x_0, \rr}}  \|_{L^2 (B_1)} <
\varepsilon_{\ast}$ together with the amplitude bound. Hence this rescaled
quadruple lies in $\class$. By Theorem~\ref{thm:aniso_abstract_reg}, the
rescaled map $\bm{u}_{x_0, \rr}$ is locally Hölder
continuous on $B_{1 / 2}$, with exponent $\jeta$ satisfying
\eqref{eq:aniso_holder_exponent}. Scaling back, $\bm{u}$
is Hölder continuous on $B_{\rr / 2} (x_0)$. Since $x_0 \in B_1$ was
arbitrary, $\bm{u} \in C^{0,
\jeta}_{\ensuremath{\operatorname{loc}}} (B_1, \RR^{n + 1})$, completing the
proof of Theorem~\ref{thm:aniso_regularity}.

\section{Smoothness of continuous Almost Harmonic Maps}\label{sec:sec6}

{\noindent}\textbf{Theorem~\ref{lemma:Bootstrap}.} (Bootstrap Regularity for Almost Harmonic
Maps)
{\itshape{ Let $U \subset \RR^m$ be an open domain, and let
$\bm{u} \in W^{1, 2}_{\mathrm{loc}} (U, \Stwo^n) \cap C^0
(U, \Stwo^n)$ be a weakly almost harmonic map solving the Euler--Lagrange
equation:
\begin{equation}
  - \Delta \bm{u}= | \nabla
  \bm{u}|^2
  \bm{u}+\bm{f} \quad \text{in }
  \mathcal{D}' (U),
\end{equation}
where the source term initially satisfies $\bm{f} \in
L^{\jqq}_{\mathrm{loc}} (U, \RR^{n + 1})$.

{\noindent}If the source integrability $\jqq$ satisfies $2 \leqslant \jqq
\leqslant m / 2$ {\opt}which can only occur in dimensions $m \geqslant
4${\cpt}, we additionally assume the map is already locally Hölder continuous:
$\bm{u} \in C^{0, \gamma}_{\mathrm{loc}} (U, \Stwo^n)$
for some $\gamma \in (0, 1]$.

{\noindent}Then, the regularity of $\bm{u}$ 
depends on the regularity of $\bm{f}$:
\smallskip

\begin{enumerate}
  \item If $\bm{f} \in L^{\jqq}_{\mathrm{loc}} (U, \RR^{n
  + 1})$ and $2 \leqslant \jqq \leqslant m$ then the regularity is bounded by
  the critical Sobolev embedding threshold:
  \begin{equation}
    \bm{u} \in W^{1, \jqq^{\ast}}_{\mathrm{loc}} (U,
    \Stwo^n)  \quad \text{with } \jqq^{\ast} = \frac{m \jqq}{m - \jqq},
  \end{equation}
  with the usual understanding that $\jqq^{\ast}$ (necessarily greater than
  $1$) can be any finite real exponent if $m = \jqq$. In particular, if $m =
  \jqq$ then $\bm{u} \in C^{0, \alpha}_{\mathrm{loc}} (U,
  \Stwo^n)$ for all $0 < \alpha < 1$, while if $\jqq = 2$ and $m = 3$ then
  $\bm{u} \in C^{0, 1 / 2}_{\mathrm{loc}} (U, \Stwo^n)$.
  \smallskip

  \item If $\bm{f} \in L^{\jqq}_{\mathrm{loc}} (U, \RR^{n
  + 1})$ with $\jqq > m$, then
  \begin{equation}
    \bm{u} \in C^{1, \eta}_{\mathrm{loc}} (U, \Stwo^n),
    \quad \text{where } \eta := 1 - m / \jqq . \label{eq:HighHold1reg}
  \end{equation}
  Note that $\eta$ does not depend on any intermediate fractional exponent. In
  particular, if $\bm{f} \in L^{\infty}_{\mathrm{loc}}
  (U, \RR^{n + 1})$ then $\bm{u} \in C^{1,
  \alpha}_{\mathrm{loc}} (U, \Stwo^n)$ for all $0 < \alpha < 1$.
  \smallskip

  \item For any integer $k \geqslant 1$ and exponent $\beta \in (0, 1)$, if
  $\bm{f} \in C^{k - 1, \beta}_{\mathrm{loc}} (U, \RR^{n
  + 1})$ then $\bm{u} \in C^{k + 1, \beta}_{\mathrm{loc}}
  (U, \Stwo^n)$. In particular, if $\bm{f} \in C^{\infty}
  (U, \RR^{n + 1})$, then $\bm{u} \in C^{\infty} (U,
  \Stwo^n)$. Again, note that $\eta$ does not play any role in the higher
  regularity class we end up in.
\end{enumerate}}}

\begin{remark}
  The true analytical power of case \text{{\itshape{i.}}} lies in the upgraded
  \text{{\itshape{Sobolev}}} regularity: specifically, the higher Lebesgue
  integrability of the gradient, $\bm{u} \in W^{1,
  q^{\ast}}_{\mathrm{loc}}$. Securing these strong $W^{1, q^{\ast}}$ bounds is
  paramount when transitioning to more advanced, non-linear settings where the
  source term depends on the solution itself (e.g., treated as a frozen
  quantity $\bm{f} (x) \equiv \bm{g}
  (x, \bm{u}(x))$). In such scenarios, one must feed the
  improved gradient integrability back into the PDE to re-evaluate
  the integrability of the frozen source, thereby closing the bootstrap loop
  and unlocking higher regularity. We will explicitly see this self-improving
  mechanism in action in the next section (see
  Section~\ref{sec:appsmicromag}), where we treat the micromagnetic case.
\end{remark}

\begin{remark}
  One must be careful about the initial hypothesis
  $\bm{u} \in C^0 $. In dimension $m = 2$, as we already
  proved, two-dimensional geometry yields sufficient critical cancellations to
  ensure that every weakly almost harmonic map $\bm{u}
  \in W^{1, 2} (U, \Stwo^n)$ is automatically continuous. However, in
  dimension $m \geqslant 3$ continuity fails in general. Because the Dirichlet
  energy scales differently in higher dimensions, topological singularities
  may have finite energy. A classical example is the radial map
  $\bm{u} (x) = x / |x|$ from the unit ball in $\RR^3$
  into $\Stwo^2$, which is weakly harmonic but has a point singularity at the
  origin. More strikingly, Rivière~{\cite{Rivi_re_1995}}, in 1995, constructed
  weakly harmonic maps in dimension three into $\Stwo^2$ that are
  \text{{\itshape{discontinuous}}} everywhere. Therefore, for $m \geqslant 3$
  the algebraic bootstrap that upgrades a locally continuous solution to
  $C^{\infty}$ remains valid, but it applies only at regular points, i.e., at
  points where the map is already locally continuous.
\end{remark}

\begin{proof}[Proof of Theorem~\ref{lemma:Bootstrap}]
  We divide the proof into nine steps.
  
{\noindent}{\itshape{Step 1: Localization and the Broken Conservation
  Law.}} Let $x_0 \in U$ be an arbitrary point. Since
  $\bm{u}$ is continuous, we can choose a sufficiently
  small radius $R > 0$ such that the oscillation of
  $\bm{u}$ on the ball $B_R (x_0) \subset U$ is strictly
  less than $1 / 2$. By applying a constant rotation to the target sphere, we
  can assume without loss of generality that $\bm{u}
  (x_0) =\ensuremath{\boldsymbol{e}}_{n + 1}$ is the north pole. By
  continuity, the image $\bm{u}(B_R (x_0))$ is strictly
  contained in the upper hemisphere; in particular, $u^{n + 1} (x) :=
  \bm{u}(x) \cdot \ensuremath{\boldsymbol{e}}_{n + 1}
  \geqslant 1 / 2$ for all $x \in B_R (x_0)$.
  
  For almost harmonic maps into
  $\Stwo^n$, the antisymmetric
  matrix of 1-forms defined by $\Omega^{\alpha \beta} = u^{\alpha} \nabla
  u^{\beta} - u^{\beta} \nabla u^{\alpha}$ is no longer
  divergence-free. Computing the weak divergence and substituting the PDE
  yields a perfect cancellation of the critical gradient terms, leaving only
  the source term:
  \begin{equation}
    \divr (u^{\alpha} \nabla u^{\beta} - u^{\beta} \nabla u^{\alpha}) =
    u^{\beta} f^{\alpha} - u^{\alpha} f^{\beta}  \quad \text{in } \mathcal{D}'
    (B_R (x_0)) .
  \end{equation}
  {\noindent}{\itshape{Step 2: The Gnomonic Projection.}} Because $u^{n
  + 1} \geqslant 1 / 2$, we can smoothly project the map from the center of
  the sphere onto the affine tangent plane at the north pole. We define the
  $\mathbb{R}^n$-valued function $\bm{w}= (w^1, \ldots,
  w^n) \in W^{1, 2}_{\mathrm{loc}} \cap C^0$ by $w^{\alpha} := u^{\alpha} /
  u^{n + 1}$ for $\alpha = 1, \ldots, n$. Applying the quotient rule for weak
  derivatives gives the  algebraic identity:
  \begin{equation}
    (u^{n + 1})^2 \nabla w^{\alpha} = u^{n + 1} \nabla u^{\alpha} - u^{\alpha}
    \nabla u^{n + 1} .
  \end{equation}
  {\noindent}{\itshape{Step 3: Linearizing the Equation.}} Taking the
  divergence of both sides, and applying the broken conservation law from Step
  1 (setting $\beta = n + 1$), we obtain:
  \begin{equation}
    \divr ((u^{n + 1})^2 \nabla w^{\alpha}) = u^{\alpha} f^{n + 1} - u^{n + 1}
    f^{\alpha} = : g^{\alpha} (x)  \quad \text{for } \alpha = 1, \ldots, n.
  \end{equation}
  Note that the quadratic gradient terms $| \nabla
  \bm{u}|^2$ have been completely eliminated. The
  critical nonlinear system has decoupled into a system of $n$ linear
  divergence-form elliptic equations:
  \begin{equation}
    \divg (A \nabla \bm{w}) =\bm{g},
    \label{eq:inhomogeneousPoisson}
  \end{equation}
  where the scalar coefficient $A = (\bm{u}\cdot
  \ensuremath{\boldsymbol{e}}_{n + 1})^2$ is strictly positive ($A (x)
  \geqslant 1 / 4$), and the new right-hand side $\bm{g}$
  depends linearly on $\bm{f}$. Since
  $|\bm{u}| = 1$ and $\bm{f} \in
  L^2_{\mathrm{loc}}$, it follows that
  $\bm{g} \in L^2_{\mathrm{loc}}$.{\smallskip}
  
  {\noindent}{\itshape{Step 4: Upgrading $C^0$ to $C^{0, \gamma}$ via De
  Giorgi--Nash--Moser.}} Because we  assumed only that
  $\bm{u}\in C^0 (U)$, the coefficient $A$ is merely
  continuous, which momentarily stalls classical Schauder theory. We cross
  this gap by invoking the  De Giorgi--Nash--Moser theorem on our
  decoupled system \eqref{eq:inhomogeneousPoisson}.
  
  Since $A \geqslant 1 / 4$, the coefficient is bounded and uniformly
  elliptic. The theorem (see {\cite[Theorem~4.13, p.83]{Han2011}}) guarantees
  that weak solutions gain local Hölder continuity provided the right-hand
  side is integrable enough. Specifically, if $\bm{g} \in
  L^q_{\mathrm{loc}}$ for some $q > m / 2$, then $\bm{w}
  \in C^{0, \gamma}_{\mathrm{loc}}$ for some $0 < \gamma < 1$.
  
  Because $g^{\alpha} = u^{\alpha} f^{n + 1} - u^{n + 1} f^{\alpha}$, the
  source $\bm{g}$  inherits the
  $L^q_{\mathrm{loc}}$ integrability of $\bm{f}$. The
  strict requirement to initiate the bootstrap is therefore $q > m / 2$.
  Because we assume a baseline integrability of $q \geqslant 2$, this
  condition is automatically satisfied for domains of dimension $m \leqslant
  3$. If $q \leqslant m / 2$, instead, our initial hypothesis explicitly
  assumed $\bm{u}\in C^{0, \gamma}_{\mathrm{loc}}$,
  bypassing this step entirely.
  
  Algebraic inversion gives $\bm{u}\cdot
  \ensuremath{\boldsymbol{e}}_{n + 1} = (1 +
  |\bm{w}|^2)^{- 1 / 2}$. Since Hölder spaces are closed
  under smooth compositions, $\bm{w} \in C^{0,
  \gamma}_{\mathrm{loc}}$  forces the vertical component to be
  Hölder continuous. Consequently, the coefficient $A =
  (\bm{u}\cdot \ensuremath{\boldsymbol{e}}_{n + 1})^2$
  is officially upgraded to $C^{0, \gamma}_{\mathrm{loc}}$, clearing the
  temporary roadblock.{\smallskip}
  
  {\noindent}{\itshape{Step 5: Linear Theory, the Flux Representation,
  and the Integrability Ceiling.}} Because $\bm{u}\in
  C^{0, \gamma}_{\mathrm{loc}}$ (from Step 4), the uniformly elliptic
  coefficient $A \in C^{0, \gamma}_{\mathrm{loc}}$. By classical linear
  elliptic theory for divergence-form equations, the regularity of
  $\bm{w}$ now depends on the source term
  $\bm{g} \in L^2_{\mathrm{loc}}$. To precisely determine
  the Sobolev space of $\bm{w}$, we must express the
  source term in divergence form so that Calderón--Zygmund theory can be
  applied. On the ball $B_R (x_0)$ defined in Step 1, let
  $\ensuremath{\boldsymbol{v}} \in W^{1, 2}_0 (B_R (x_0))$ be the unique weak
  solution to the Dirichlet problem $\Delta
  \ensuremath{\boldsymbol{v}}=\bm{g}$ with zero boundary
  conditions. By standard $L^p$ elliptic regularity up to the boundary, since
  $\bm{g} \in L^2 (B_R (x_0))$, we have
  $\ensuremath{\boldsymbol{v}} \in W^{2, 2} (B_R (x_0))$. Setting
  $\bm{G} := \nabla \ensuremath{\boldsymbol{v}}$, we
  rewrite the linearized equation on $B_R (x_0)$ as:
  \begin{equation}
    \divg (A \nabla \bm{w}) = \divg
    \bm{G}, \quad \text{where }
    \bm{G} \in W^{1, 2} (B_R (x_0)) .
    \label{eq:inhomogeneousPoissonflux}
  \end{equation}
  {\itshape{Proof of \text{{\upshape{i.}}} The $L^q$ Bottleneck and
  Dimensional Dependency.}} Provided the coefficient $A$ is at least
  uniformly continuous, Calderón--Zygmund theory guarantees that the gradient
  of the solution $\nabla \bm{w}$  inherits the
  \text{{\itshape{Lebesgue}}} integrability of the flux
  $\bm{G}$. In general, if $\bm{f}
  \in L^q_{\mathrm{loc}}$, the flux satisfies $\bm{G} \in
  W^{1, q}_{\mathrm{loc}}$. This integrability is  governed by the
  classical Sobolev embedding theorem, which introduces a hard dependency on
  the dimension $m$: if $\bm{G} \in W^{1,
  q}_{\mathrm{loc}}$ with $q \leqslant m$, then $\bm{G}
  \in L^{q^{\ast}}_{\mathrm{loc}}$, where $q^{\ast} = m  q / (m - q)$.
  Therefore, by Calderón--Zygmund, $\nabla \bm{w} \in
  L^{q^{\ast}}_{\mathrm{loc}}$, meaning that $\bm{w} \in
  W^{1, q^{\ast}}_{\mathrm{loc}}$. (As usual, if $q = m$, $q^{\ast}$ can be
  any finite real exponent). If $q^{\ast} > m$ (which is algebraically
  equivalent to $q > m / 2$), Morrey's inequality yields
  $\bm{w} \in C^{0, \alpha}_{\mathrm{loc}}$ with $\alpha
  = 1 - m / q^{\ast}$. In particular:
  \begin{itemize}
    \item If $q = m$, then $q^{\ast}$ is arbitrarily large, yielding
    $\bm{w} \in C^{0, \alpha}_{\mathrm{loc}} \cap W^{1,
    p}_{\mathrm{loc}}$ for all $\alpha < 1$ and all $1 \leqslant p < \infty$.
    The regularity becomes arbitrarily close to $C^1$, but  stalls
    just short of it.
      \smallskip

    \item If $q = 2$ and $m = 3$, then $q^{\ast} = 6$, yielding
    $\bm{w} \in C^{0, 1 / 2}_{\mathrm{loc}} \cap W^{1,
    6}_{\mathrm{loc}}$.
      \smallskip

    \item If $q \leqslant m / 2$ (e.g., $q = 2, m \geqslant 4$), we cannot
    even guarantee continuity from the integrability of the source alone.
  \end{itemize}
  Finally, because the algebraic inversion from $\bm{w}$
  back to $\bm{u}$ is smooth,
  $\bm{u}$  inherits these $W^{1,
  q^{\ast}}_{\mathrm{loc}}$ and $C^{0, \alpha}_{\mathrm{loc}}$ spaces. This
  proves assertion \text{{\itshape{i.}}}
  
  {\noindent}{\itshape{Proof of \text{{\upshape{ii.}}} Step 6: Pushing
  the Bootstrap forward via Morrey and Campanato.}} To push the bootstrap
  forward, we must cross the boundary from Lebesgue integrability ($L^p$) into
  classical Hölder continuity ($C^{1, \eta_0}_{\mathrm{loc}}$). The
  mathematical bridge between these two realms  requires the
  right-hand side of our PDE, $\bm{g}$, to cross the
  dimensional threshold $q > m$. Recalling the algebraic definition
  $g^{\alpha} = u^{\alpha} f^{n + 1} - u^{n + 1} f^{\alpha}$ and the geometric
  bound $|\bm{u}| = 1$, we observe that
  $\bm{g}$  inherits the Lebesgue integrability
  of the original source $\bm{f}$. Therefore, we 
  assume higher integrability on our original source:
  $\bm{f} \in L^q_{\mathrm{loc}}$ for some $q > m$, which
  guarantees $\bm{g} \in L^q_{\mathrm{loc}}$.
  
  Following the exact same flux representation as above, the potential
  $\ensuremath{\boldsymbol{v}}$ (solving $\Delta
  \ensuremath{\boldsymbol{v}}=\bm{g}$) now belongs to
  $W^{2, q} (B_R (x_0))$, which implies our flux $\bm{G}
  := \nabla \ensuremath{\boldsymbol{v}}$ belongs to $W^{1, q} (B_R (x_0))$.
  Because $q > m$, by Morrey's inequality, the flux
  $\bm{G}$ is no longer just an integrable function; it
  is continuous. Specifically, it embeds into a Hölder space:
  \begin{equation}
    \bm{G} \in C^{0, \alpha} (B_R (x_0)), \quad
    \text{where } \alpha := 1 - m / q.
  \end{equation}
  This is the critical transition point in the argument. In the equation
  \eqref{eq:inhomogeneousPoissonflux}, both the coefficient $A \in C^{0,
  \gamma}_{\mathrm{loc}}$ and the flux $\bm{G} \in C^{0,
  \alpha}_{\mathrm{loc}}$ are now locally Hölder continuous. The
  Calderón--Zygmund $L^q$-theory is no longer sufficient, and we instead
  appeal to the classical Campanato regularity theory for divergence-form
  elliptic equations. These estimates (see, for instance, Giaquinta
  {\cite[Chapter III]{Giaquinta1983}} or Giusti {\cite[Theorem
  5.19]{Giusti_2003}}) imply that $\nabla \bm{w}$
  inherits the Hölder continuity of the coefficients and forcing terms, with
  regularity determined by the smaller of the two exponents. More precisely,
  \begin{equation}
    \nabla \bm{w} \in C^{0, \eta_0}_{\mathrm{loc}} (B_R
    (x_0)), \qquad \eta_0 = \min (\gamma, \alpha),
  \end{equation}
  with $\alpha = 1 - m / q$. Consequently, we obtain:
  \begin{equation}
    \bm{w} \in C^{1, \eta_0}_{\mathrm{loc}} (B_R (x_0)),
    \quad \text{where } \eta_0 = \min (\gamma, 1 - m / q) .
  \end{equation}
  {\noindent}{\itshape{Step 7: Bootstrapping back to
  $\bm{u}$ and the Self-Improving Loop.}} Assuming the
  source meets the threshold integrability $\bm{f} \in
  L^q_{\mathrm{loc}}$ ($q > m$), Step 6 guarantees that
  $\bm{w} \in C^{1, \eta_0}_{\mathrm{loc}}$. We now
  algebraically invert the Gnomonic projection to recover the original map
  $\bm{u}$. Since $|\bm{w}|^2 =
  (\bm{u}\cdot \ensuremath{\boldsymbol{e}}_{n + 1})^{-
  2} - 1$, solving for the vertical component $u^{n + 1}
  =\bm{u}\cdot \ensuremath{\boldsymbol{e}}_{n + 1}$
  yields $\bm{u}\cdot \ensuremath{\boldsymbol{e}}_{n +
  1} = (1 + |\bm{w}|^2)^{- 1 / 2}$. Because Hölder spaces
  are Banach algebras and are closed under composition with smooth functions,
  the composition of $\bm{w} \in C^{1,
  \eta_0}_{\mathrm{loc}}$ with the smooth function $y \mapsto (1 + |y|^2)^{- 1
  / 2}$ implies $\bm{u}\cdot
  \ensuremath{\boldsymbol{e}}_{n + 1} \in C^{1, \eta_0}_{\mathrm{loc}} (B_R
  (x_0))$. Consequently, the remaining components are recovered via smooth
  products: $u^{\alpha} = w^{\alpha} u^{n + 1} \in C^{1,
  \eta_0}_{\mathrm{loc}} (B_R (x_0))$. Since the base point $x_0$ was
  arbitrary, we have successfully lifted the entire map to
  $\bm{u}\in C^{1, \eta_0}_{\mathrm{loc}} (U,
  \Stwo^n)$.
  
  {\noindent}{\itshape{Step 8: Unconditional Hölder Exponents
  \text{{\upshape{(}}}Dropping $\gamma$\text{{\upshape{)}}}.}} We can now
  completely eliminate the initial fractional dependency. Because
  $\bm{u}\in C^{1, \eta_0}_{\mathrm{loc}}$, its gradient
  is locally continuous and thus locally bounded
  ($L^{\infty}_{\mathrm{loc}}$). We return to the original harmonic map
  system, treating the right-hand side as a fixed source for the
  constant-coefficient Laplacian: $- \Delta
  \bm{u}=\ensuremath{\boldsymbol{\varphi}} (x) := |
  \nabla \bm{u}|^2
  \bm{u}+\bm{f}$. Since $\nabla
  \bm{u}\in L^{\infty}_{\mathrm{loc}}$, the nonlinear
  term $| \nabla \bm{u}|^2 \bm{u}
  \in L^{\infty}_{\mathrm{loc}} \subset L^q_{\mathrm{loc}}$. Therefore, the
  entire right-hand side $\ensuremath{\boldsymbol{\varphi}}$ belongs to
  $L^q_{\mathrm{loc}}$. Applying standard Calderón--Zygmund estimates for the
  constant-coefficient Laplacian yields $\bm{u}\in W^{2,
  q}_{\mathrm{loc}}$. By Morrey's inequality, this  embeds into $C^{1,
  \eta}_{\mathrm{loc}}$ with the pure exponent $\eta = 1 - m / q$, completely
  forgetting the initial $\gamma$.
  
  In particular, if $\bm{f} \in
  L^{\infty}_{\mathrm{loc}}$ (which corresponds to $q \to \infty$), the
  right-hand side is in $L^p_{\mathrm{loc}}$ for all $p < \infty$, yielding
  $\bm{u}\in C^{1, \alpha}_{\mathrm{loc}}$ for all $0 <
  \alpha < 1$. This completes the proof of \text{{\itshape{ii}}}.{\smallskip}
  
  {\noindent}{\itshape{Step 9: Proof of \text{{\upshape{iii.}}} The
  Standard Schauder Bootstrap \text{{\upshape{(}}}$C^{1, \beta} \to
  C^{\infty}$\text{{\upshape{)}}}}} Once the gradient $\nabla
  \bm{u}$ is Hölder continuous, the geometric criticality
  of the PDE no longer obstructs the iteration. We treat the system as a standard
  Poisson equation (\text{{\scshape{pe}}}):
  \begin{equation}
    - \Delta \bm{u}=\bm{\varphi} :=
    | \nabla \bm{u}|^2
    \bm{u}+\bm{f} (x) .
  \end{equation}
  The regularity of $\bm{u}$ now depends exclusively upon
  the smoothness of the external source $\bm{f}$. We
  proceed by induction using standard Schauder estimates
  (\text{{\scshape{se}}}) for the Laplacian.
  \begin{itemize}
    \item {\itshape{Base Step \text{{\upshape{(}}}$k =
    1$\text{{\upshape{)}}}:}} Assume $\bm{f} \in C^{0,
    \beta}_{\mathrm{loc}}$. Since $\bm{f}$ is locally
    bounded ($\bm{f} \in L^{\infty}_{\mathrm{loc}}$), the
    acquired point \text{{\itshape{ii.}}} of the theorem guarantees that
    $\bm{u}\in C^{1, \alpha}_{\mathrm{loc}}$ for
    \text{{\itshape{all}}} $\alpha < 1$. By choosing $\alpha = \beta$, we have
    $\nabla \bm{u}\in C^{0, \beta}_{\mathrm{loc}}$.
    Consequently, the nonlinear term satisfies $| \nabla
    \bm{u}|^2 \bm{u}\in C^{0,
    \beta}_{\mathrm{loc}}$. The total source
$\bm{\varphi}$ is therefore in $C^{0,\beta}_{\mathrm{loc}}$. Applying Schauder estimates to the Poisson
    equation yields:
    \begin{equation}
      \Delta \bm{u}\in C^{0, \beta}_{\mathrm{loc}}
      \hspace{1.2em} \overset{\text{{\scshape{se}}}}{\Longrightarrow}
      \hspace{1.2em} \bm{u}\in C^{2,
      \beta}_{\mathrm{loc}} .
    \end{equation}
    \item{\itshape{Inductive Step:}} Assume that for some integer $k
    \geqslant 2$, we have $\bm{u}\in C^{k,
    \beta}_{\mathrm{loc}}$ and we are given a smoother source
    $\bm{f} \in C^{k - 1, \beta}_{\mathrm{loc}}$. Because
    $\bm{u}\in C^{k, \beta}_{\mathrm{loc}}$, its
    gradient is $\nabla \bm{u}\in C^{k - 1,
    \beta}_{\mathrm{loc}}$. Since Hölder spaces $C^{k - 1,
    \beta}_{\mathrm{loc}}$ are Banach algebras, the product $| \nabla
    \bm{u}|^2 \bm{u}$ belongs to
    $C^{k - 1, \beta}_{\mathrm{loc}}$. Therefore, the total source is
    $\bm{\varphi} \in C^{k - 1, \beta}_{\mathrm{loc}}$.
    The bootstrap pushes the solution up by two derivatives relative to the
    source:
    \begin{equation}
      \Delta \bm{u}\in C^{k - 1, \beta}_{\mathrm{loc}}
      \hspace{1.2em} \overset{\text{{\scshape{se}}}}{\Longrightarrow}
      \hspace{1.2em} \bm{u}\in C^{k + 1,
      \beta}_{\mathrm{loc}} .
    \end{equation}
  \end{itemize}
  Overall, if $\bm{f} \in C^{k - 1,
  \beta}_{\mathrm{loc}}$ and $\bm{u}\in C^{0,
  \eta}_{\mathrm{loc}}$ then $\bm{u}\in C^{k + 1,
  \beta}$. Again, note that $\eta$ does not play any role in the higher
  regularity class we end up in.
  
  By applying this inductive loop recursively, we see that if
  $\bm{f} \in C^{\infty}$, then
  $\bm{u}\in C^{k, \beta}_{\mathrm{loc}}$ for all
  integers $k \geqslant 0$. Thus, we conclude that
  $\bm{u}\in C^{\infty} (U,
  \Stwo^n)$.
\end{proof}

\section{Applications to Micromagnetics: Smoothness of Magnetic
Skyrmions}\label{sec:appsmicromag}

We now illustrate the power of the Bootstrap Lemma~\ref{lemma:Bootstrap} by
applying it to the variational theory of micromagnetics. This macroscopic
continuum theory is of paramount importance in modern condensed matter physics
and spintronics, particularly for the rigorous analysis of complex topological
spin textures such as magnetic skyrmions. Skyrmions are localized,
particle-like chiral spin configurations characterized by a non-trivial
topological winding number, which grants them inherent stability against
continuous deformations (topological protection). Because they can be
nucleated, manipulated, and driven by ultra-low electrical currents, skyrmions
are currently at the forefront of next-generation solid-state technologies,
including high-density non-volatile magnetic storage (such as racetrack memory
architectures), logic devices, and neuromorphic computing
paradigms~{\cite{Fert_2013}}. A rigorous understanding of the analytical
properties, such as the maximal regularity, of these magnetic configurations
is therefore fundamentally necessary to validate both theoretical predictions
and numerical simulations.

The state of a rigid ferromagnetic body occupying a domain $U \subset \RR^m$
(with $m = 2$ or $3$) is described by its magnetization vector field
$\bm{u}$. Below the Curie temperature, the magnitude of
the magnetization is assumed to be constant, leading to the non-convex
pointwise saturation constraint $|\bm{u}(x) | = 1$,
meaning $\bm{u} \in H^1 (U, \Stwo^2)$.

The stable magnetic configurations are the local minimizers of the
micromagnetic energy functional
{\cite{Alouges_2015,Bertotti1998,Brown1963,Di_Fratta_2020a}}:
\begin{equation}
  \mathcal{G} (\bm{u}) := \int_U \left[ \frac{1}{2} |
  \nabla \bm{u}|^2 + \kappa \bm{u}
  \cdot \curl \bm{u}+ \varphi
  (\bm{u}) \right]  \hspace{0.17em} \mathrm{d} x
  +\mathcal{E}_{\ensuremath{\operatorname{stray}}}
  (\bm{u}), \quad
  \mathcal{E}_{\ensuremath{\operatorname{stray}}}
  (\bm{u}) := - \frac{1}{2} \int_U
  \ensuremath{\boldsymbol{h}} [\bm{u}] \cdot
  \bm{u}.
\end{equation}
Each term in this functional models a distinct quantum-mechanical or
macroscopic physical interaction:
\begin{itemize}
  \item {{\em Exchange Energy}} ($\frac{1}{2} | \nabla
  \bm{u}|^2$): This is the standard Dirichlet energy,
  which penalizes spatial variations and favors the uniform, parallel
  alignment of magnetic spins.
    \smallskip

  \item {{\em Dzyaloshinskii-Moriya Interaction}} {\opt}DMI{\cpt} ($\kappa
  \bm{u} \cdot \curl
  \bm{u}$): This anti-symmetric exchange term arises from
  spin-orbit coupling in bulk materials or interfaces lacking inversion
  symmetry. Unlike the standard exchange that favors parallel alignment, the
  DMI favors a specific chirality or handedness, promoting the formation of
  complex topological spin textures such as spin helices and magnetic
  skyrmions (see {\cite{Di_Fratta_2023}} for a rigorous mathematical treatment
  of DMI).
    \smallskip

  \item {{\em Magnetocrystalline Anisotropy}} ($\varphi
  (\bm{u})$): This term models the tendency of the
  magnetization to align with specific crystallographic axes. The function
  $\varphi : \Stwo^2 \to \RR$ is typically a smooth (at least $C^k$) function
  that vanishes at a finite set of points (the ``easy axes''). A standard
  example is the uniaxial anisotropy $\varphi (\bm{u}) =
  K (1 - (\bm{u} \cdot e_3)^2)$ for some material
  constant $K > 0$.
  \smallskip
  
  \item {{\em Demagnetizing / Stray Field Energy}}
  ($\mathcal{E}_{\ensuremath{\operatorname{stray}}}
  (\bm{u})$): This accounts for the long-range
  dipole-dipole interactions. In 3D, it is represented via the demagnetizing
  field $\ensuremath{\boldsymbol{h}} [\bm{u}]$ generated
  by $\bm{u}$, which solves the magnetostatic Maxwell
  equations. Mathematically, $\ensuremath{\boldsymbol{h}}= - \nabla \Delta^{-
  1} \divr (\bm{u} \chi_U)$ where
  $\bm{u} \chi_U$ denotes the extension by zero of
  $\bm{u}$ to the whole of $\RR^3$ . Because it is a
  zero-order Calderón-Zygmund singular integral operator, the map
  $\ensuremath{\boldsymbol{h}}: L^2 \rightarrow L^2$ is bounded on $L^p$
  spaces for $1 < p < \infty$ and satisfies the standard elliptic properties
  for any $k \in \NN$, $\eta \in (0, 1)$, and $p \in (1, \infty)$:
  \begin{equation}
    \bm{u} \in L^p \quad \Longrightarrow \quad
    \ensuremath{\boldsymbol{h}} [\bm{u}] \in L^p .
    \label{eq:demagregp}
  \end{equation}
  \begin{equation}
    \bm{u} \in C^{k, \eta}_{\mathrm{loc}} \quad
    \Longrightarrow \quad \ensuremath{\boldsymbol{h}}
    [\bm{u}] \in C^{k, \eta}_{\mathrm{loc}} .
    \label{eq:demagreg}
  \end{equation}
  In 2D the nonlocal stray field operator collapses to a local term of the
  form $- (\bm{u} \cdot \ensuremath{\boldsymbol{e}}_3)
  \ensuremath{\boldsymbol{e}}_3$ and, more generally to $-
  (\bm{u} \cdot \ensuremath{\boldsymbol{n}})
  \ensuremath{\boldsymbol{n}}$ if $U$ is a compact surface in $\RR^3$ (we
  refer to {\cite{Di_Fratta_2020c,Di_Fratta_2023,Di_Fratta_2024}} for the rigorous derivation
  of these reduced thin-film models).
\end{itemize}
Regardless, for what follows, it is sufficient for us to assume that
$\ensuremath{\boldsymbol{h}}: L^2 \rightarrow L^2$ is a linear, possibly
non-local operator satisfying the regularity behavior in \eqref{eq:demagregp}
and \eqref{eq:demagreg}.

The Euler-Lagrange equations associated with the critical points of
$\mathcal{G}$, constrained to the sphere $\Stwo^2$, are known as Brown's
static equations. They take the form:
\begin{equation}
  - \Delta \bm{u}= | \nabla
  \bm{u}|^2
  \bm{u}+\bm{u} \times \left(
  \bm{u} \times \left( 2 \kappa
  \curl \bm{u}+
  \grad_{\bm{u}} \varphi (\bm{u})
  -\ensuremath{\boldsymbol{h}}[\bm{u}] \right) \right) .
\end{equation}
From a PDE regularity perspective, we can absorb the lower-order magnetic
interactions into a single effective source term, rewriting the system
as an almost harmonic map equation:
\begin{equation}
  - \Delta \bm{u}= | \nabla
  \bm{u}|^2
  \bm{u}+\bm{f}, \quad \text{with }
  \bm{f} := \bm{u} \times \left(
  \bm{u} \times \left( 2 \kappa
  \curl \bm{u}+
  \nabla_{\bm{u}} \varphi (\bm{u})
  -\ensuremath{\boldsymbol{h}}[\bm{u}] \right) \right) .
\end{equation}
We now state the main regularity result for micromagnetic configurations,
which follows cleanly from our abstract framework.

{\noindent}\textbf{Theorem~\ref{thm:micro_regularity}.} (Interior Regularity of Micromagnetic Maps)
{\itshape{Let $U \subseteq \RR^m$ {\opt}$m = 2, 3${\cpt} be a bounded open set and
let $\bm{u} \in H^1 \cap C^0 \left( U, \Stwo^2 \right)$
be a continuous critical point of the micromagnetic energy $\mathcal{G}$.
Assume that $\ensuremath{\boldsymbol{h}}: L^2 \left( U, \RR^3 \right)
\rightarrow L^2 \left( U, \RR^3 \right)$ satisfies the regularity hypotheses
{{\em \eqref{eq:demagregp}}} and {{\em \eqref{eq:demagreg}}}.

If the anisotropy density $\varphi$ is of class $C^{k, \beta}$ for $k
\geqslant 1$ and exponent $\beta \in (0, 1)$, then
$\bm{u} \in C^{k + 1, \beta}_{\mathrm{loc}} \left( U,
\Stwo^2 \right)$. In particular, if $\varphi \in C^{\infty}$, then
$\bm{u} \in C^{\infty} \left( U, \Stwo^2 \right)$.

For $m = 2$, the initial continuity assumption is automatically satisfied for
finite-energy solutions.}}

\begin{remark}
  From a geometric perspective, the physical dimensionality of the magnetic
  material dictates the structure of its singular set. In $m = 2$ dimensions,
  finite-energy solutions are a priori known to be continuous everywhere;
  therefore, 2D micromagnetic configurations are unconditionally smooth if
  $\varphi \in C^{\infty}$. In $m = 3$ dimensions, however, topological point
  defects (such as the hedgehog singularity $x / |x|$) may exhibit finite
  energy. Thus, 3D solutions are $C^{\infty}$ smooth  outside of their
  singular set (the specific discrete points where initial continuity fails).
\end{remark}

\begin{proof}
  We apply the Bootstrap Theorem \ref{lemma:Bootstrap} to establish the higher
  regularity of weak critical points. Because the physical dimension is
  restricted to $m \leqslant 3$, the baseline integrability of the effective
  source term ($q = 2$)  satisfies the De Giorgi--Nash--Moser
  non-stalling condition $q > m / 2$. Therefore, the abstract bootstrap is
  unconditionally triggered from mere continuity, and we are not required to
  assume any fractional Hölder regularity a priori.
  
  Because $\bm{u} \in W^{1, 2}$, the differential term
  \text{{\bfseries{curl }}}$\bm{u} \in L^2$. Furthermore,
  the geometric constraint $|\bm{u}| = 1$ ensures
  $\bm{u} \in L^{\infty}$, meaning the singular integral
  operator guarantees (via \eqref{eq:demagregp}) that
  $\ensuremath{\boldsymbol{h}} [\bm{u}] \in
  L^p_{\mathrm{loc}}$ for all finite $p$. It follows that the
  composite effective source term initiates at $\bm{f}
  \in L^2_{\mathrm{loc}}$.
  
  In the following argument, we focus on the physically critical dimension $m
  = 3$, where the baseline integrability is weakest. The proof for $m = 2$
  proceeds identically but crosses the classical regularity thresholds sooner.
 \smallskip
 
  \noindent {{\em First Bootstrap Step {\opt}Sobolev Improvement{\cpt}.}} By
    Theorem~\ref{lemma:Bootstrap}.{{\em i}}, since
    $\bm{f} \in L^2_{\mathrm{loc}}$, we automatically
    deduce that $\bm{u} \in W^{1,
    2^{\ast}}_{\mathrm{loc}} (\Om, \Stwo^2)$. In $m = 3$ dimensions, the
    critical Sobolev exponent is $2^{\ast} = 6$. Consequently, $\nabla
    \bm{u} \in L^6_{\mathrm{loc}}$. It follows that
    $\curl \bm{u} \in
    L^6_{\mathrm{loc}}$ and $\ensuremath{\boldsymbol{h}}
    [\bm{u}] \in L^p_{\mathrm{loc}}$ for all $p <
    \infty$. Overall, our source term  improves to
    $\bm{f} \in L^6_{\mathrm{loc}}$.
        \smallskip
        
   \noindent {{\em Second Bootstrap Step {\opt}Hölder Gradients{\cpt}.}} We
    have now established $\bm{f} \in L^q_{\mathrm{loc}}$
    with $q = 6$. Because $q > m$ (i.e., $6 > 3$), we invoke
    Theorem~\ref{lemma:Bootstrap}.{{\em ii}} to cross the threshold into
    classical differentiability. We deduce that $\bm{u}
    \in C^{1, \eta}_{\mathrm{loc}}$ with $\eta := 1 - m / q = 1 / 2$.
    \smallskip
    
    \noindent {{\em Higher Regularity.}} Because $\bm{u}
    \in C^{1, \eta}_{\mathrm{loc}}$, its first-order derivatives (such as
    $\curl \bm{u}$) belong to
    $C^{0, \eta}_{\mathrm{loc}}$. Concurrently, by \eqref{eq:demagreg}, the
    zero-order stray field $\ensuremath{\boldsymbol{h}}
    [\bm{u}]$ is $C^{1, \eta}_{\mathrm{loc}}$. The
    regularity of the gradient-dependent source term
    $\bm{f}$ is dictated by its least regular component,
    ensuring $\bm{f} \in C^{0, \eta}_{\mathrm{loc}}$.
    Defining the full right-hand side as $\ensuremath{\boldsymbol{\psi}} := |
    \nabla \bm{u}|^2
    \bm{u}+\bm{f}$, we see that
    $\ensuremath{\boldsymbol{\psi}} \in C^{0, \eta}_{\mathrm{loc}}$. Standard
    linear elliptic regularity for the Poisson equation then yields
    $\bm{u} \in C^{2, \eta}_{\mathrm{loc}}$. This
    elevated regularity forces $\ensuremath{\boldsymbol{\psi}} \in C^{1,
    \eta}_{\mathrm{loc}}$, subsequently yielding $\bm{u}
    \in C^{3, \eta}_{\mathrm{loc}}$, and so forth.
    
  Iterating this Schauder-type bootstrap, we conclude that the configuration
  $\bm{u}$ precisely inherits the maximal regularity
  permitted by the anisotropy density: $\bm{u} \in C^{k +
  1, \beta}_{\mathrm{loc}} (U, \Stwo^2)$ if $\varphi \in C^{k, \beta}$. In
  particular, $\bm{u} \in C^{\infty} (U, \Stwo^2)$
  whenever the anisotropy $\varphi$ is smooth.
\end{proof}

\section{Appendix: Interior Gradient Estimates for Harmonic
Functions}\label{sec:app1}

In this appendix, we collect classical interior estimates for maps with
harmonic components. These results, which bound the gradient of a harmonic (or
anisotropic harmonic) map in terms of its boundary values, are frequently
invoked throughout the paper and are provided here for completeness.

The section is structured around three main results. First,
Proposition~\ref{prop:Poissongradest} leverages the explicit Poisson kernel on
Euclidean balls to establish a scale-invariant gradient bound controlled
by the $L^1$-norm of the boundary trace, including a precise
characterization of the constant's blow-up rate near the boundary. Second, to
address non-spherical geometries, Lemma~\ref{thm:poisson_gradient} recalls a
potential theory bound by Widman~{\cite{Widman1967}} for the Poisson kernel on
arbitrary $C^{1, 1}$ domains. Finally,
Proposition~\ref{prop:aniso_poisson_kernel} synthesizes these concepts: by
applying an affine transformation and invoking the Widman bound on the
resulting ellipsoid, it establishes the analogous interior gradient estimate
for anisotropic harmonic functions, explicitly tracking the constant's
dependence on the ellipticity ratio.

\subsection{Interior Gradient Estimate for Harmonic Functions}
{\myprop{\begin{proposition}[Interior Gradient Estimate for Harmonic Functions
via Traces]
  \label{prop:Poissongradest}Let $u \in W^{1, 1} (B_r (x_0), \RR)$ be a weakly
  harmonic function in the open ball $B_r (x_0) \subset \RR^m$. Let $u_{\restr
  \partial B_r (x_0)} \in L^1 (\partial B_r (x_0))$ denote its boundary trace.
  Then, for every shrinking factor $0 < \lambda < 1$, $u$ is smooth in the
  interior, and the following interior gradient estimate holds:
  \begin{equation}
    \sup_{x \in B_{\lambda r} (x_0)} \lvert \nabla u (x) \rvert
     \leqslant  \frac{C_{\lambda, m}}{r^m} 
    \int_{\xi \in \partial B_r (x_0)} \lvert u (\xi) \rvert \hspace{0.17em}
    \mathrm{d} \xi, \label{eq:estLinfbdry}
  \end{equation}
  where $C_{\lambda, m} > 0$ depends only on $\lambda$ and $m$.
  
  Furthermore, the constant $C_{\lambda, m}$ satisfies the explicit upper
  bound:
  \begin{equation}
    C_{\lambda, m} \leqslant \frac{2 (m + \lambda)}{\dot{\omega}_m  (1 -
    \lambda)^m},
  \end{equation}
  where $\dot{\omega}_m$ is the surface area of the unit sphere in $\RR^m$.
  Consequently, the optimal constant exhibits a blow-up of order $\mathcal{O}
  ((1 - \lambda)^{- m})$ as $\lambda \to 1^-$.
\end{proposition}}}

\begin{remark}
  Because the difference of harmonic functions is harmonic, the proposition
   implies that for any constant $a \in \RR$, we have:
  \begin{equation} \sup_{x \in B_{\lambda r} (x_0)} \lvert \nabla u (x) \rvert \leqslant
     \frac{C_{\lambda, m}}{r^m}  \int_{{\xi} \in \partial
     B_r (x_0)} |u (\xi) - a| \hspace{0.17em} {\mathrm{d}
     \xi} . \end{equation}
  This shifted estimate is the key for bounding gradients using the
  $L^1$-oscillation on the boundary, even when the boundary values are only
  understood in the sense of traces.
\end{remark}

\begin{proof}
  By translation and scaling, it suffices to prove the estimate on the unit
  ball centered at the origin ($x_0 = 0$, $r = 1$). We define the rescaled
  function $v (y) := u (x_0 + ry)$. Because $u \in W^{1, 1} (B_r (x_0))$ is
  weakly harmonic, $v \in W^{1, 1} (B_1)$ is weakly harmonic and its trace
  belongs to $L^1 (\partial B_1)$. We aim to show:
  \begin{equation}
    \sup_{|y| \leqslant \lambda} \lvert \nabla v (y) \rvert \leqslant
    C_{\lambda, m}  \int_{\zeta \in \partial B_1} \lvert v
    ({\zeta}) \rvert \hspace{0.17em} \mathrm{d} \zeta .
  \end{equation}
  {\noindent}\text{{\itshape{The Poisson integral formula via traces.}}}
  Because $v$ is weakly harmonic and its boundary values exist in $L^1
  (\partial B_1)$ in the trace sense, classical elliptic regularity theory
  guarantees that $v$ is equivalent to a smooth harmonic function in the
  interior of $B_1$, and it can be recovered  by the Poisson integral
  formula of its trace. Thus, for almost every $y \in B_1$ (and everywhere for
  the smooth representative):
  \begin{equation}
    v (y) = \int_{\zeta \in \partial B_1} P (y, {\zeta}) v
    ({\zeta}) \hspace{0.17em} \mathrm{d} \zeta, \qquad P
    (y, {\zeta}) = \frac{1 - |y|^2}{\dot{\omega}_m 
    \hspace{0.17em} |y - {\zeta} |^m} .
  \end{equation}
  Since $v \in L^1 (\partial B_1)$ and the kernel $P (y, \zeta)$ alongside all
  its spatial derivatives with respect to $y$ are smooth and uniformly bounded
  on any compact sub-domain $|y| \leqslant \lambda < 1$, we may differentiate
  under the integral sign using the Leibniz integral rule (justified by
  Lebesgue's Dominated Convergence Theorem):
  \begin{equation}
    \nabla v (y) = \int_{\zeta \in \partial B_1} \nabla_y P (y,
    {\zeta}) v ({\zeta})
    \hspace{0.17em} \mathrm{d} \zeta .
  \end{equation}
  The kernel map $(y, {\zeta}) \mapsto \nabla_y P (y,
  {\zeta})$ is continuous (hence bounded) on the compact
  set
  \begin{equation}
    K := \{(y, {\zeta}) \in \RR^m
    \times {\Stwo}^{m - 1}\hspace{0.17em} :
    \hspace{0.17em} |y| \leqslant \lambda, \hspace{0.17em} |
    {\zeta} | = 1\} .
  \end{equation}
  Let $M_{\lambda, m} = \sup_K | \nabla_y P (y, {\zeta})
  |$ (finite, depending only on $\lambda$ and $m$). Then
  \begin{equation}
    | \nabla v (y) | \leqslant M_{\lambda, m}  \int_{\zeta \in \partial B_1}
    |v ({\zeta}) | \hspace{0.17em} \mathrm{d} \zeta .
  \end{equation}
  This is the desired estimate for the unit ball with $C_{\lambda, m} =
  M_{\lambda, m}$. To recover the general estimate on $B_r (x_0)$, we reverse
  the scaling and translation. Recall that $v (y) = u (x_0 + ry)$. By the
  chain rule, the gradient transforms as:
  \begin{equation}
    \nabla v (y) = r \nabla u (x_0 + r  y) \quad \Longrightarrow \quad
    \sup_{|y| \leqslant \lambda} | \nabla v (y) | = r \sup_{x \in B_{\lambda
    r} (x_0)} | \nabla u (x) | .
  \end{equation}
  For the right-hand side, we apply the change of variables $\xi = x_0 + r
  \zeta$ to the boundary integral. As $\zeta$ parameterizes the unit sphere
  $\partial B_1 (0)$, the variable $\xi$ parameterizes the sphere $\partial
  B_r (x_0)$. The surface area element transforms as $\mathrm{d}
  \mathcal{H}^{m - 1} (\xi) = r^{m - 1} \mathrm{d} \mathcal{H}^{m - 1}
  (\zeta)$. Therefore:
  \begin{equation}
    \int_{\zeta \in \partial B_1} |v (\zeta) | \hspace{0.17em} \mathrm{d}
    \zeta = \int_{\xi \in \partial B_r (x_0)} |u (\xi) | \frac{1}{r^{m - 1}}
    \hspace{0.17em} \mathrm{d} \xi .
  \end{equation}
  Substituting these two relations back into the unit ball estimate yields:
  \begin{equation}
    r \sup_{x \in B_{\lambda r} (x_0)} | \nabla u (x) | \leqslant M_{\lambda,
    m}  \frac{1}{r^{m - 1}}  \int_{\xi \in \partial B_r (x_0)} |u (\xi) |
    \hspace{0.17em} \mathrm{d} \xi .
  \end{equation}
  Dividing both sides by $r$ produces the exact $1 / r^m$ scaling factor,
  yielding the general result.
  
  {\noindent}\text{{\itshape{The blow-up of $C_{\lambda, m}$ as $\lambda
  \rightarrow 1^-$.}}} Finally, the blow-up $C_{\lambda, m} \to + \infty$ as
  $\lambda \to 1^-$ follows from the explicit form of $P$. Precisely
  \begin{equation}
    \nabla_y P (y, {\zeta}) = \frac{1}{\dot{\omega}_m}
    \nabla_y \left( \frac{1 - |y|^2}{|y - {\zeta} |^m}
    \right) = \frac{1}{\dot{\omega}_m} \left( \frac{- 2 y \hspace{0.17em} |y -
    {\zeta} |^2 - m (1 - |y|^2) (y -
    {\zeta})}{|y - {\zeta} |^{m + 2}}
    \right) .
  \end{equation}
  Taking the norm, we get that
  \begin{equation}
    | \nabla_y P (y, {\zeta}) | \leqslant
    \frac{1}{\dot{\omega}_m}  \left( \frac{2| y|}{|y -
    {\zeta} |^m} + \frac{m (1 - |y|^2)}{|y -
    {\zeta} |^{m + 1}} \right) .
  \end{equation}
  We now restrict our attention to $y \in \overline{B_{\lambda}}$ (so $|y|
  \leqslant \lambda$) and $\zeta \in \partial B_1$ (so $| \zeta | = 1$). We
  utilize the algebraic factorization $1 - |y|^2 = (1 - |y|)  (1 + |y|)$.
  Since $|y| < 1$, we can  bound $1 + |y| < 2$. Furthermore, by the
  reverse triangle inequality, $|y - \zeta | \geqslant 1 - |y|$. Combining
  these gives the  cancellation:
  \begin{equation}
    1 - |y|^2 \leqslant 2 (1 - |y|) \leqslant 2 |y - \zeta | .
  \end{equation}
  Substituting this cancellation into the second term of our bound yields:
  \begin{equation}
    | \nabla_y P (y, \zeta) | \leqslant \frac{1}{\dot{\omega}_m}  \left(
    \frac{2 \lambda}{|y - \zeta |^m} + \frac{2 m}{|y - \zeta |^m} \right) =
    \frac{2 (\lambda + m)}{\dot{\omega}_m  |y - \zeta |^m} .
  \end{equation}
  Finally, applying the lower bound $|y - \zeta | \geqslant 1 - \lambda$, we
  secure the uniform upper bound:
  \begin{equation}
    \sup_{|y| \leqslant \lambda} | \nabla_y P (y, \zeta) | \leqslant \frac{2
    (m + \lambda)}{\dot{\omega}_m  (1 - \lambda)^m} = : C_{\lambda, m} .
  \end{equation}
  This establishes the estimate for the unit ball, explicitly demonstrating
  the $\mathcal{O} ((1 - \lambda)^{- m})$ blow-up.
\end{proof}

\subsection{Interior Gradient Estimate for Anisotropic Harmonic Functions}
Proposition~\ref{prop:Poissongradest} establishes an interior gradient
estimate for classical harmonic functions on Euclidean balls, deriving its
constants directly from the closed-form expression of the Poisson kernel.
However, linearizing variable-coefficient elliptic problems naturally yields
the constant-coefficient equation $\divr (A_0 \nabla h) = 0$, where $A_0$ is a
symmetric, uniformly elliptic matrix. A standard affine transformation reduces
this anisotropic equation to the classical Laplace equation, but it
simultaneously maps Euclidean balls into ellipsoids whose geometry depends on
the spectrum of $A_0$. Because no explicit Poisson kernel formula exists for
arbitrary ellipsoids, the direct computational approach of
Proposition~\ref{prop:Poissongradest} cannot be applied. To overcome this
deficit, we invoke a general gradient estimate for Poisson kernels on bounded
$C^{1, 1}$ domains, originally established by Widman~{\cite{Widman1967}}.

We start with the following preliminary estimate.

\begin{lemma}
  \label{thm:poisson_gradient}Let $U \subset \RR^m$ be a bounded $C^{1, 1}$
  domain, and let $P (y, \zeta)$ denote the Poisson kernel for the Dirichlet
  Laplacian on U. There exists a constant $C > 0$, depending only on $m$ and
  $\partial U$, such that
  \begin{equation}
    | \nabla_y P (y, \zeta) | \leqslant \frac{C}{|y - \zeta |^m}  \qquad
    \text{for all } y \in U, \zeta \in \partial U.
  \end{equation}
\end{lemma}

\begin{remark}
  More precisely, the constant $C$ depends on the domain $U$  through
  the dimension $m$ and the radii of its uniform interior and exterior tangent
  balls. This precise geometric dependence is crucial for the proof of
  Proposition~\ref{prop:aniso_poisson_kernel}, as it allows us to deduce a
  uniform bound across a family of ellipsoids, yielding a final constant that
  depends solely on the ellipticity parameters of $A_0$ rather than the
  particular matrix itself.
\end{remark}

\begin{proof}
  We first recall the standard pointwise estimate
  \begin{equation}
    \label{eq:poisson_bound} P (y, \zeta) \leqslant K
    \frac{\ensuremath{\operatorname{dist}} (y, \partial U)}{|y - \zeta |^m},
    \qquad y \in U, \zeta \in \partial U,
  \end{equation}
  where $K > 0$ depends only on $m$ and on the $C^{1, 1}$ geometry of
  $\partial U$. This follows from Widman's interior gradient estimate for the
  Dirichlet Green function $G$ of $- \Delta$ on $U$ (see
  Widman~{\cite[Theorem~2.3(iii)]{Widman1967}}). Namely, for all distinct
  interior points $x, y \in U$,
  \begin{equation}
    | \nabla_x G (y, x) | \leqslant K \hspace{0.17em}
    \frac{\ensuremath{\operatorname{dist}} (y, \partial U)}{|y - x|^m} .
    \label{eq:APKtemp1}
  \end{equation}
  where $K$ depends only on $m$ and $\partial U$. Since the Poisson kernel is
  given by the normal derivative of the Green function,
  \begin{equation}
    P (y, \zeta) := - \frac{\partial G (y, \zeta)}{\partial
    \ensuremath{\boldsymbol{n}}_{\zeta}} = - \lim_{x \to \zeta, x \in U}
    \nabla_x G (y, x) \cdot \ensuremath{\boldsymbol{n}} (\zeta),
    \label{eq:APKtemp2}
  \end{equation}
  where $\ensuremath{\boldsymbol{n}} (\zeta)$ is the outward unit normal at
  $\zeta$, we obtain, for fixed $y \in U$ and $\zeta \in \partial U$,
  \begin{equation} P (y, \zeta) \leqslant \limsup_{x \to \zeta} | \nabla_x G (y, x) |
     \leqslant K \hspace{0.17em} \frac{\ensuremath{\operatorname{dist}} (y,
     \partial U)}{|y - \zeta |^m} . \end{equation}
  Thus \eqref{eq:poisson_bound} holds.
  
  Now fix $\zeta \in \partial U$. It is well known that the function $y
  \mapsto P (y, \zeta)$ is positive and harmonic in $U$. Positivity follows
  from the Hopf boundary point lemma, while harmonicity follows from the fact
  that $P (\cdot, \zeta)$ arises as the local uniform limit of harmonic
  functions.
  
  Let $y \in U$ and set $\rho := \tfrac{1}{2} \ensuremath{\operatorname{dist}}
  (y, \partial U)$. Then $B_{\rho} (y) \subset U$, and for every $w \in
  B_{\rho} (y)$ we have
  \begin{equation}
    \ensuremath{\operatorname{dist}} (w, \partial U) \leqslant
    \ensuremath{\operatorname{dist}} (y, \partial U) + |w - y| \leqslant 2
    \rho + \rho = 3 \rho .
  \end{equation}
  Moreover, since $\zeta \in \partial U$,
  \begin{equation}
    |w - \zeta | \geqslant |y - \zeta | - |w - y| \geqslant |y - \zeta | -
    \rho   \geqslant |y - \zeta | - \tfrac{1}{2} |y - \zeta | =
    \tfrac{1}{2} |y - \zeta | .
  \end{equation}
  Applying \eqref{eq:poisson_bound} at the point $w$ yields
  \begin{equation}
    P (w, \zeta) \leqslant C_1  \hspace{0.17em}
    \frac{\ensuremath{\operatorname{dist}} (w, \partial U)}{|w - \zeta |^m}
    \leqslant C_1  \hspace{0.17em} \frac{3 \rho}{\left( \tfrac{1}{2} |y -
    \zeta | \right)^m} = 3 \cdot 2^m  \hspace{0.17em} C_1  \hspace{0.17em}
    \frac{\rho}{|y - \zeta |^m} .
  \end{equation}
  Since $P (\cdot, \zeta)$ is harmonic in $B_{\rho} (y)$, the standard
  interior gradient estimate gives
  \begin{equation}
    \label{eq:interior_grad} | \nabla_y P (y, \zeta) | \leqslant
    \frac{C_m}{\rho} \sup_{w \in B_{\rho} (y)} P (w, \zeta),
  \end{equation}
  where $C_m > 0$ depends only on the dimension. Combining this with the
  previous bound, we find
  \begin{equation} | \nabla_y P (y, \zeta) | \leqslant \frac{C_m}{\rho} \cdot 3 \cdot 2^m 
     \hspace{0.17em} C_1  \hspace{0.17em} \frac{\rho}{|y - \zeta |^m} =
     \frac{C}{|y - \zeta |^m}, \end{equation}
  with $C := 3 \cdot 2^m  \hspace{0.17em} C_1 C_m$. This proves the lemma.
\end{proof}

\begin{proposition}
  \label{prop:aniso_poisson_kernel}Let $A_0 \in \RR^{m \times m}$ be a
  symmetric matrix satisfying the uniform ellipticity condition $\lambda I
  \leqslant A_0 \leqslant \Lambda I$ for some $0 < \lambda \leqslant \Lambda <
  \infty$. Let $h \in W^{1, 1} (B_{\rr} (x_0))$ be a weakly $A_0$-harmonic
  function on $B_{\rr} (x_0) \subset \RR^m$; that is, $\divr (A_0 \nabla h) =
  0$ in the weak sense in $B_{\rr} (x_0)$. Then $h$ is smooth in the interior
  of $B_{\rr} (x_0)$, and for every $\theta \in (0, 1)$,
  \begin{equation}
    \label{eq:estLinfbdry_aniso} \sup_{x \in B_{\theta \rr} (x_0)} | \nabla h
    (x) | \leqslant \frac{C_{\theta, \lambda, \Lambda, m}}{r^m}  \int_{\zeta
    \in \partial B_{\rr} (x_0)} |h (\xi) | \hspace{0.17em} \mathrm{d} \xi,
  \end{equation}
  where $C_{\theta, \lambda, \Lambda, m} > 0$ depends only on $\theta$, $m$,
  $\lambda$, and $\Lambda$, and satisfies $C_{\theta, \lambda, \Lambda, m}
  =\mathcal{O} ((1 - \theta)^{- m})$ as $\theta \to 1^-$.
\end{proposition}

\begin{proof}
  By translation invariance, we may assume $x_0 = 0$. Standard elliptic
  regularity for constant-coefficient operators gives $h \in C^{\infty} (B_r
  (0))$.
  
  \text{{\itshape{Step 1: Affine change of variables.}}} Define the linear map
  \begin{equation}
    \Phi (x) := r^{- 1} A_0^{- 1 / 2} x, \qquad x \in \RR^m,
  \end{equation}
  and set $\tilde{h} (y) := h (r  A_0^{1 / 2} y) = h (\Phi^{- 1} (y))$. Since
  $|A_0^{1 / 2} y| = |x| / r$ whenever $x = rA_0^{1 / 2} y$, the map $\Phi$
  sends $B_r$ bijectively onto the ellipsoid
  \begin{equation}
    E_1 := \left\{ y \in \RR^m \st |A_0^{1 /
    2} y| < 1 \right\} = \left\{ y \in \RR^m
    \st \langle y, A_0 y \rangle < 1 \right\},
  \end{equation}
  whose semi-axes have lengths in $[1 / \sqrt{\Lambda}, \hspace{0.17em} 1 /
  \sqrt{\lambda}]$.
  
  A direct chain-rule computation gives
  \begin{equation}
    \Delta \tilde{h} (y) = r^2 \hspace{0.17em} \divr (A_0 \nabla h) (x) = 0,
    \qquad y = \Phi (x),
  \end{equation}
  so $\tilde{h}$ is classically harmonic on $E_1$. Moreover, the identity
  $\nabla h (x) = r^{- 1} A_0^{- 1 / 2} \nabla \tilde{h} (y)$ together with
  the operator-norm bound $\|A_0^{- 1 / 2} \| \leqslant 1 / \sqrt{\lambda}$
  yelds the pointwise estimate
  \begin{equation}
    \label{eq:grad_relation_aniso_kernel} | \nabla h (x) | \leqslant
    \frac{1}{r \sqrt{\lambda}} \hspace{0.17em} | \nabla \tilde{h} (y) |,
    \qquad y = \Phi (x) .
  \end{equation}
  \text{{\itshape{Step 2: Distance from the inner ellipsoid to $\partial
  E_1$.}}} The image of $B_{\theta r}$ under $\Phi$ is the inner ellipsoid
  \begin{equation}
    E_{\theta} := \left\{ y \in \RR^m \st |A_0^{1 / 2} y| < \theta \right\} .
  \end{equation}
  For every $y \in E_{\theta}$ and every $z \in \partial E_1$, the bound
  $\|A_0^{1 / 2} \| \leqslant \sqrt{\Lambda}$ and the triangle inequality
  yield
  \begin{equation}
    1 - \theta \leqslant |A_0^{1 / 2} z| - |A_0^{1 / 2} y| \leqslant |A_0^{1 /
    2} (z - y) | \leqslant \sqrt{\Lambda}  \hspace{0.17em} |z - y| .
  \end{equation}
  Taking the infimum over all admissible pairs $(y, z)$ produces the uniform
  lower bound
  \begin{equation}
    \label{eq:dist_lower} d := \ensuremath{\operatorname{dist}} (E_{\theta},
    \partial E_1) \geqslant \frac{1 - \theta}{\sqrt{\Lambda}} > 0.
  \end{equation}
  \text{{\itshape{Step 3: Gradient bound for $\tilde{h}$ on $E_1$.}}} The
  function $\tilde{h}$ is harmonic on the smooth bounded domain $E_1$ and has
  an $L^1$ boundary trace: $\tilde{h}_{\restr \partial E_1} \in L^1 (\partial
  E_1)$. The classical Poisson integral representation therefore applies:
  \begin{equation} \tilde{h} (y) = \int_{\zeta \in \partial E_1} P_{E_1} (y, \zeta)
     \hspace{0.17em} \tilde{h} (\zeta) \hspace{0.17em} \mathrm{d} \zeta,
     \qquad y \in E_1, \end{equation}
  and since $P_{E_1}  (\cdot, \zeta)$ is smooth in the interior of $E_1$,
  differentiation under the integral sign gives
  \begin{equation}
    \nabla \tilde{h} (y) = \int_{\zeta \in \partial E_1} \nabla_y P_{E_1} (y,
    \zeta) \hspace{0.17em} \tilde{h} (\zeta) \mathrm{d} \zeta, \qquad y \in
    E_1 .
  \end{equation}
  Now we can apply Lemma~\ref{thm:poisson_gradient} to $E_1$, to get that
  \begin{equation}
    | \nabla_y \tilde{h} (y) | \leqslant C \int_{\zeta \in \partial E_1}
    \frac{| \tilde{h} (\zeta) |}{| y - \zeta |^m} \hspace{0.17em} \mathrm{d}
    \zeta, \qquad y \in E_1 .
  \end{equation}
  for a positive constant $C > 0$, depending only on $m$ and $\partial E_1$.
  However, a priori, $\partial E_1$ depends on the specific form of $A_0$. We
  want to show that the constant provided by the lemma can be chosen depending
  only on $m, \lambda, \Lambda$, and not on the specific matrix $A_0$. The
  boundary $\partial E_1$ is the regular level set $\{f = 1\}$ of $f (\zeta)
  := \langle \zeta, A_0 \zeta \rangle$, whose gradient and Hessian are
  \begin{equation}
    \nabla f (\zeta) = 2 A_0 \zeta, \qquad \nabla^2 f = 2 A_0 .
  \end{equation}
  On $\partial E_1$, the relation $\langle \zeta, A_0 \zeta \rangle = 1$
  combined with the operator inequality $\lambda A_0 \leqslant A_0^2 \leqslant
  \Lambda A_0$ (which follows by applying the spectral theorem to $A_0$, since
  each eigenvalue $\mu$ of $A_0$ satisfies $\lambda \mu \leqslant \mu^2
  \leqslant \Lambda \mu$) yields
  \begin{equation}
    \sqrt{\lambda} \leqslant |A_0 \zeta | \leqslant \sqrt{\Lambda} .
  \end{equation}
  The outward unit normal at $\zeta \in \partial E_1$ is
  $\ensuremath{\boldsymbol{n}} (\zeta) = A_0 \zeta / |A_0 \zeta |$, and the
  shape operator on the tangent space $T_{\zeta} \partial E_1$ takes the form
  \begin{equation}
    S = \frac{P_T \hspace{0.17em} \nabla^2 fP_T}{| \nabla f (\zeta) |} 
    \hspace{0.17em} = \frac{P_T A_0 P_T}{|A_0 \zeta |} \hspace{0.17em},
  \end{equation}
  where $P_T$ denotes orthogonal projection onto $T_{\zeta} \partial_1$.
  Therefore, given that $\lambda I \leqslant A_0 \leqslant \Lambda I$, the
  principal curvatures $\kappa_1, \ldots, \kappa_{m - 1}$ of $\partial E_1$
  are the eigenvalues of the shape operator $S$, we have
  \begin{equation}
    \frac{\lambda}{\sqrt{\Lambda}} \leqslant \kappa_i \leqslant
    \frac{\Lambda}{\sqrt{\lambda}}, \qquad i = 1, \ldots, m - 1.
  \end{equation}
  The upper bound on the curvatures yields a uniform interior tangent-ball
  condition on $E_1$ with radius at least $\sqrt{\lambda} / \Lambda$, while
  convexity of $E_1$ provides the exterior tangent-ball condition trivially.
  Combined with the diameter bound $\ensuremath{\operatorname{diam}}E_1
  \leqslant 2 / \sqrt{\lambda}$, this places $E_1$ in a family of $C^{1, 1}$
  domains whose geometric data are controlled by $\lambda$ and $\Lambda$
  alone. Consequently, Lemma~\ref{thm:poisson_gradient} yields a constant $K =
  K (m, \lambda, \Lambda)$, independent of the particular matrix $A_0$, such
  that
  \begin{equation}
    | \nabla_y P_{E_1} (y, \zeta) | \leqslant \frac{K}{|y - \zeta |^m}, \qquad
    y \in E_1, \zeta \in \partial E_1 .
  \end{equation}
  For $y \in E_{\theta}$ and $\zeta \in \partial E_1$, the lower
  bound~\eqref{eq:dist_lower} gives $|y - \zeta | \geqslant d \geqslant (1 -
  \theta) / \sqrt{\Lambda}$, whence
  \begin{equation}
    \sup_{y \in E_{\theta}, \zeta \in \partial E_1} | \nabla_y P_{E_1} (y,
    \zeta) | \leqslant \frac{K}{d^m} \leqslant \frac{K \hspace{0.17em}
    \Lambda^{m / 2}}{(1 - \theta)^m},
  \end{equation}
  and integrating along $\partial E_1$ gives
  \begin{equation}
    \label{eq:grad_h_tilde} | \nabla \tilde{h} (y) | \leqslant \frac{K
    \hspace{0.17em} \Lambda^{m / 2}}{(1 - \theta)^m}  \int_{\zeta \in \partial
    E_1} | \tilde{h} (\zeta) | \hspace{0.17em} \mathrm{d} \zeta, \qquad y \in
    E_{\theta} .
  \end{equation}
 {\itshape{Step 4: Pull-back of the boundary integral via the area
  formula.}} It remains to express the boundary integral over $\partial E_1$
  in terms of an integral over $\partial B_r$. For a linear map $L : \RR^m \to
  \RR^m$ and a hypersurface $\Sigma$ with unit normal
  $\ensuremath{\boldsymbol{\nu}}$, the change-of-variables formula which links
  the $(m - 1)$-dimensional Hausdorff measure $\mathcal{H}^{m - 1}_{L
  (\Sigma)}$ on $L (\Sigma)$ to the corresponding measure $\mathcal{H}^{m -
  1}_{\Sigma}$ on $\Sigma$ reads as
  \begin{equation} \mathrm{d} \mathcal{H}^{m - 1}_{L (\Sigma)}  (L \xi) = | \det L| \cdot
     |L^{- \top} \ensuremath{\boldsymbol{\nu}}(\xi) | \hspace{0.17em}
     \mathrm{d} \mathcal{H}^{m - 1}_{\Sigma} (\xi) . \end{equation}
  With $L = r^{- 1} A_0^{- 1 / 2}$, $\Sigma = \partial B_r (0)$,
  $\ensuremath{\boldsymbol{\nu}} (\xi) = \xi / r$, we have $\det L = r^{- m}
  \det (A_0)^{- 1 / 2}$ and $L^{- \top} = rA_0^{1 / 2}$, whence $|L^{- \top}
  \ensuremath{\boldsymbol{\nu}}(\xi) | = |A_0^{1 / 2} \xi |$ and
  \begin{equation}
    \mathrm{d} \mathcal{H}^{m - 1}_{\partial E_1} (\zeta) = r^{- m} \det
    (A_0)^{- 1 / 2}  |A_0^{1 / 2} \xi | \hspace{0.17em} \mathrm{d}
    \mathcal{H}^{m - 1}_{\partial B_r} (\xi), \qquad \zeta = \Phi (\xi) .
  \end{equation}
  Moreover, on $\partial B_r$, $|A_0^{1 / 2} \xi | \leqslant \sqrt{\Lambda}
  \hspace{0.17em} | \xi | = r \sqrt{\Lambda}$, and $\det (A_0)^{- 1 / 2}
  \leqslant \lambda^{- m / 2}$, so
  \begin{equation}
    \label{eq:measure_bound} \mathrm{d} \mathcal{H}^{m - 1}_{\partial E_1}
    (\zeta) \leqslant \frac{\sqrt{\Lambda}}{r^{m - 1}  \hspace{0.17em}
    \lambda^{m / 2}} \mathrm{d} \mathcal{H}^{m - 1}_{\partial B_r} (\xi) .
  \end{equation}
  Combining \eqref{eq:grad_relation_aniso_kernel}, \eqref{eq:grad_h_tilde},
  and \eqref{eq:measure_bound}, and using that $\tilde{h} (\zeta) = h (\xi)$
  for $\zeta = \Phi (\xi)$, we obtain for every $x \in B_{\theta r}$,
  \begin{eqnarray}
    | \nabla h (x) | & \leqslant & \frac{1}{r \sqrt{\lambda}} \cdot \frac{K
    \hspace{0.17em} \Lambda^{m / 2}}{(1 - \theta)^m} \cdot
    \frac{\sqrt{\Lambda}}{r^{m - 1}  \hspace{0.17em} \lambda^{m / 2}} 
    \int_{\xi \in \partial B_r} |h (\xi) | \hspace{0.17em} \mathrm{d} \xi
    \nonumber\\
    & \leqslant & \frac{K}{(1 - \theta)^m  \hspace{0.17em} r^m} \left(
    \frac{\Lambda}{\lambda} \right)^{(m + 1) / 2} \int_{\xi \in \partial B_r}
    |h (\xi) | \hspace{0.17em} \mathrm{d} \xi . \nonumber
  \end{eqnarray}
  Setting
  \begin{equation} C_{\theta, \lambda, \Lambda, m} := K (m,
     \lambda, \Lambda) \hspace{0.17em} \left( \frac{\Lambda}{\lambda}
     \right)^{(m + 1) / 2}  (1 - \theta)^{- m} \end{equation}
  establishes~\eqref{eq:estLinfbdry_aniso} and displays the explicit
  $\mathcal{O} ((1 - \theta)^{- m})$ blow-up of the constant as $\theta \to
  1^-$.
\end{proof}

\section*{Declarations}

\noindent
\textit{Conflict of Interest:} The author declares that he has no conflicts of interest.
\smallskip

\noindent
\textit{Data Availability:} No datasets were generated or analyzed during the current study.

\smallskip

\noindent
\textit{Declaration of Generative AI and AI-Assisted Technologies:}
During the preparation of this manuscript, the author used DeepL and Gemini for
proofreading, language refinement, and improving the clarity and readability of
the text. No artificial intelligence tools were used in the mathematical
reasoning, derivation of results, or proofs presented in this manuscript.
The manuscript was initially prepared using TeXmacs and subsequently exported
to LaTeX. Gemini was used to identify and correct translation or conversion
errors arising during this process, to assist with LaTeX formatting, and to
generate TikZ code for a diagram that was not converted during the export.
The author is fully responsible for the accuracy, validity, and integrity of
all mathematical results presented in this work.

\section*{Acknowledgements}
The author is a member of GNAMPA--INdAM. He gratefully acknowledges partial financial support from the GNAMPA Project CUP\_E53C25002010001, from the University of Naples Federico II through the FRA Project-B ``VarMoCry'' (\emph{Variational Analysis and Modeling of Liquid Crystals}), and from the Italian Ministry of University and Research (MUR) through the PRIN 2022 project \emph{Variational Analysis of Complex Systems in Material Science, Physics and Biology} (No.~2022HKBF5C).

The author wishes to thank Tristan Rivière for insightful discussions regarding the historical context of these problems and for bringing the reference \cite{Bethuel_1993} to his attention.

\bibliographystyle{siam}     
\bibliography{AHM2.bib}

@Book{Giusti_2003,
  author    = {Giusti, Enrico},
  publisher = {World Scientific Publishing Company},
  title     = {Direct Methods in the Calculus of Variations},
  year      = {2003},
  address   = {River Edge, NJ},
  isbn      = {9789812795557},
  month     = jan,
  doi       = {10.1142/5002},
}

@Book{Giaquinta1983,
  author    = {Giaquinta, Mariano},
  publisher = {Princeton University Press},
  title     = {Multiple Integrals in the Calculus of Variations and Nonlinear Elliptic Systems},
  year      = {1983},
  address   = {Princeton, NJ},
  series    = {Annals of Mathematics Studies},
  volume    = {105},
}

@Article{Chang_1999,
  author    = {Chang, Sun-Yung A. and Wang, Lihe and Yang, Paul C.},
  journal   = {Communications on Pure and Applied Mathematics},
  title     = {Regularity of harmonic maps},
  year      = {1999},
  issn      = {1097-0312},
  month     = sep,
  number    = {9},
  pages     = {1099--1111},
  volume    = {52},
  doi       = {10.1002/(sici)1097-0312(199909)52:9<1099::aid-cpa3>3.0.co;2-o},
  publisher = {Wiley},
}

@Article{Bethuel_1993,
  author    = {Bethuel, Fabrice},
  journal   = {Manuscripta Mathematica},
  title     = {On the singular set of stationary harmonic maps},
  year      = {1993},
  issn      = {1432-1785},
  month     = dec,
  number    = {1},
  pages     = {417--443},
  volume    = {78},
  doi       = {10.1007/bf02599324},
  publisher = {Springer Science and Business Media LLC},
}

@Article{Coifman1993,
  author    = {Coifman, Ronald and Lions, Pierre-Louis and Meyer, Yves and Semmes, Stephen},
  journal   = {Journal de Math{\'e}matiques Pures et Appliqu{\'e}es. Neuvi{\`e}me S{\'e}rie},
  title     = {Compensated compactness and {H}ardy spaces},
  year      = {1993},
  number    = {3},
  pages     = {247--286},
  volume    = {72},
  publisher = {Elsevier},
}

@Article{Evans_1991,
  author    = {Evans, Lawrence C.},
  journal   = {Archive for Rational Mechanics and Analysis},
  title     = {Partial regularity for stationary harmonic maps into spheres},
  year      = {1991},
  issn      = {1432-0673},
  number    = {2},
  pages     = {101--113},
  volume    = {116},
  doi       = {10.1007/bf00375587},
  publisher = {Springer Science and Business Media LLC},
}

@Article{Helein1990,
  author    = {H{\'e}lein, Fr{\'e}d{\'e}ric},
  journal   = {Comptes Rendus de l'Acad{\'e}mie des Sciences. S{\'e}rie I, Math{\'e}matique},
  title     = {R{\'e}gularit{\'e} des applications faiblement harmoniques entre une surface et une sph{\`e}re},
  year      = {1990},
  number    = {9},
  pages     = {519--524},
  volume    = {311},
  publisher = {Gauthier-Villars},
}

@Article{Helein1991,
  author    = {H{\'e}lein, Fr{\'e}d{\'e}ric},
  journal   = {Comptes Rendus de l'Acad{\'e}mie des Sciences. S{\'e}rie I, Math{\'e}matique},
  title     = {R{\'e}gularit{\'e} des applications faiblement harmoniques entre une surface et une vari{\'e}t{\'e} riemannienne},
  year      = {1991},
  number    = {8},
  pages     = {591--596},
  volume    = {312},
  publisher = {Gauthier-Villars},
}

@Book{H_lein_2002,
  author    = {Hélein, Frédéric},
  publisher = {Cambridge University Press},
  title     = {Harmonic Maps, Conservation Laws and Moving Frames},
  year      = {2002},
  address   = {Paris},
  isbn      = {9780511543036},
  month     = jun,
  doi       = {10.1017/cbo9780511543036},
}

@Article{Hildebrandt_1977,
  author    = {Hildebrandt, Stefan and Kaul, Helmut and Widman, Kjell-Ove},
  journal   = {Acta Mathematica},
  title     = {An existence theorem for harmonic mappings of Riemannian manifolds},
  year      = {1977},
  issn      = {0001-5962},
  number    = {0},
  pages     = {1--16},
  volume    = {138},
  doi       = {10.1007/bf02392311},
  publisher = {International Press of Boston},
}

@Article{Carbou_1997,
  author    = {Carbou, Gilles},
  journal   = {Calculus of Variations and Partial Differential Equations},
  title     = {Regularity for critical points of a non local energy},
  year      = {1997},
  issn      = {1432-0835},
  month     = jul,
  number    = {5},
  pages     = {409--433},
  volume    = {5},
  doi       = {10.1007/s005260050073},
  publisher = {Springer Science and Business Media LLC},
}

@Book{Giaquinta_2012,
  author    = {Giaquinta, Mariano and Martinazzi, Luca},
  publisher = {Scuola Normale Superiore},
  title     = {An Introduction to the Regularity Theory for Elliptic Systems, Harmonic Maps and Minimal Graphs},
  year      = {2012},
  isbn      = {9788876424434},
  doi       = {10.1007/978-88-7642-443-4},
}

@Book{Han2011,
  author    = {Han, Qing and Lin, Fanghua},
  publisher = {American Mathematical Society},
  title     = {Elliptic Partial Differential Equations},
  year      = {2011},
  address   = {Providence, RI},
  series    = {Courant Lecture Notes in Mathematics},
  volume    = {1},
}

@Article{Rivi_re_1995,
  author    = {Rivière, Tristan},
  journal   = {Acta Mathematica},
  title     = {Everywhere discontinuous harmonic maps into spheres},
  year      = {1995},
  issn      = {0001-5962},
  number    = {2},
  pages     = {197--226},
  volume    = {175},
  doi       = {10.1007/bf02393305},
  publisher = {International Press of Boston},
}

@Book{Moser_2005,
  author    = {Moser, Roger},
  publisher = {World Scientific},
  title     = {Partial Regularity for Harmonic Maps and Related Problems},
  year      = {2005},
  isbn      = {9789812701312},
  month     = feb,
  doi       = {10.1142/5691},
}

@Article{Di_Fratta_2020a,
  author    = {Di Fratta, Giovanni and Muratov, Cyrill B. and Rybakov, Filipp N. and Slastikov, Valeriy V.},
  journal   = {SIAM Journal on Mathematical Analysis},
  title     = {Variational Principles of Micromagnetics Revisited},
  year      = {2020},
  issn      = {1095-7154},
  month     = jan,
  number    = {4},
  pages     = {3580--3599},
  volume    = {52},
  doi       = {10.1137/19m1261365},
  publisher = {Society for Industrial \& Applied Mathematics (SIAM)},
}

@Article{Di_Fratta_2024,
  author    = {Di Fratta, Giovanni and Muratov, Cyrill B. and Slastikov, Valeriy V.},
  journal   = {Mathematical Models and Methods in Applied Sciences},
  title     = {Reduced energies for thin ferromagnetic films with perpendicular anisotropy},
  year      = {2024},
  issn      = {1793-6314},
  month     = jul,
  number    = {10},
  pages     = {1861--1904},
  volume    = {34},
  doi       = {10.1142/s0218202524500386},
  publisher = {World Scientific Pub Co Pte Ltd},
}

@Article{Di_Fratta_2023,
  author    = {Di Fratta, Giovanni and Slastikov, Valeriy},
  journal   = {Nonlinear Analysis},
  title     = {Curved thin-film limits of chiral {D}irichlet energies},
  year      = {2023},
  issn      = {0362-546X},
  month     = sep,
  pages     = {113303},
  volume    = {234},
  doi       = {10.1016/j.na.2023.113303},
  publisher = {Elsevier BV},
}

@Article{Alouges_2015,
  author    = {Alouges, François and Di Fratta, Giovanni},
  journal   = {Proceedings of the Royal Society A: Mathematical, Physical and Engineering Sciences},
  title     = {Homogenization of composite ferromagnetic materials},
  year      = {2015},
  issn      = {1471-2946},
  month     = oct,
  number    = {2182},
  pages     = {20150365},
  volume    = {471},
  doi       = {10.1098/rspa.2015.0365},
  publisher = {The Royal Society},
}

@Article{Di_Fratta_2020c,
  author    = {Di Fratta, Giovanni},
  journal   = {Zeitschrift für angewandte Mathematik und Physik},
  title     = {Micromagnetics of curved thin films},
  year      = {2020},
  issn      = {1420-9039},
  month     = jun,
  number    = {4},
  volume    = {71},
  doi       = {10.1007/s00033-020-01336-2},
  publisher = {Springer Science and Business Media LLC},
}

@Book{Brown1963,
  author    = {Brown, William Fuller},
  publisher = {Interscience Publishers},
  title     = {Micromagnetics},
  year      = {1963},
  address   = {New York},
  series    = {Interscience Tracts on Physics and Astronomy},
}

@Book{Bertotti1998,
  author    = {Bertotti, Giorgio},
  publisher = {Elsevier},
  title     = {Hysteresis in Magnetism: For Physicists, Materials Scientists, and Engineers},
  year      = {1998},
  address   = {San Diego},
  isbn      = {9780120932702},
  doi       = {10.1016/b978-0-12-093270-2.x5048-x},
}

@Article{Fert_2013,
  author    = {Fert, Albert and Cros, Vincent and Sampaio, João},
  journal   = {Nature Nanotechnology},
  title     = {Skyrmions on the track},
  year      = {2013},
  issn      = {1748-3395},
  month     = mar,
  number    = {3},
  pages     = {152--156},
  volume    = {8},
  doi       = {10.1038/nnano.2013.29},
  publisher = {Springer Science and Business Media LLC},
}

@Book{Morrey1966,
  author    = {Morrey, Charles B.},
  publisher = {Springer-Verlag},
  title     = {Multiple Integrals in the Calculus of Variations},
  year      = {1966},
  address   = {Berlin},
  note      = {Classic reference for 2D regularity; sometimes cited under 1968 reprint/edition},
}

@InCollection{Schoen1984,
  author    = {Schoen, Richard},
  booktitle = {Seminar on Nonlinear Partial Differential Equations},
  publisher = {Springer-Verlag},
  title     = {Analytic aspects of the harmonic map problem},
  year      = {1984},
  address   = {New York},
  pages     = {321--358},
  series    = {Mathematical Sciences Research Institute Publications},
  volume    = {2},
}

@Article{Helein1991a,
  author    = {Hélein, Frédéric},
  journal   = {Manuscripta Mathematica},
  title     = {Regularity of weakly harmonic maps from a surface into a manifold with symmetries},
  year      = {1991},
  issn      = {1432-1785},
  month     = dec,
  number    = {1},
  pages     = {203--218},
  volume    = {70},
  doi       = {10.1007/bf02568371},
  publisher = {Springer Science and Business Media LLC},
}

@Article{Schoen1982,
  author  = {Schoen, Richard and Uhlenbeck, Karen},
  journal = {Journal of Differential Geometry},
  title   = {A regularity theory for harmonic maps},
  year    = {1982},
  pages   = {307--335},
  volume  = {17},
}

@Article{Changbi_1999,
  author    = {Chang, Sun-Yung A. and Wang, Lihe and Yang, Paul C.},
  journal   = {Communications on Pure and Applied Mathematics},
  title     = {A regularity theory of biharmonic maps},
  year      = {1999},
  issn      = {1097-0312},
  month     = sep,
  number    = {9},
  pages     = {1113--1137},
  volume    = {52},
  doi       = {10.1002/(sici)1097-0312(199909)52:9<1113::aid-cpa4>3.0.co;2-7},
  publisher = {Wiley},
}

@article{Topping2004,
  author       = {Topping, Peter},
  title        = {Repulsion and quantization in almost-harmonic maps, and asymptotics of the harmonic map flow},
  journal      = {Annals of Mathematics},
  volume       = {159},
  number       = {2},
  year         = {2004},
  pages        = {465--534},
  doi          = {10.4007/annals.2004.159.465}
}

@Article{Wang2017,
  author  = {Wang, Wendong and Wei, Dongyi and Zhang, Zhifei},
  journal = {Journal of Functional Analysis},
  title   = {Energy identity for approximate harmonic maps from surfaces to general targets},
  year    = {2017},
  pages   = {776--803},
  volume  = {272},
}

@Article{Moser2005,
  author  = {Moser, Roger},
  journal = {Mathematische Zeitschrift},
  title   = {Energy concentration for almost harmonic maps and the {W}illmore functional},
  year    = {2005},
  pages   = {293--311},
  volume  = {251},
  doi     = {10.1007/s00209-005-0803-z},
}

@Article{Riviere2007,
  author    = {Rivière, Tristan},
  journal   = {Inventiones mathematicae},
  title     = {Conservation laws for conformally invariant variational problems},
  year      = {2006},
  issn      = {1432-1297},
  month     = dec,
  number    = {1},
  pages     = {1--22},
  volume    = {168},
  doi       = {10.1007/s00222-006-0023-0},
  publisher = {Springer Science and Business Media LLC},
}

@Article{Mueller2009,
  author    = {Müller, Frank and Schikorra, Armin},
  journal   = {Analysis},
  title     = {Boundary regularity via {U}hlenbeck–{R}ivière decomposition},
  year      = {2009},
  issn      = {0174-4747},
  month     = jan,
  number    = {2},
  pages     = {199--220},
  volume    = {29},
  doi       = {10.1524/anly.2009.1025},
  publisher = {Walter de Gruyter GmbH},
}

@Article{Sormani_2018,
  author    = {Sormani, Christina and Bray, Hubert L. and Minicozzi, William P. and Eichmair, Michael and Huang, Lan-Hsuan and Yau, Shing-Tung and Uhlenbeck, Karen and Kusner, Rob and Codá Marques, Fernando and Mese, Chikako and Fraser, Ailana},
  journal   = {Notices of the American Mathematical Society},
  title     = {The {M}athematics of {R}ichard {S}choen},
  year      = {2018},
  issn      = {1088-9477},
  month     = dec,
  number    = {11},
  pages     = {1},
  volume    = {65},
  doi       = {10.1090/noti1749},
  publisher = {American Mathematical Society (AMS)},
}

@Article{Hardt1997,
  author     = {Hardt, Robert M.},
  journal    = {Bull. Amer. Math. Soc. (N.S.)},
  title      = {Singularities of harmonic maps},
  year       = {1997},
  issn       = {0273-0979,1088-9485},
  number     = {1},
  pages      = {15--34},
  volume     = {34},
  doi        = {10.1090/S0273-0979-97-00692-7},
  fjournal   = {American Mathematical Society. Bulletin. New Series},
  mrclass    = {58E20},
  mrnumber   = {1397098},
  mrreviewer = {Martin\ Fuchs},
  url        = {https://doi.org/10.1090/S0273-0979-97-00692-7},
}

@Article{Widman1967,
  author  = {Widman, Kjell-Ove},
  journal = {Mathematica Scandinavica},
  title   = {Inequalities for the {G}reen Function and Boundary Continuity of the Gradient of Solutions of Elliptic Differential Equations},
  year    = {1967},
  pages   = {17--37},
  volume  = {21},
}

@Article{Bethuel1992,
  author     = {Bethuel, Fabrice},
  journal    = {C. R. Acad. Sci. Paris S\'er. I Math.},
  title      = {Un r\'esultat de r\'egularit\'e{} pour les solutions de l'\'equation de surfaces \`a{} courbure moyenne prescrite},
  year       = {1992},
  issn       = {0764-4442},
  number     = {13},
  pages      = {1003--1007},
  volume     = {314},
  fjournal   = {Comptes Rendus de l'Acad\'emie des Sciences. S\'erie I. Math\'ematique},
  mrclass    = {53A10 (35B65 35J60)},
  mrnumber   = {1168525},
  mrreviewer = {Dennis\ M.\ DeTurck},
}

\end{document}